\documentclass[11pt]{amsart}

\usepackage{latexsym}
\usepackage{amsmath}
\usepackage{amsthm}
\usepackage{amssymb}
\usepackage{vmargin}
\usepackage{amscd}
\usepackage{stmaryrd}
\usepackage{euscript}
\usepackage{mathrsfs}
\usepackage{amscd}
\usepackage[all]{xy}
\usepackage{xr}

\DeclareMathAlphabet{\mathpzc}{OT1}{pzc}{m}{it}

\newtheorem{theorem}{Theorem}[section]
\newtheorem{proposition}[theorem]{Proposition}
\newtheorem{corollary}[theorem]{Corollary}

\newtheorem{lemma}[theorem]{Lemma}
\newtheorem*{theorem*}{Theorem}
\newtheorem*{proposition*}{Proposition}
\newtheorem*{corollary*}{Corollary}
\newtheorem*{lemma*}{Lemma}
\newtheorem*{conjecture*}{Conjecture}

\theoremstyle{definition}
\newtheorem{definition}[theorem]{Definition}

\newtheorem*{definition*}{Definition}

\theoremstyle{remark}
\newtheorem{example}[theorem]{Example}
\newtheorem{examples}[theorem]{Examples}
\newtheorem{remark}[theorem]{Remark}
\newtheorem{remarks}[theorem]{Remarks}

\newtheorem*{example*}{Example}
\newtheorem*{examples*}{Examples}
\newtheorem*{remark*}{Remark}
\newtheorem*{remarks*}{Remarks}
\newtheorem*{exercise*}{Exercise}

\newcommand\da{\!\downarrow\!}
\newcommand\ra{\rightarrow}
\newcommand\la{\leftarrow}
\newcommand\lra{\longrightarrow}
\newcommand\lla{\longleftarrow}
\newcommand\id{\mathrm{id}}

\newcommand\ten{\otimes}
\newcommand\vareps{\varepsilon}
\newcommand\eps{\epsilon}

\newcommand\CC{\mathrm{C}}

\renewcommand\H{\mathrm{H}}
\newcommand\z{\mathrm{Z}}

\newcommand\N{\mathbb{N}}
\newcommand\Z{\mathbb{Z}}

\newcommand\bH{\mathbb{H}}
\newcommand\bI{\mathbb{I}}

\newcommand\bS{\mathbb{S}}

\newcommand\C{\mathcal{C}}

\newcommand\cA{\mathcal{A}}
\newcommand\cB{\mathcal{B}}

\newcommand\cD{\mathcal{D}}
\newcommand\cE{\mathcal{E}}
\newcommand\cF{\mathcal{F}}

\newcommand\cN{\mathcal{N}}

\newcommand\cO{\mathcal{O}}

\newcommand\cU{\mathcal{U}}
\newcommand\cV{\mathcal{V}}

\newcommand\D{\mathcal{D}}

\newcommand\cDel{\mathcal{DEL}}

\newcommand\n{\mathfrak{n}}

\newcommand\g{\mathfrak{g}}

\newcommand\cHom{\mathcal{H}\!\mathit{om}}

\newcommand\cHHom{\underline{\mathcal{H}\!\mathit{om}}}

\newcommand\Ho{\mathrm{Ho}}
\newcommand\Ring{\mathrm{Ring}}

\newcommand\Hom{\mathrm{Hom}}
\newcommand\Map{\mathrm{Map}}
\newcommand\map{\mathrm{map}}
\newcommand\HHom{\underline{\mathrm{Hom}}}

\newcommand\End{\mathrm{End}}

\newcommand\Der{\mathrm{Der}}

\newcommand\coker{\mathrm{coker\,}}

\newcommand\Ob{\mathrm{Ob}\,}
\newcommand\Ch{\mathrm{Ch}}

\newcommand\Ab{\mathrm{Ab}}
\newcommand\Top{\mathrm{Top}}
\newcommand\Gp{\mathrm{Gp}}

\newcommand\Shf{\mathrm{Shf}}

\newcommand\Set{\mathrm{Set}}

\newcommand\Cat{\mathrm{Cat}}
\newcommand\CSS{\mathcal{CSS}}

\newcommand\Dat{\mathrm{Dat}}

\newcommand\Sing{\mathrm{Sing}}

\newcommand\ad{\mathrm{ad}}

\newcommand\Lim{\varprojlim}
\newcommand\LLim{\varinjlim}
\DeclareMathOperator*{\holim}{holim}
\newcommand\ho{\mathrm{ho}\!}
\newcommand\into{\hookrightarrow}
\newcommand\onto{\twoheadrightarrow}
\newcommand\abuts{\implies}
\newcommand\xra{\xrightarrow}
\newcommand\xla{\xleftarrow}
\newcommand\pr{\mathrm{pr}}

\newcommand\alg{\mathrm{alg}}

\newcommand\bt{\bullet}
\newcommand\by{\times}

\newcommand\mc{\mathrm{MC}}
\newcommand\mmc{\underline{\mathrm{MC}}}

\newcommand\cMC{\mathcal{MC}}

\newcommand\Gg{\mathrm{Gg}}

\newcommand\ddef{\mathrm{Def}}
\newcommand\ddel{\mathrm{Del}}

\newcommand\Vect{\mathrm{Vect}}
\newcommand\Rep{\mathrm{Rep}}

\newcommand\Tot{\mathrm{Tot}\,}
\newcommand\diag{\mathrm{diag}\,}
\newcommand\toph{\top_{\mathrm{h}}}

\newcommand\hor{\mathrm{hor}}
\newcommand\ver{\mathrm{ver}}

\newcommand\pd{\partial}

\newcommand\half{\frac{1}{2}}

\newcommand\Ar{\mathrm{Ar}}

\newcommand\gr{\mathrm{gr}}

\newcommand\ab{\mathrm{ab}}

\newcommand\gpd{\mathrm{Gpd}}
\newcommand\Gpd{\mathrm{Gpd}}
\newcommand\Dpd{\mathrm{Dpd}}

\renewcommand\alg{\mathrm{alg}}

\newcommand\Fr{\mathrm{Fr}}

\newcommand\cosk{\mathrm{cosk}}

\newcommand\op{\mathrm{opp}}

\newcommand\co{\colon\thinspace}

\newcommand\oR{\mathbf{R}}

\newcommand\oL{\mathbf{L}}

\newcommand\on{\mathbf{n}}
\newcommand\om{\mathbf{m}}
\newcommand\ok{\mathbf{k}}

\newcommand\oO{\mathbf{0}}
\newcommand\oI{\mathbf{1}}

\newcommand\uleft\underleftarrow
\newcommand\uline\underline

\renewcommand\alg{\mathrm{alg}\,}

\begin{document}

\begin{abstract}
Lada introduced strong homotopy algebras to describe the structures on a deformation retract of an algebra in topological spaces. However, there is no  satisfactory general definition of a morphism of   strong homotopy (s.h.) algebras. 
Given a monad $\top$ on a simplicial category $\C$, we instead show how s.h. $\top$-algebras over $\C$ naturally form a 
Segal space. Given a distributive monad-comonad pair $(\top, \bot)$, the same is true for s.h. $(\top, \bot)$-bialgebras over $\C$; in particular this yields the homotopy theory of s.h. sheaves of s.h. rings.  There are similar statements for quasi-monads and quasi-comonads. We also show how the structures arising are related to derived connections on bundles.
\end{abstract}

\title{The homotopy theory of strong homotopy algebras and bialgebras}
\author{J.P.Pridham}
\thanks{This work was supported by Trinity College, Cambridge; and by the Engineering and Physical Sciences Research Council [grant number  EP/F043570/1].}
\maketitle

\section*{Introduction}

Given a monad $\top$ acting on the category of topological spaces, Lada introduced (in \cite{loop}) the notion of a strong homotopy (s.h.) $\top$-algebra. This characterises the structures arising on   deformation retracts of  $\top$-algebras. Indeed, when $\top$ is an operad, there is a bar construction which realises every s.h. $\top$-algebra as such a deformation retract.  The formulation of s.h. algebras  does not use any special properties of topological spaces, so adapts to any simplicial category, and likewise adapts to describe s.h. coalgebras of a comonad $\bot$.

Structures such as Hopf algebras or sheaves of rings cannot be described as algebras of a monad $\top$ or as coalgebras of a comonad $\bot$. However, in both cases, there are both  a  natural monad $\top$ (governing the algebraic structure) and a comonad $\bot$ (governing the coalgebraic structure), satisfying a distributivity condition. This seems to have first been described by  
Van Osdol in \cite{osdol} in order to develop bicohomology theory. Independent rediscoveries have appeared in \cite{bissonjoyal} Appendix C, \cite{powerwatanabe} and \cite{paper2} \S 2. This permits the characterisation of a compatibility condition for the algebraic and coalgebraic structures.

In \cite{ddt2}, it was  observed that the equations defining a s.h. $\top$-algebra can also be used to define s.h. $(\top, \bot)$-bialgebras associated to a distributive monad-comonad pair $(\top, \bot)$. In particular, this gives rise to a notion of s.h. sheaves of s.h. algebras (on any site with enough points), yielding important applications in algebro-geometric deformation theory. 

The first main result in this paper is Proposition \ref{enrichtop}, which provides a single unified framework for dealing with algebras, coalgebras and bialgebras. This then combines with Proposition \ref{MCsegal} and Corollaries \ref{NCSS} and \ref{DefCSS} to give three possible models for the $\infty$-category of s.h. algebras, coalgebras and bialgebras. These models are shown to be equivalent in Propositions \ref{MCNequiv} and \ref{ddefmap}. Finally, Theorem \ref{cfexp} shows how this $\infty$-category is related to the Maurer-Cartan functor featuring in  \cite{Man2} and \cite{hinstack}.

These results have applications in  derived deformation theory, which  starts with a moduli functor from algebras to groupoids, and seeks a derived moduli functor from simplicial algebras to $\infty$-groupoids. By describing deformation problems in terms of bialgebraic structures, \cite{ddt2} and \cite{higher} apply the results of this paper to construct derived deformation functors, and these are being extended to (global) derived moduli functors in \cite{dmc}. Even where there are other possible approaches to defining derived moduli functors, strong homotopy bialgebras often provide a more concrete description, and have the crucial property that the functor is left exact, making it easier to verify Lurie's representability theorem (\cite{lurie}).   For   some examples, such as  Hopf algebras, strong homotopy bialgebras provide the only known means of constructing derived deformations.

A major failing of the theory of s.h. algebras  is that there is no  satisfactory general definition of  morphisms. In \cite{loop}, this difficulty was obviated by considering morphisms of the associated bar constructions, but this has several disadvantages. For applications in deformation theory, the main problem is that the bar construction does not respect finite limits, in general. For bialgebras, the difficulty is even more fundamental, since the bar and cobar constructions will in general be incompatible, so the s.h. structures cannot be rectified. This is essentially the phenomenon that a  lax sheaf of lax simplicial rings cannot be replaced by a strict sheaf of strict simplicial rings.

Rather than seeking to define  simplicial categories of s.h. algebras or bialgebras, we instead construct Segal spaces. These are a special type of simplicial space (i.e. bisimplicial set) introduced by Rezk in \cite{rezk} as a model for  homotopy theories, and Bergner showed in \cite{bergner3} that the associated model category is Quillen-equivalent to the model category of simplicial categories, so any Segal space naturally gives rise to a simplicial category.

Our approach makes use of a generalisation of the theory of homotopy monoids expounded by Leinster in \cite{leinster}. We introduce a slight generalisation, the quasi-comonoid, of a homotopy comonoid, and  associate a simplicial set $\mmc(E)$, the Maurer-Cartan space, to any simplicial quasi-comonoid $E$. When $E^0$ is a group rather than a monoid, a more natural object is the Deligne space $\ddel(E)$, which is the  homotopy quotient of $\mmc(E)$ by $E^0$. We also develop the concept of a quasi-descent datum; put simply, a quasi-descent datum is to a quasi-comonoid as a category is to a monoid. 

Proposition \ref{enrichtopbot} then shows how a distributive  monad-comonad pair  on a simplicial category  $\C$ naturally enriches it to a simplicial quasi-descent datum, with $\mmc$ evaluated at any object $x$ of $\C$ giving the space of s.h. bialgebras over $x$.  Moreover, there is  a natural quasi-comonoid associated to any diagram in $\C$, and we use this to define a simplicial space $\cMC$ by mimicking the nerve construction (Definition \ref{MCdef}).  In Proposition \ref{MCsegal} this is seen to be a Segal space. 

Another source of quasi-comonoids is from cosimplicial groups, with the construction $\cE$ given in Definition \ref{cEdef}. In this case, the Maurer-Cartan space admits a simpler description (Proposition \ref{cfmc}). If $X$ is a simplicial set, $G$ a simplicial group and $\CC^{\bt}(X,G)$ the cosimplicial simplicial group of $G$-cochains on $X$, then $\ddel(\cE\CC^{\bt}(X,G))$  is equivalent to $\HHom(X, \bar{W}G)$, where $\bar{W}G$ is the classifying space of $G$ (Proposition \ref{pathdef} and Remark \ref{pathdef}). If the cosimplicial group comes from denormalising and exponentiating a differential graded Lie algebra, then the Maurer-Cartan space is equivalent to  the  classical Maurer-Cartan space of derived connections, given by the equation
$$
d\omega + \half[\omega,\omega]=0,
$$
for $\omega$ of total degree $1$ (Corollary \ref{expmap2}).

The structure of the paper is as follows.

In Section \ref{leinrev}, we recall Leinster's  comonoids up to homotopy, and introduce  quasi-comonoids and then the category $Q\Dat$ of quasi-descent data.  There is an adjunction (Lemma \ref{algadj}):
$$
\xymatrix@1{ \Cat \ar@<1ex>[r]^{\alg^*}_{\bot} & Q\Dat \ar@<1ex>[l]^{\alg}}.
$$
In Section \ref{algsn}, we introduce monads and comonads, and show how a  monad $\top$ on a category  $\C$ naturally enriches it to a quasi-descent datum $\cD(\C, \top)$.  This has the properties that $\alg\cD(\C, \top)$ is the category $\C^{\top}$ of $\top$-algebras, and that $\alg^*\C= \cD(\C, \id)$.
 There is a similar result for $(\top, \bot)$-bialgebras. 

Section \ref{simplicialsn} begins the work of extending these constructions to simplicial categories. Distributive monad-comonad pairs  on simplicial categories give rise to simplicial quasi-descent data,  but the functor  $\alg$ above destroys  higher order information. We therefore begin (Definition \ref{mcdefqm}) by defining the higher Maurer-Cartan functor $\mmc$ from simplicial quasi-comonoids $sQM^*$ to simplicial sets $\bS$. Given an object $x$ of a simplicial quasi-descent datum $\cD$, there is a simplicial quasi-comonoid $\cD(x,x)$, and the vertices of $\mmc(\cD(x,x))$  correspond closely to  Lada's set of strong homotopy $\top$-algebras over $x$, when $\cD=\cD(\C,\top)$ (Remark \ref{ladark}). We then extend this to  a simplicial space (i.e. bisimplicial set) $\cMC(\cD)$ (Definition \ref{MCdef}), which  mimics the nerve construction by developing  constructions loosely corresponding to strong homotopy diagrams of strong homotopy algebras. When the underlying category is a groupoid, $\cDel(\cD)$ is the homotopy quotient of $\cMC(\cD)$ by morphisms in $\cD$, and $\ddel$ is similarly related to $\mmc$.

Section \ref{absn} is primarily concerned with the study of quasi-comonoids in $(\Ab,\by)$  (abelian groups with the monoidal structure $\by$). These are equivalent to cosimplicial abelian groups (Lemma \ref{abqm}), and we exploit this to develop  cohomology of quasi-comonoids, and relate it to homotopy groups of $\mmc$ and $\ddel$. We also study nerves $B\Gamma$ of quasi-comonoids in groupoids, in order to understand  fundamental groupoids of simplicial quasi-comonoids. We then introduce linear quasi-comonoids (i.e. quasi-comonoids in $(\Ab, \ten)$). These allow us to relate cohomology of a simplicial quasi-comonoid to homology of the underlying simplicial abelian group (Proposition \ref{adams}). This gives a cohomological characterisation of when a cofibrant simply connected simplicial quasi-comonoid is contractible (Corollary \ref{trivchar}). These results are then extended to quasi-bicomonoids, which play a crucial r\^ole in the construction of $\cMC$.

Although the functor $\alg$ above is poorly suited to simplicial categories, its left adjoint $\alg^*$ extends naturally and preserves weak equivalences. There is a model structure on $sQ\Dat$, and the purpose of Section \ref{mapsn} is to relate, for $\cD$ fibrant, the simplicial spaces $\cMC(\cD)$ and $\cDel(\cD)$ with the the simplicial space $\cN(\cD)$ given by the derived function complexes
$$
\cN(\cD)_n= \Map_{sQ\Dat}^h(\alg^*\on, \cD),
$$
where $\on$ is the category associated to the poset $[0,n]$. When $\cD$ is constructed from  a distributive  monad-comonad pair  on a simplicial category $\cB$, we think of $\cMC(\cD),\cDel(\cD)$ and $\cN(\cD)$ as  candidates for the   simplicial space of s.h. bialgebras over $\cB$.

It turns out that $\cMC(\cD)$ is a Segal space, while $\cN(\cD)$ and $\cDel(\cD)$ are complete Segal spaces (Proposition \ref{MCsegal},  Corollaries \ref{NCSS} and \ref{DefCSS}). There are Dwyer-Kan equivalences between these Segal spaces (Propositions \ref{MCNequiv} and \ref{ddefmap}), so they are all weakly equivalent in the complete Segal space model structure from \cite{rezk}. Moreover,  Theorem \ref{catswork} shows how $\cMC$ and $\cN$  can both be regarded as the derived right adjoint of $\alg^*$, even though $\alg^*$ is not left Quillen. 

Section \ref{class} establishes a  simpler Maurer-Cartan space construction for cosimplicial simplicial groups, yielding the equivalence $\ddel(\cE\CC^{\bt}(X,G)) \simeq \HHom(X, \bar{W}G)$ described earlier, and  simplifies this description further when $G$ is a formal Lie group, via Theorem \ref{cfexp}.

The appendices establish a framework for applying these results more widely still. In Appendix  \ref{qmonadsn}, quasi-monads and quasi-comonads are introduced, motivated by the need for an analogue of the homotopy operads of \cite{laan}. We show what it would mean for a quasi-monad to distribute over a quasi-comonad, while ensuring that these weaker structures still enhance a category to form a quasi-descent datum. In Appendix \ref{shaa}, we relate $A_{\infty}$-algebras to quasi-semigroups, and thus compare quasi-operads with homotopy operads.

I would like to thank Carlos Simpson for alerting me to the similarities between the SDCs of \cite{paper1} and Leinster's homotopy monoids (\cite{leinster}). I would also like  to thank the anonymous referee heartily for diligently identifying many errors and omissions in the original manuscript.

\tableofcontents

\section{Monoidal structures up to homotopy}\label{leinrev}

In this section, we will introduce various structures which will provide the framework for the rest of the paper. The main concepts are those of a quasi-comonoids and quasi-descent data, which will  provide the intermediate step between monads and comonads one one hand, and strong homotopy structures on the other.

\subsection{Quasi-comonoids}

\begin{definition}\label{delta**}
Define $\Delta_{**}$ to be the subcategory of the ordinal number category $\Delta$ containing only  those non-decreasing (i.e. $f(i+1) \ge f(i)$) morphisms $f:\mathbf{m} \to \mathbf{n}$ with $f(0)=0, f(m)=n$. We define a monoidal structure on this category by setting $\mathbf{m}\ten \mathbf{n}= \mathbf{m+n}$, with 
$$
(f\ten g)(i)= \left\{ \begin{matrix} f(i) & i\le m \\ g(i-m)+p & i \ge m, \end{matrix} \right.
$$ 
for $f:\mathbf{m} \to \mathbf{p}$, $g:\mathbf{n} \to \mathbf{q}$. 
\end{definition}

\begin{remark}\label{cflein1}
There is an isomorphism $\Delta_{**}^{\op} \cong \Delta_0$, the category of finite sets (i.e. $\Delta \sqcup \emptyset$), given by $\mathbf{n} \mapsto \mathbf{n-1}$, where $\mathbf{-1}:=\emptyset$, with the coboundary morphisms $\pd^i$ mapping to $\sigma_{i-1}$, and the coface morphisms $\sigma^i$ mapping to $\pd_i$. 
\end{remark}

\begin{remark}\label{qmstructure}
Given a category $\C$, a functor $X:\Delta_{**} \to \C$ consists of objects $X^n \in \C$, with all of the operations $\pd^i, \sigma^i$ of a cosimplicial complex except $\pd^0, \pd^{n+1}:X^n\to X^{n+1}$. 
\end{remark}

\begin{definition}\label{qmdef}
Given a monoidal category $\C$, define a quasi-comonoid $X$ in $\C$ to be a lax monoidal functor $X: \Delta_{**} \to \C$. This means that we have maps
$$
\zeta^{mn}:  X^m\ten X^n \to X^{m+n} ,\quad \zeta^0:1 \to X^0 ,
$$
 satisfying naturality and coherence, where $1$ is the unit in the category. If $\C$ is a model category, say that $X$ is a homotopy comonoid whenever the maps $\zeta^{mn}, \zeta^0$ are all weak equivalences. This is equivalent to the definition in \cite{leinster}, via the comparison of Remark \ref{cflein1}. 
 
 Define a quasi-monoid in $\C$ to be a quasi-comonoid in $\C^{\op}$. 
 We let  $QM^*(\C)$ denote the category of quasi-comonoids in $\C$.
\end{definition}

\begin{lemma}\label{qmlemma}
Giving  a quasi-comonoid $X$ in $\C$ is equivalent to giving objects $X^n \in \C$ for $n \in \N_0$, together with morphisms
$$
\begin{matrix}
\pd^i:X^n \to X^{n+1} & 1\le i \le n\\
\sigma^i:X^{n}\to X^{n-1} &0 \le i <n,
\end{matrix}
$$
an associative product $\zeta^{mn}:X^m \ten X^n \to X^{m+n}$, with identity $\zeta^0: 1 \to X^0$, where $1$ is the unit in the category, satisfying:
\begin{enumerate}
\item $\pd^j\pd^i=\pd^i\pd^{j-1}\quad i<j$.
\item $\sigma^j\sigma^i=\sigma^i\sigma^{j+1} \quad i \le j$.
\item 
$
\sigma^j\pd^i=\left\{\begin{matrix}
			\pd^i\sigma^{j-1} & i<j \\
			\id		& i=j,\,i=j+1 \\
			\pd^{i-1}\sigma^j & i >j+1
			\end{matrix} \right. .
$
\item $\zeta^{m+1,n}(\pd^i\ten \id)=\pd^i\zeta^{mn}$.
\item $\zeta^{m,n+1}(\id\ten \pd^i)=\pd^{i+m}\zeta^{mn}$.
\item $\zeta^{m-1,n}(\sigma^i\ten \id)=\sigma^i\zeta^{mn}$.
\item $\zeta^{m,n-1}(\id\ten \sigma^i)=\sigma^{i+m}\zeta^{mn}$.
\end{enumerate}
\end{lemma}
\begin{proof}
This is a straightforward consequence of Remark \ref{qmstructure}, together with an analysis of the interaction in $\Delta_{**}$ of the monoidal structure and the morphisms $\sigma^i, \pd^i$.
\end{proof}

\begin{remark}
When the maps $\zeta^{mn}, \zeta^0$ are all isomorphisms, this becomes equivalent to the definition of a comonoid $C$. The correspondence is given by setting $X^n:= C^{\ten n}$, and letting $\pd^1: C \to C\ten C$ be the coproduct.
\end{remark}

In order to simplify the notation, we will write $x*y$ instead of $\zeta^{m,n}(x,y)$ from now on.

\subsubsection{Maurer-Cartan}

Observe the category of comonoids in $(\Set,\by)$ is equivalent to $\Set$ itself, since comultiplication $\Delta:X \to X \by X$ is necessarily the diagonal embedding. This gives a functor $\iota: \Set \to QM^*(\Set)$ from sets to quasi-comonoids.

\begin{definition}\label{mcdef}
Define the functor $\mc: QM^*(\Set) \to \Set$ by 
$$
\mc(E):= \Hom(\iota\bt, E),
$$
where $\bt$ is the one-point set. Explicitly,
$$
\mc(E) = \{\omega \in E^1\,:\, \sigma^0\omega =1, \,\pd^1\omega = \omega *\omega\}.
$$
\end{definition}

The reason for notation is that any cosimplicial unital ring $R$ has  a quasi-comonoid structure, with $*$ given by the Alexander-Whitney cup product $\cup$, and then 
\begin{eqnarray*}
\mc(R)&=& \{1_1+\alpha \in R^1\,:\, \sigma^0\alpha=0\, 1_2+ \pd^1\alpha = 1_2+ 1_1\cup\alpha +\alpha\cup 1_1 + \alpha\cup \alpha  \}\\
&=& 1+\{\alpha \in N^1R\,:\, \pd^1\alpha = \pd^2\alpha+\pd^0\alpha+ \alpha\cup \alpha\}\\
&=& 1+\{\alpha \in N^1R\,:\, d\alpha + \alpha\cup \alpha =0\},
\end{eqnarray*}
where $NR$ is the cosimplicial normalisation of $R$, with differential $d$. Thus $\mc(R)$  is just the classical set of Maurer-Cartan forms of the differential graded algebra $NR$, which parametrises flat connections on a vector bundle.

\subsection{Quasi-descent data}
We will now introduce quasi-descent data, which will form the bridge between categories equipped with monads and/or comonads, and Segal spaces. \S \ref{algsn} will show how monads and comonads give rise to quasi-descent data, while much of the remainder of the paper relates quasi-descent data to Segal spaces.

\begin{definition}
Given a monoidal category $\C$ and a set $\cO$, let a $\C$-valued descent datum $D=(C,G)$ on objects $\cO$ consist of:
\begin{enumerate}
		\item objects $G(a,b) \in \C$ for each pair $a,b \in \cO$;
	\item compatible systems  $G(a,b)\ten G(b,c) \cong  G(a,c)$ and $G(a,a) \cong 1$ of transition isomorphisms, for all $a,b,c \in \cO$ (the cocycle condition);
	\item comonoids $C(a) \in \C$ for each $a \in \cO$;
	\item isomorphisms $C(a) \ten G(a,b)\cong G(a,b)\ten C(b)$ for all $a,b \in \cO$, compatible with the comonoidal structures and transition isomorphisms.
\end{enumerate}
Note that these conditions imply that $G(b,a)\ten C(a) \ten G(a,b)$ is isomorphic as a comonoid to $C(b)$.
\end{definition}

The reason for this terminology is that for an open  cover $\{U_a\}_{a \in \cO}$  of topological spaces $X$, a descent datum consists of sheaves $C_a$ on $U_a$, with isomorphisms $g_{ab}: C_b|_{U_a\cap U_b} \to C_a|_{U_a\cap U_b}$, with  $g_{ab}\circ g_{bc}=g_{ac}: C_c|_{U_a\cap U_b\cap U_c} \to C_a|_{U_a\cap U_b\cap U_c} $ (and $g_{aa}$ necessarily the identity).

\begin{definition}
Let $\Dat_{\cO}(\C)$ be the category of $\C$-valued descent data on objects $\cO$, and let $\Dat(\C)$ be the category of pairs $(\cO,D)$, for $\cO$ a set and $D=(C,G)$ a descent datum on objects $\cO$. Write $\Dat:= \Dat(\Set, \by)$, and similarly for $\Dat_{\cO}$. Note that there is a functor $\Dat(\C) \to \Cat(\C)$ to $\C$-enriched categories, sending $(\cO,D)$ to the category with objects  $\cO$ and morphisms $G(x,y)$.
\end{definition}

\begin{definition}\label{qudesc}
Given a monoidal category $\C$ and a set $\cO$, let a $\C$-valued quasi-descent datum on objects $\cO$ consist of:
\begin{enumerate}
		\item objects $X(a,b) \in \C^{\Delta_{**}}$ for all $a,b \in \cO$;
		\item morphisms $X(a,b)^m \ten X(b,c)^n \xra{*} X(a,c)^{m+n}$ making the following diagram commute for all $a,b,c \in \cO$
		$$
		\begin{CD} \Delta_{**} \by \Delta_{**} @>{ X(a,b)\ten X(b,c)}>> \C  \\
		@V{\by }VV @VV{*}V\\
		\Delta_{**} @>{X(a,c) }>> \C.
		\end{CD}
		$$
		\item morphisms $1 \to X(a,a)^0$    for all $a \in \cO$, acting as the identity for the multiplication $*$.
		\end{enumerate}
\end{definition}

\begin{remarks}\label{quasidescentcat}
A $\C$-valued (quasi-)descent datum on the set  with one object is just a (quasi-)comonoid in $\C$. Given a  $\C$-valued descent datum $(C,G)$ on objects $\cO$, we may define a quasi-descent datum $X$ by $X(a,b)^n:= C(a)^{\ten n}\ten G(a,b)\cong G(a,b)\ten C(b)^{\ten n}$ via the transition maps, with the maps $\pd^i$ given by the coproducts on $C$. 

If the category $\C$ contains all finite coproducts, then we may define a monoidal structure on $\C^{\Delta_{**}}$ by setting
$$
(X\ten Y)^n := \coprod_{a+b=n} X^a\ten Y^b,
$$
with operations given by 
\begin{eqnarray*}
\pd^i(x \ten y) &=& \left\{\begin{matrix}(\pd^ix)\ten y & i\le a \\ x\ten (\pd^{i-a}y) & i>a; \end{matrix} \right.\\
\sigma^i(x\ten y)&=& \left\{\begin{matrix}(\sigma^ix)\ten y & i< a \\ x\ten (\sigma^{i-a}y) & i\ge a. \end{matrix} \right.
\end{eqnarray*}
Then a quasi-descent datum on objects $\cO$ is just a $\C^{\Delta_{**}}$-enriched category with objects $\cO$.
\end{remarks}

\begin{definition}
Let $Q\Dat(\C)$ be the category of $\C$-valued quasi-descent data, i.e. of pairs $(\cO,X)$ for $\cO$ a set and $X$ a quasi-descent datum on objects $\cO$. This admits a functor to the category of $\C$-enriched categories by taking the $\Hom$-space underlying $X(a,b)$ to be $X(a,b)^0$. Let $Q\Dat:= Q\Dat(\Set, \by)$ and $sQ\Dat:= Q\Dat(\bS, \by)$. 

Let $Q\Dpd$ (resp. $sQ\Dpd$) be the full subcategory of $Q\Dat$ (resp. $sQ\Dat$) consisting of 
objects whose underlying categories (resp. simplicial categories) are groupoids (resp. simplicial groupoids). In other words, all elements of $X(a,b)^0$ must be invertible under $*$.
\end{definition}

\subsection{Adjoint functors}

From now on, we will systematically make use of the identification in Remarks \ref{quasidescentcat} of quasi-descent data with  $\C^{\Delta_{**}}$-enriched categories.

We have seen that there is a functor $(-)^0: Q\Dat \to \Cat$ given by sending $\cD$ to the category $\cD^0$ with objects $\Ob \cD^0:=\Ob \cD$ and morphisms $\D^0(a,b):= \cD(a,b)^0$.

\begin{lemma}\label{alg*}
The functor $(-)^0$ has a right adjoint $\alg^*$, given by $\Ob\alg^*\C= \Ob \C$, with 
$$
(\alg^*\C)(a,b)^n= \C(a,b)
$$
for all $n$, with multiplication as in $\C$, and all operations $\pd^i, \sigma^i$ acting trivially.
\end{lemma}
\begin{proof}
Given a functor  $F: \cD^0\to \C$, the associated morphism $\cD \to \alg^*\C $ in $Q\Dat$ is given by $F$ on objects, with morphisms $ F \circ (\sigma^0)^n:  \cD(a,b)^n \to \C(a,b)= (\alg^*\C)(a,b)^n $. That all morphisms $\cD \to \alg^*\C $ arise in this way follows because $\sigma^0$ acts trivially on $(\alg^*\C)(a,b)$.
\end{proof}

\begin{definition}\label{alg}
Define $\alg:  Q\Dat \to \Cat$ as follows. The objects of  $\alg(\cD)$ are pairs $(D, \omega)$, for $D \in \Ob \cD$ and $\omega \in \mc(\cD(D,D))$, for $\mc$ as in Definition \ref{mcdef}. Morphisms from $(D, \omega)$ to $(D', \omega')$ are given by  $f \in \cD(D,D')^0$ such that
$$
f*\omega = \omega'*f \in \cD(D, D')^1.
$$
\end{definition}

\begin{lemma}\label{algadj}
The functor $\alg^*$ is left adjoint to $\alg$.
\end{lemma}
\begin{proof}
Given a functor $F:\C \to \alg(\cD)$, construct  the corresponding morphism $G: \alg^*\C  \to \cD$ as follows. For  $a \in \Ob C$, set $G(a) \in \Ob D$ to be the object underlying $F(a)$, and then define $G: \alg^*\C(a,b)^n \to \cD(Ga,Gb)^n$ by $G(x):= \omega_{F(b)}^n * F(x)= F(x)*\omega_{F(a)}^n$, where $F(a)= (G(a), \omega_{F(a)})$.  Conversely,  any morphism $ G: \alg^*\C  \to \cD$ is determined by $\alg^*\C^0 \to \cD^0$, together with the elements $G(\id_a) \in \cD(a,a)^1$.   
\end{proof}

\begin{remark}
Note that the functor $\alg^*$ is fully faithful, so $(\alg^*\C)^0 \simeq \C \simeq \alg \alg^*\C$.
\end{remark}

\begin{lemma}\label{idemptyset}
The functor $(-)^0: Q\Dat \to \Cat$ has a left adjoint.
\end{lemma}
\begin{proof}
The left adjoint is the functor $(\id/\emptyset)$ given by $\Ob(\id/\emptyset)(\C)= \Ob \C$, with morphisms 
$$
(\id/\emptyset)(\C)(a,b)^n = \left\{\begin{matrix} \C(a,b) & n=0 \\ \emptyset & n>0.\end{matrix}\right. 
$$
\end{proof}

\subsection{Quasi-bicomonoids and diagonals}

\begin{definition}\label{qmmdef}
Given a monoidal category $\C$, let the category  $QMM^*(\C)$ of quasi-bicomonoids consist of lax monoidal functors
$$
X : \Delta_{**} \by \Delta_{**} \to \C.
$$
\end{definition}

\begin{definition}
Define a functor $\diag : QMM^*(\C) \to QM^*(\C)$ via the diagonal functor $\Delta_{**} \to \Delta_{**} \by \Delta_{**}$. Explicitly, $(\diag E)^n=E^{nn}$, with the same product and identity as $E$, and operations $\pd^i=\pd^i_h\pd^i_v$ and $\sigma^i= \sigma^i_h\sigma^i_v$.

The functor $\diag$ preserves all limits, so by the Special Adjoint Functor Theorem (\cite{mac} Theorem V.8.2), it will have  a left adjoint $\diag^*$ for all the categories $\C$ which we will encounter, since they all satisfy the solution set condition. \end{definition}

\begin{definition}
The category of quasi-comonoids in $\Set \by \Set$ is isomorphic to $\Set \by \Set$, with comultiplication necessarily given by the diagonal $(X,Y) \to (X \by X, Y\by Y)$.
Inclusion of comonoids in
quasi-comonoids  then gives a  functor $\iota: \Set\by \Set \to QMM^*(\Set)$, with  $\iota(X,Y)^{m,n}= X^m\by Y^n$
\end{definition}

\begin{lemma}\label{diagset}
For $E^{\bt,\bt} \in QMM^*(\Set)$,
$$
\mc(\diag E) \cong \{(\alpha, \beta) \in \mc(E^{\bt,0}) \by \mc(E^{0, \bt})\,:\, \alpha*\beta=\beta*\alpha \in E^{11}\} 
$$
\end{lemma}
\begin{proof}
Given $\omega \in \mc(\diag E)$, set $\alpha:= \sigma^0_v\omega,\, \beta:=\sigma^0_h\omega$. Since $\sigma^0\omega=1$, we have $\sigma^0_h\alpha= \sigma^0_v\beta=1$. Applying the operations $(\sigma^0_v)^2, \sigma^1_h\sigma^0_v, \sigma^1_h\sigma^0_v, (\sigma^0_h)^2$ (respectively) 
to the equation $\omega*\omega=\pd^1\omega$ gives the equations
$$
\alpha*\alpha= \pd^1_h\alpha, \quad \alpha*\beta =\omega,\quad \beta*\alpha = \omega, \quad \beta*\beta=\pd^1_v\beta.
$$

This shows that the function $\omega \mapsto (\alpha, \beta)$ is well-defined on the sets above, and has inverse $(\alpha, \beta) \mapsto \alpha*\beta$. 
\end{proof}

\begin{corollary}\label{iotadiag}
$\diag^*\iota\bt=\iota(\bt,\bt)$.
\end{corollary}
\begin{proof}
Given a morphism $f: \iota(\bt,\bt) \to E$ in $QMM^*(\Set)$, let $\alpha, \beta$ be the images of the unique elements of $\iota(\bt,\bt)^{1,0}, \iota(\bt,\bt)^{0,1}$. These generate $\iota(\bt,\bt)$, subject to  the conditions of Lemma \ref{diagset}.
\end{proof}

\subsection{Nerves}
We will now see how to associate quasi-comonoids to small diagrams. In Remark \ref{shdiagram}, this will give rise to the notion of strong homotopy diagrams.

\subsubsection{Categories}

Given a set $\cO$, we will write $\Cat_{\cO}$ for the category of categories with object set $\cO$.

\begin{definition}
 Given $K\in \bS$,  define $P_K: \Cat_{K_0} \to QM^*(\Set)$ by 
$$
P_K(\C)^a= \prod_{x \in K_a}\C((\pd_0)^{a}x, (\pd_1)^{a}x), 
$$ 
with operations 
\begin{eqnarray*}
\pd^i(e)(x)&:=& e(\pd_i x)\\
\sigma^j(e)&:=& e(\sigma_j y),\\ 
(f*e)(z)&:=& f((\pd_{a+1})^bz)\circ e((\pd_{0})^az),
\end{eqnarray*}
for $f \in P_K(\C)^a, e \in P_K(\C)^b$.

Let $P_n:= P_{\Delta^n}$.
\end{definition}

Let $\Cat_{n}$ be the category of categories on (the $n+1$) objects $[0,n]$.
\begin{lemma}\label{pnleft}
The functor $P_n:\Cat_n \to QM^*(\Set)$ has a left adjoint $P_n^*$, given by
$$
(P_n^*E)(i,j) = \left\{ \begin{matrix} E^{j-i} & j \ge i \\ \emptyset & j < i, \end{matrix} \right.
$$
with multiplication given by $*$, and identities $ 1 \in E^0$.
\end{lemma}
\begin{proof}
Define $\nabla^n \in \Set^{\Delta_{**}^{\op}}$ by $(\nabla^n)_r :=\Hom_{\Delta_{**}}(\mathbf{r}, \mathbf{n})$. If $U: \bS \to \Set^{\Delta_{**}^{\op}}$ is the forgetful functor, then define $b: UK \to K_0 \by K_0$ by $x \mapsto ((\pd_1)^nx, (\pd_0)^nx)$, for $x \in K_n$. Then
$$
U\Delta^n = \coprod_{(i,j) \in \Delta^n_0} b^{-1}(i,j)\cong \coprod_{0\le i \le j \le n} \nabla^{j-i}.
$$

For  $\C \in \Cat_n$, $P_n\C$ decomposes in $\Set^{\Delta_{**}} $ as  
$$
P_nE= \prod_{0\le i \le j \le n} \C(i,j)^{\nabla^{j-i}},
$$
where, for $K \in \Set^{\Delta_{**}^{\op}}$  and a set $S$, we define $S^K \in \Set^{\Delta_{**}}$ by $(S^K)_r= S^{K_r}$.

For $E \in QM^*(\Set)$ and $\C \in \Cat_n$, this implies that 
\begin{eqnarray*}
\Hom_{\Set^{\Delta_{**}}}(E, P_n\C) &=& \prod_{0\le i \le j \le n} \Hom_{\Set^{\Delta_{**}}}(E,\C(i,j)^{\nabla^{j-i}})\\
& =& \prod_{0\le i \le j \le n} \Hom_{\Set}(E^{j-i}, \C(i,j)).
\end{eqnarray*}
Analysis of the product now gives the required result.
\end{proof}

\begin{definition}
Given a category $\C$, a set $\cO$ and a morphism $f: \cO \to \Ob \C$, define $f^{-1}\C$ to be the category with objects $\cO$ and morphisms $(f^{-1}\C)(a,b)= \C(f a, f b)$.
\end{definition}


\begin{lemma}\label{pnnerve}
For any category $\C$, there is a natural isomorphism
$$
(B\C)_n \cong \coprod_{f \in (\Ob\C)^{[0,n]}}\mc(P_nf^{-1}\C).
$$
\end{lemma}
\begin{proof}
Note that the right-hand side is just $\Hom_{\Cat}(P_n^*\iota\bt, \C)$. Since $(\iota\bt)^n=\bt$ for all $n$, we have $ P_n^*\iota\bt \cong \on$ by Lemma \ref{pnleft}, where we regard $\on$ as a category (with objects $[0,n]$, and a single morphism $i\to j$ whenever $i \le j$).  This completes the proof,  since the nerve is given by $(B\C)_n= \Hom_{\Cat}(\mathbf{n}, \C)$.
\end{proof}

We now generalise these results to more general  categories.

\begin{definition}\label{uleft}
Given $E \in \Set^{\Delta_{**}}$ and $X \in \Set^{\Delta_{**}^{\op}}$, define the set $X\uleft{\by} E$ to be the quotient of 
$
\{ (x, e) \in \coprod_n X_n\by E^n\}$ by the equivalence relation generated by
$$
 (x, \pd^ie)\sim (\pd_ix,e),\, (x, \sigma^ie)\sim (\sigma_ix,e).
$$
\end{definition}

\begin{definition}
Given $K \in \bS$ and $a,b \in K_0$, define $K(a,b) \in \Set^{\Delta_{**}^{\op}}$ by setting 
$$
K(a,b)_n:= \{x \in K_n\,:\, (\pd_1)^nx=a,\, (\pd_0)^nx=b\}.
$$
\end{definition}

\begin{lemma}\label{pkleft}
Given a (small) category $\bI$, the left adjoint $P_{B\bI}^*$ to the functor $P_{B\bI}:\Cat_{\Ob \bI} \to QM^*(\Set)$ is given by 
$$
(P_{B\bI}^*E)(a,b) = (B\bI)(a,b) \uleft{\by}E 
$$
for $a, b \in \Ob \bI$,  with product given by
$$
(x, e)\circ (y,f):= (x\star y, e*f),
$$
where $\star$ denotes concatenation of strings of morphisms in $\bI$, and $*$ is the product on $E$.
 \end{lemma}
\begin{proof}
For $K \in \bS$, the category $P_K^*E$ has  objects $K_0$.  Morphisms are generated  under composition by $ K(a,b) \uleft{\by}E$ in $(P_{K}^*E)(a,b)$ for $a, b \in K_0$, subject to the condition that for $e \in E^m, f\in E^n$ and $x \in K_{m+n}$, 
$$
(x, e*f) \sim ((\pd_{m+1})^nx, e)\circ (\pd_0)^mx, f).
$$

When $K = \bI$, the map  $((\pd_{m+1})^n, (\pd_0)^m): K_{m+n} \to K_{m}\by_{(\pd_0)^m,K_0,(\pd_{1})^n}K_n$ is an isomorphism, so any product of generators is a generator, giving the required result.
\end{proof}

\begin{lemma}\label{pknerve}
For any  category $\C$, there is a natural isomorphism
$$
\Hom_{\Cat}(\bI, \C) \cong \coprod_{f \in (\Ob\C)^{\Ob\bI}}\mc(P_{B\bI}f^{-1}\C).
$$
\end{lemma}
\begin{proof}
The proof of Lemma \ref{pnnerve} carries over, noting that $P_{B\bI}^*\iota\bt \cong \bI$.
\end{proof}

\subsubsection{Quasi-descent data}

\begin{definition}\label{pnset}
 Given $K \in \bS$, define $P_K: Q\Dat_{K_0} \to QMM^*(\Set)$ by 
$$
P_K(\cD)^{a,b}= \prod_{x \in K_b}\cD((\pd_0)^{b}x, (\pd_1)^{b}x)^a, 
$$ 
with operations 
\begin{eqnarray*}
\pd_v^i(e)(x)&:=& e(\pd_i x)\\
\sigma_v^j(e)&:=& e(\sigma_j y),\\ 
(f*e)(z)&:=& f((\pd_{b+1})^{b'}z)* e((\pd_{0})^bz).
\end{eqnarray*}
for $f \in P_K(\C)^{a,b}, \, e \in P_K(\C)^{a',b'}$. The horizontal operations are $\pd^i_h=\pd^i_{\cD}, \sigma^i_h=\sigma^i_{\cD}$. 

Note that $P_n= P_{\Delta^n}$.
\end{definition}

Let $Q\Dat_{n}$ be the category of quasi-descent data on the $n+1$ objects $[0,n]$.
\begin{lemma}\label{pnleft2}
The functor $P_n:Q\Dat_n \to QMM^*(\Set)$ has a left adjoint $P_n^*$, given by
$$
(P_n^*)(i,j)^a = \left\{ \begin{matrix} E^{a,j-i} & j \ge i \\ \emptyset & j < i, \end{matrix} \right.
$$
with multiplication given by $*$, identities $ 1 \in E^{00}$, and operations $\pd^i_h, \sigma^i_h$.
\end{lemma}
\begin{proof}
The proof of Lemma \ref{pnleft} adapts to this generality.
\end{proof}

\begin{definition}
Given a quasi-descent datum $\cD$, a set $\cO$ and a morphism $f: \cO \to \Ob \cD$, define $f^{-1}\cD$ to be the quasi-descent datum with objects $\cO$ and morphisms $(f^{-1}\cD)(a,b)^i= \cD(f a, f b)^i$.
\end{definition}

\begin{lemma}\label{pndiag}
For any quasi-descent datum  $\cD$, there is a natural isomorphism
$$
(B\alg\cD)_n \cong \coprod_{f \in (\Ob\cD)^{[0,n]}}\mc(\diag P_nf^{-1}\cD).
$$
\end{lemma}
\begin{proof}
By Lemma \ref{algadj}, the left-hand side is $\Hom_{\Cat}(\mathbf{n}, \alg\cD) \cong \Hom_{Q\Dat}( \alg^*\mathbf{n},\cD)$. Meanwhile, the right-hand side is $\Hom_{Q\Dat}(P_n^*\diag^*\iota\bt, \cD)$, so it suffices to show that $\alg^*\mathbf{n} \cong P_n^*\diag^*\iota\bt$. By Corollary \ref{iotadiag}, $P_n^*\diag^*\iota\bt= P_n^*\iota(\bt,\bt)$, and Lemma \ref{pnleft2} shows that $(P_n^*\iota(\bt,\bt))(i,j)^a$ is $\bt$ for $j \ge i$ and $\emptyset$ for $j<i$, so  $P_n^*\iota(\bt,\bt)= \alg^*\on $.
\end{proof}

\begin{lemma}\label{pkleft2}
Given a category $\bI$, the left adjoint $P_{B\bI}^*$ to the functor $P_{B\bI}:Q\Dat_{\Ob \bI} \to QMM^*(\Set)$ is given by 
$$
(P_{B\bI}^*E)(a,b)^n = (B\bI)(a,b) \uleft{\by}E^n 
$$
for $a, b \in \Ob \bI$,  with product given by
$$
(x, e)\circ (y,f):= (x\star y, e*f),
$$
where $\star$ denotes concatenation of strings of morphisms in $\bI$, and $*$ is the product on $E$.
 \end{lemma}
\begin{proof}
The proof of Lemma \ref{pkleft} adapts to this generality.
\end{proof}

\begin{lemma}\label{sdcdiagramsub}
For any quasi-descent datum  $\cD$ and a category $\bI$, there is a natural isomorphism

$$
\Hom_{\Cat}(\bI, \alg\cD) \cong \coprod_{f \in (\Ob\cD)^{\Ob\bI}}\mc(\diag P_{B\bI} f^{-1}\cD).
$$
\end{lemma}
\begin{proof}
The proof of Lemma \ref{pndiag} carries over, noting that $P_{B\bI}^*\iota(\bt,\bt) \cong \alg^*\bI$.
\end{proof}

\section{Algebras, coalgebras and bialgebras}\label{algsn}

\subsection{Algebras and coalgebras}

\begin{definition}
A monad (or triple) on a category $\cB$ is a monoid in the category of endofunctors of $\cB$ (with the monoidal structure given by composition of functors). A comonad (or cotriple) is a comonoid in the category of endofunctors of $\cB$.
\end{definition}

\begin{lemma}\label{adjnmonad}
Take an adjunction
$$
\xymatrix@1{\cA \ar@<1ex>[r]^G_{\top} & \cE \ar@<1ex>[l]^F}
$$
with unit $\eta:\id \to GF$ and co-unit $\vareps:FG \to \id$. Then $\top:=GF$ is a monad on $\cE$ with unit $\eta$ and multiplication $\mu:=G\vareps F$, while $\bot:= FG$ is a comonad on $\cA$, with  co-unit $\vareps$ and comultiplication $\Delta:=F\eta G$.
\end{lemma}
\begin{proof}
For the monad $\top$, this is \cite{mac} \S VI.1, with the comonadic results following by duality. 
\end{proof}

\begin{definition}
Given a monad $(\top,\mu, \eta) $ on a category $\cE$, 
define the category $\cE^{\top}$ of $\top$-algebras to have objects
$$
\top E \xra{\theta} E
$$ 
(for $E \in \cE$),
such that $\theta\circ \eta_E=\id$ and $\theta \circ \top \theta= \theta \circ  \mu_E$.
	
A morphism 
$$
g: ( \top E_1 \xra{\theta} E_1 ) \to  (\top E_2 \xra{\phi} E_2)  
$$	
of $\top$-algebras is a morphism $g:E_1 \to E_2$ in $\cE$ such that $\phi\circ \top g= g \circ \theta$.
\end{definition}

We define the comparison functor $K:\cA \to \cE^{\top}$ by
$$
B \mapsto ( GFGB \xra{G\vareps_B}  GB )
$$
on objects, and $K(g)=G(g)$ on morphisms.

\begin{definition}
The adjunction 
$$
\xymatrix@1{\cA \ar@<1ex>[r]^G_{\top} & \cE \ar@<1ex>[l]^F},
$$
is said to be  monadic (or tripleable) if $K:\cA \to \cE^{\top}$ is an equivalence.
\end{definition}

\begin{examples}
Intuitively, monadic adjunctions correspond to algebraic theories, such as the adjunction
$$
\xymatrix@1{\Ring \ar@<1ex>[r]^U_{\top} &\Set  \ar@<1ex>[l]^{\Z[-]},}
$$
 between rings and sets, $U$ being the forgetful functor. Other examples are $k$-algebras   over $k$-vector spaces, or groups over sets.
\end{examples}

\begin{definition}
Dually, given a comonad $(\bot, \Delta, \vareps)$ on a category $\cA$, we  define the category $\cA_{\bot}$ of $\bot$-coalgebras by 
$$
(\cA_{\bot})^{\op} := (\cA^{\op})^{\bot},
$$
noting that $\bot$ is a monad on $\cA^{\op}$. The adjunction of Lemma \ref{adjnmonad} is said to be comonadic (or cotripleable) if the adjunction on opposite categories is monadic.
\end{definition}

\begin{examples}
If $X$ is a topological space (or any site with enough points) and $X'$ is the set of points of $X$, let  $u:X'\to X$ be the associated morphism. 
Then the adjunction
$$
\xymatrix@1{\Shf(X') \ar@<1ex>[r]^{u_*}_{\top} &\Shf(X)  \ar@<1ex>[l]^{u^{-1}},}
$$
 on the associated categories of sheaves  is comonadic, so  $\Shf(X)$ is equivalent  $u^{-1}u_*$-coalgebras in  the  category $\Shf(X')$ of sheaves  (or equivalently presheaves)  on $X'$.

A more prosaic example is that for any ring $A$, the category of $A$-coalgebras is comonadic over the category of $A$-modules.
\end{examples}

\subsection{Quasi-descent data from monads}

Given a monad  $(\top, \mu, \eta)$ on a category $\cB$, and an object $B \in \cB$,  there is a  quasi-comonoid $E(B)$ given by
$$
E^n(B)= \Hom_{\cB}(\top^n B, B)
$$
in $(\Set, \by)$, 
with product
$g*h=g\circ\top^n h$, and  for $g \in E^n(B)$,
\begin{eqnarray*}
\pd^i(g) &=& g \circ \top^{i-1}\mu_{\top^{n-i}B}\\
\sigma^i(g) &=& g \circ \top^{i}\eta_{\top^{n-i-1}B}.
\end{eqnarray*}

If we replace $\cB$ with a simplicial category, then $E(B)$ becomes a quasi-comonoid in $(\bS, \by)$.
Note that these constructions also all work for a comonad $(\bot, \Delta, \vareps)$, by contravariance. 

\begin{lemma} Given an object $B \in \cB$, the set of $\top$-algebra structures on $B$ is
$
\mc(E(B))
$.
\end{lemma}
\begin{proof}
This follows immediately from the explicit description in Definition \ref{mcdef}.
\end{proof}

\begin{proposition}\label{enrichtop}
Given a monad  $(\top, \mu, \eta)$ (resp. a comonad $(\bot, \Delta, \vareps)$) on a category $\cB$, there is a natural structure of a  $\Set^{\Delta_{**}}$-enriched category on $\cB$, i.e. a quasi-descent datum on objects $\Ob \cB$.
\end{proposition}
\begin{proof}
Set 
$$
\cHom(B,B')^n:= \Hom_{\cB}(\top^n B, B') \quad \text{(resp. } \Hom_{\cB}( B,\bot^n B') \text{),}
$$
with product and operations as above.
\end{proof}

\begin{proposition}\label{algworks}
The category $\cB^{\top}$ (resp. $\cB_{\bot}$) of $\top$-algebras (resp. $\bot$-coalgebras) on $\cB$ is isomorphic to the image under the functor $\alg: Q\Dat \to \Cat$ (Definition \ref{alg}) of the quasi-descent datum on $\cB$ given in Proposition \ref{enrichtop}.
\end{proposition}
\begin{proof}
This follows immediately from the definitions.
\end{proof}

\subsection{Bialgebras }
We now show how a bialgebraic structure on a category gives rise to a quasi-descent datum.

As in \cite{osdol} \S IV, take a category $\cB$ equipped with both a monad $(\top, \mu, \eta)$ and a comonad $(\bot, \Delta, \vareps)$, together with a distributivity transformation
$\lambda: \top\bot \abuts \bot \top$
for which the following diagrams commute:
$$
\xymatrix{
\top\bot \ar@{=>}[rr]^{\lambda} \ar@{=>}[d]_{\top\Delta}& &  \bot \top\ar@{=>}[d]_{\Delta\top} \\
\top\bot^2 \ar@{=>}[r]^{\lambda\bot} &\bot\top\bot \ar@{=>}[r]^{\bot\lambda} & \bot^2\top
}
\quad
\xymatrix{
\top\bot \ar@{=>}[rr]^{\lambda} & &  \bot \top \\
\top^2\bot\ar@{=>}[u]_{\mu\bot}  \ar@{=>}[r]^{\top\lambda} & \top\bot\top \ar@{=>}[r]^{\lambda\top}& \bot\top^2\ar@{=>}[u]_{\bot\mu}
}
$$

$$
\xymatrix{
\top\bot \ar@{=>}[rr]^{\lambda} \ar@{=>}[dr]_{\top\vareps}& &  \bot \top\ar@{=>}[dl]^{\vareps\top} \\
& \top} 
\quad
\xymatrix{
\top\bot \ar@{=>}[rr]^{\lambda} & & \bot \top \\
& \bot\ar@{=>}[ul]^{\eta\bot}  \ar@{=>}[ur]_{\bot\eta},}
$$

\begin{definition}\label{bialgdef}
Given a distributive monad-comonad pair $(\top, \bot)$ on a category $\cB$, define the category $\cB^{\top}_{\bot}$ of bialgebras as follows. The objects of  $\cB^{\top}_{\bot}$ are triples $(\alpha, B, \beta)$ with $(\top B \xra{\alpha} B)$ an object of $\cB^{\top}$ and $B \xra{\beta} \bot B$ an object of $\cB_{\bot}$, such that the composition  $(\beta \circ \alpha):\top B \to \bot B$ agrees with the composition 
$$
\top B \xra{\top \beta} \top\bot B \xra{\lambda} \bot \top B \xra{\bot \alpha} \bot B. 
$$  

A morphism $f:(\alpha, B, \beta) \to (\alpha', B', \beta')$ is a morphism $f: B \to B'$ in $\cB$ such that $\alpha'\circ \top f = f \circ \alpha$ and $\beta' \circ f = \bot f \circ \beta$.  
\end{definition}

\begin{proposition}\label{enrichtopbot}
The data above give $\cB$ the natural structure of an  $\Set^{\Delta_{**}}$-enriched category, with 
$$
\cHom_{\cB}(B,B')^n= \Hom_{\cB}(\top^n B, \bot^n B').
$$
\end{proposition}
\begin{proof}
We follow \cite{paper2} in describing the operations. Since $\lambda$ is natural, 
$$
(\lambda \bot\top)\circ(\top\bot \lambda)=(\bot\top \lambda)\circ (\lambda \top\bot).
$$
Therefore any composition of $\lambda$'s gives us the same canonical map
$$
\lambda_{m}^{n}:\top^m\bot^n \to \bot^n\top^m,
$$
and we define the product on $\cHom(B',B'')^m \by \cHom(B,B')^n \to \cHom(B, B'') ^{m+n}$ by
$$
g*h= \bot^n(g)\circ\lambda_m^{n}\circ \top^m(h).
$$

The other operations are given by
\begin{eqnarray*}
\pd^i(g) &=&    \bot^{n-i}\Delta_{\bot^{i-1}B}\circ	g \circ \top^{i-1}\mu_{\toph^{n-i}B}\\
\sigma^i(g) &=&	 \bot^{n-i-1}\eps_{\bot^{i}B}\circ	g \circ \top^{i}\eta_{\top^{n-i-1}B}.
\end{eqnarray*}
\end{proof}

To understand how the data $(\top, \bot, \eta, \mu, \vareps, \Delta, \lambda)$ above occur naturally, note that by \cite{osdol} \S IV or \cite{paper2} \S \ref{paper2-gensdc}, these data are equivalent to a diagram
$$
\xymatrix@=8ex{
\cD \ar@<1ex>[r]^{U}_{\top} \ar@<-1ex>[d]_{V} 
&\ar@<1ex>[l]^{F} \cE  \ar@<-1ex>[d]_{V} 
\\
\ar@<-1ex>[u]_{G}^{\dashv}	\cA \ar@<1ex>[r]^{U}_{\top} 
&\ar@<1ex>[l]^{F} \ar@<-1ex>[u]_{G}^{\dashv} \cB,  
}
$$
with $F\dashv U$ monadic, $G\vdash V$ comonadic and $U,V$ commuting with everything (although $G$ and $F$ need not commute). The associated monad on $\cB$ is $\top=UF$, and the comonad $\bot= VG$. Distributivity ensures that $\cD \simeq \cE^{\top}\simeq (\cB_{\bot})^{\top}$ and $\cD\simeq \cA_{\bot}\simeq (\cB^{\top})_{\bot}$. In other words, $\cD \simeq \cB^{\top}_{\bot}$. The functors $F$ are both free $\top$-algebra functors, while the functors $G$ are  both cofree $\bot$-coalgebra functors.

\begin{example}
If $X$ is a topological space (or any site with enough points) and $X'$ is the set of points of $X$, let $\cD$ be the category of sheaves  of rings on $X$. If $\cB$ is the category of sheaves (or equivalently presheaves) of sets on $X'$, then the description above characterises $\cD$ as a category of bialgebras over $\cB$, with the comonad being $u^{-1}u_*$ for $u:X'\to X$, and the monad  being the free polynomial  functor. 
\end{example}

\begin{proposition}\label{bialgworks}
The category $\cB^{\top}_{\bot}$ $(\top,\bot)$-algebras  on $\cB$ is isomorphic to the image under the functor $\alg: Q\Dat \to \Cat$ (Definition \ref{alg}) of the quasi-descent datum $\tilde{\cB}$ on $\cB$ given in Proposition \ref{enrichtopbot}.
\end{proposition}
\begin{proof}
This is essentially \cite{paper2} Theorem \ref{paper2-main}.
Note that $\alg(\tilde{\cB})$ arises naturally as the diagonal of a $\Set^{\Delta_{**}\by \Delta_{**}}$-enriched category.  By Lemma \ref{diagset} and Proposition \ref{algworks}, the objects of $\alg(\tilde{\cB})$ over $B \in \cB$ correspond to pairs $(\alpha, \beta)$, for $(\top B \xra{\alpha} B) \in \cB^{\top}$ and $(B \xra{\beta} \bot B) \in \cB_{\bot}$ satisfying the conditions of Definition \ref{bialgdef}. The description of the morphisms follows similarly.
\end{proof}

\section{Simplicial structures}\label{simplicialsn}
\subsection{Simplicial quasi-comonoids}\label{comonoid}

We now study the structure of the category of simplicial quasi-comonoids introduced in Definition \ref{qmdef}.

\begin{proposition}\label{reedymodel}
There is cofibrantly generated Reedy simplicial model structure on $\bS^{\Delta_{**}}$ in which a map $f:X \to Y$ is 
\begin{enumerate}
\item a weak equivalence if the maps $f^n:X^n\to Y^n$ are all weak equivalences,
\item a cofibration if the maps $f^n:X^n\to Y^n$ are all injective.
\end{enumerate}
\end{proposition}
\begin{proof}
The category ${\Delta_{**}}$ naturally has the structure of a Reedy category, with $\Delta_{**,+}$ and $\Delta_{**,-}$ the  subcategories of injective and surjective maps, thus  giving $\bS^{\Delta_{**}}$ a Reedy model structure. Since $\Delta_{**,-}$ =$\Delta_-$, the matching objects in $\C^{\Delta}$ and $\C^{\Delta_{**}}$ are isomorphic, so it follows that this  model structure on $\bS^{\Delta_{**}}$ is cofibrantly generated, from the  corresponding result for $\bS^{\Delta}$  (as in  \cite{sht} \S VII.4).

It only remains to describe the cofibrations. A morphism $f$ is a cofibration if the $\Delta_{**}$-latching maps are cofibrations in $\bS$. This is equivalent to saying that, for all $i$, the $\Delta_{**}$-latching maps
$$
L^n(f_i):X^{n+1}_i\cup_{ L^n(X_i)} L^n(Y_i)\to Y^{n+1}_i
$$
are injective.

Now, under the comparison $\Delta_{**}\cong \Delta_0^{\op}$, $X_i \in \Set^{\Delta_{**}}$ corresponds to an augmented simplicial set $\breve{(X_i)}$. Thus injectivity of the latching maps says that  $\breve{(X_i)}_{\ge 0} \to \breve{(Y_i)}_{\ge 0}$ is a cofibration of simplicial sets, and that $\breve{(X_i)}_{-1} \to \breve{(Y_i)}_{-1}$ is injective. Since cofibrations in $\bS$ are precisely levelwise injective maps, this is equivalent to saying that the maps $f^n_i: X^n_i\to Y^n_i$ are all injective.
\end{proof}

\begin{lemma}\label{qmsmodel}
There is a cofibrantly generated  simplicial model structure on $QM^*(\bS)$ for which a morphism $f$ is a fibration or weak equivalence whenever the underlying map in $\bS^{\Delta_{**}}$ is so.
\end{lemma}
\begin{proof}
Since the forgetful functor $QM^*(\bS)\to \bS^{\Delta_{**}}$ preserves filtered direct limits  and has a  left adjoint $F$, for any finite object $I \in \bS^{\Delta_{**}}$, the object $FI$ is finite in $QM^*(\bS)$, so \emph{a fortiori} permits the small object argument.
The model structure on $\bS^{\Delta_{**}}$ is cofibrantly generated by finite objects, so  \cite{Hirschhorn} Theorem 11.3.2 gives the required model structure on $QM^*(\bS)$.
\end{proof}

\begin{remark}\label{iota}
Observe that the category of comonoids in $(\bS,\by)$ is just $\bS$ itself, since the comultiplication $\Delta:X \to X \by X$ is necessarily given by the diagonal. Thus there is a functor $\iota: \bS \to QM^*(\bS)$, given by $\iota(X)^m= \overbrace{X\by X \by \ldots \by X}^m$, sending the comonoid $X$ to its associated quasi-comonoid. 
\end{remark}

The following follows immediately from Definition \ref{mcdef}
\begin{lemma}
If $E \in QM^*(\bS)$, and  $\bt$ denotes the constant simplicial set on one element, then 
$$
\mc(E_0) \cong \Hom_{QM^*(\bS)}(\iota\bt,E).
$$
\end{lemma}

\begin{definition}\label{mcdefqm}
Define $\mmc: QM^*(\bS) \to \bS$ by 
$$
\mmc(E) \subset \prod_{n\ge 0} (E^{n+1})^{I^n}
$$
(where $I= \Delta^1 \in \bS$), consisting of those $\underline{\omega}$ satisfying: 
 \begin{eqnarray*}
\omega_m(s_1,\ldots, s_m)*\omega_n(t_1,\ldots, t_n)&=&\omega_{m+n+1}(s_1,\ldots, s_m,0,t_1,\ldots, t_n);\\ 
\pd^i\omega_n(t_1,\ldots,t_n)&=&\omega_{n+1}(t_1, \ldots,t_{i-1},1,t_i,\ldots,t_n);\\ 
\sigma^i\omega_n(t_1,\ldots,t_n)&=&\omega_{n-1}(t_1, \ldots,t_{i-1},\min\{t_i,t_{i+1}\},t_{i+2},\ldots, t_n);\\
\sigma^0\omega_n(t_1,\ldots,t_n)&=&\omega_{n-1}(t_2, \ldots,t_n);\\
\sigma^{n}\omega_n(t_1,\ldots,t_n)&=&\omega_{n-1}(t_1, \ldots,t_{n-1}),\\ 
\sigma^0\omega_0&=&1.
\end{eqnarray*}
We will refer to these as the higher Maurer-Cartan relations.

Define $\mc :QM^*(\bS) \to \Set$ by $\mc(E)=\mmc(E)_0$, noting that this agrees with Definition \ref{mcdef} when $E \in QM^*(\Set)$. Also note that we can recover $\mmc$ from $\mc$, since $\mmc(E)_n = \mc(E^{\Delta^n})$.
\end{definition}

\begin{remark}\label{shob}
Given a distributive monad-comonad pair $(\top, \bot)$ on a simplicial category $\cB$, and an object $B$ of $\cB$,  Proposition \ref{enrichtopbot} gives $\cHHom_{\cB}(B,B) \in QM^*(\bS)$, and we then regard $\mmc(\cHHom_{\cB}(B,B) )$ as being the space of strong homotopy $(\top, \bot)$-bialgebras over $B$. If $\bot$ is trivial, this is essentially the same as Lada's definition of the space of strong homotopy $\top$-bialgebras from \cite{loop} (see Remark \ref{ladark} for differences).
\end{remark}

\begin{definition}\label{sdcmatchsub}
We now define matching objects for $E \in QM^*(\bS)$  by $M^{0}E:=\bullet$, $M^1E:=E^0$, and for $n\ge 2$
$$
M^nE=\{(e_0,e_1,\ldots,e_{n-1}) \in (E^{n-1})^{n}\,|\, \sigma^ie_j=\sigma^{j-1}e_i\, \forall i<j \}.
$$
These correspond to the matching objects $M^{n-1}E$ of  \cite{sht} Lemma VII.4.9, but we have renumbered for consistency with the wider theory of Reedy categories.
\end{definition}

The definition of the Reedy model structure on $\bS^{\Delta_{**}}$ implies the following. 
\begin{lemma}
A morphism $f:E \to F$ in $ QM^*(\bS)$ is a fibration (resp. a trivial fibration) if and only if the relative matching maps 
$$
E^n \to F^n\by_{M^{n}F}M^{n}E
$$
are fibrations (resp.  trivial fibrations) in $\bS$.
\end{lemma}

\begin{lemma}\label{mcqsub}
For any trivial fibration $E \to F$ in $ QM^*(\bS)$, the map
$$
\mmc(E) \to \mmc(F)
$$
is a trivial fibration.
\end{lemma}
\begin{proof}
The idea is to write $\mmc(E)$ as $\Lim \mmc(E)^n$, where we define $\mmc(E)^{n} \subset \prod_{0 \le r \le n} (E^{r+1})^{I^r}$ satisfying the relations of Definition \ref{mcdefqm} above (to level $n$).  We can summarise the Maurer-Cartan relations involving  $\pd^j$ and $*$ as defining a function $f:\mmc(E)^{n-1}\to (E^{n+1})^{\pd I^n}$, where $\pd I^n$ is the boundary of the simplicial set $I^n$. The relations involving  $\sigma^j$ define a function  $g: \mmc(E)^{n-1}\to (M^{n+1}E)^{I^n}$. If we set $\mmc(E)^{-1}=\bullet$, this allows us to write $\mmc(E)^{n}$ as the fibre product
$$
\xymatrix{
\mmc(E)^{n} \ar[r]\ar[d]&\mmc(E)^{n-1} \ar[d]^{(f,g)}\\
(E^{n+1})^{I^n} \ar[r] & (E^{n+1})^{\pd I^n}\by_{(M^{n+1}E)^{\pd I^n}} (M^{n+1}E)^{I^n}.
}
$$

Since the pullback of a trivial fibration is a trivial fibration, it suffices to show that 
$$
(E^{n+1})^{I^n} \to  [(E^{n+1})^{\pd I^n}\by_{(M^{n+1}E)^{\pd I^n}} (M^{n+1}E)^{I^n}]\by_{[ (F^{n+1})^{\pd I^n}\by_{(M^{n+1}F)^{\pd I^n}} (M^{n+1}F)^{I^n}] }(F^{n+1})^{I^n}
$$
is a trivial fibration. By definition, the maps 
$$
E^{n+1} \to  M^{n+1}E\by_{M^{n+1}F}F^{n+1}
$$ 
are trivial fibrations.
 If $X \to Y$ is a trivial fibration, then $X^{I^n} \to X^{\pd I^n}\by_{Y^{\pd I^n}}Y^{I^n} $ is a trivial fibration, since $\pd I^n \to I^n$ is a cofibration in $\bS$ (using the simplicial structure of $\bS$). This gives the required result. 
\end{proof}

\begin{definition}\label{Xidef}
Define  an object $\Xi \in QM^*(\bS) $ by  $\Xi^0=\bt$ and $\Xi^{n+1}=I^{n}$ for $n\ge 0$, where $I=\Delta^1$, with operations
\begin{eqnarray*}
(s_1,\ldots, s_m)*(t_1,\ldots, t_n)&=&(s_1,\ldots, s_m,0,t_1,\ldots, t_n);\\ 
\pd^i(t_1,\ldots,t_n)&=&(t_1, \ldots,t_{i-1},1,t_i,\ldots,t_n);\\ 
\sigma^i(t_1,\ldots,t_n)&=&(t_1, \ldots,t_{i-1},\min\{t_i,t_{i+1}\},t_{i+2},\ldots, t_n);\\
\sigma^0(t_1,\ldots,t_n)&=&(t_2, \ldots,t_n);\\
\sigma^{n}(t_1,\ldots,t_n)&=&(t_1, \ldots,t_{n-1}).
\end{eqnarray*}
\end{definition}

\begin{proposition}\label{rhommc}
For $E \in QM^*(\bS)$ fibrant, there is a natural weak equivalence
$$
\oR \HHom_{QM^*(\bS)}(\iota\bt, E) \simeq \mmc(E)
$$
in $\bS$.
\end{proposition}
\begin{proof}
It follows from Definition \ref{mcdefqm} that
$$
\mmc(E) = \HHom_{ QM^*(\bS)}(\Xi,E).
$$

Next, observe that the unique  map $\Xi \to \iota\bt$ is  a  weak equivalence, since the maps $I^n \to \bt$ are weak equivalences. For any  trivial fibration $E \to F$ in $ QM^*(\bS)$, Lemma \ref{mcqsub} shows that 
$$
\Hom_{ QM^*(\bS)}(\Xi, E) \to \Hom_{ QM^*(\bS)}(\Xi,F)
$$
is surjective, so $\Xi$ has the left lifting property with respect to trivial fibrations, making   it a cofibrant replacement for $\iota\bt$.
Thus, for $E$ fibrant,
$$
\mmc(E)\simeq\oR \HHom_{QM^*(\bS)}(\iota\bt, E).
$$
\end{proof}

\begin{corollary}\label{mcquillen}
The functor $\mmc: QM^*(\bS) \to \bS$ is right Quillen.
\end{corollary}
\begin{proof}
Given a (trivial) fibration $E \to F$, the morphism $\mmc(E) \to \mmc(F)$ is  $\HHom_{ QM^*(\bS)}(\Xi,E)\to\HHom_{ QM^*(\bS)}(\Xi,F) $. This is a (trivial) fibration, since  $QM^*(\bS)$ is a simplicial model category.
\end{proof}

\begin{remarks}\label{ladark}
Lada's definition of a strong homotopy algebra in \cite{loop} differs from Definition \ref{mcdefqm} in that it omits all of the degeneracy conditions except $\sigma^0\omega_0=0$. Proposition \ref{rhommc} would not be true if we omitted those degeneracy conditions. 

Now, consider a $\top$-algebra $A$ in topological spaces, and a retraction $r:A \to X$, with section $s$. Given a homotopy $h$ from $sr$ to $\id_A$, Lada constructs a system $\{\omega_n: \top^{n+1}X\by |I|^n \to X\}$ by setting 
$$
\omega_n(a_1, \ldots a_n):= r\circ \vareps_A * h(a_1) * \vareps_A * h(a_2) * \ldots * h(a_n)*\vareps_A \circ s. 
$$ 

If we impose the additional conditions that $r\circ h(a)= r$, $h(a)\circ s = s$ and $h(a)\circ h(b) = h(\min(a,b))$, for all $a, b \in [0,1]$, then $\omega$ satisfies the higher  Maurer-Cartan relations of Definition \ref{mcdefqm}.

Moreover, a similar description holds for retracts $X$ of $(\top,\bot)$-bialgebras $A$, replacing $\vareps_A$ with $\eta_A \circ \vareps_A: \top A \to \bot A$. 

Another way to look at this is that $|\Xi|$ is a cofibrant resolution of  $\iota(|\Delta^0|)$ in $QM^*(\Top)$. An alternative (and possibly more natural) cofibrant replacement $\Phi$ of $\iota(|\Delta^0|)$ is given by $\Phi^0=|\Delta^0|$ and $\Phi^{n+1}=|I|^{n}$, with the same operations as $\Xi$, except that we replace
the map $\min: |I|\by|I| \to |I|$ with the map $(a,b)\mapsto ab$. Thus, for $E \in QM^*(\Top)$ and $\Sing: \Top \to \bS$, the space  $\mmc(\Sing E)$ is equivalent to the  subset of  $\prod_{n\ge 0} \HHom_{\Top}(|I|^n,E^{n+1})$ satisfying the conditions of Definition \ref{mcdefqm}, except that we replace $\min$ with multiplication. The procedure above then allows us to  construct a point $\omega$ of this space from a deformation retract, provided we modify the conditions above by requiring that the homotopy satisfies $h(a)\circ h(b) = h(ab)$.
\end{remarks}

\begin{definition}\label{ddefdefsub}
For $E \in QM^*(\bS)$ with $E^0$ a group (rather than just a monoid), there is an adjoint action of $E^0$ on $\mmc(E)$, given by $(g,\omega)\mapsto g^{-1}*\omega*g$. 
We then define $\ddel(E)$ to be the homotopy quotient (or Borel construction) $\ddel(E)= [\mmc(E)/^hE^0]= \mmc(E)\by_{E^0}WE^0$, for $W$ the universal cover of $BE^0=\bar{W}E^0$, as in \cite{sht} Ch.V.4.
\end{definition}

\subsection{Simplicial categories}

\begin{definition}\label{pi0C}
Given a simplicial category $\C$, recall from \cite{bergner} that the category $\pi_0\C$ is defined to have the same objects as $\C$, with morphisms 
$$
\Hom_{\pi_0\C}(x,y)=\pi_0\HHom_{\C}(x,y). 
$$
A morphism in $\HHom_{\C}(x,y)_0$ is said to be a homotopy equivalence if its image in $\pi_0\C$ is an isomorphism.
\end{definition}

\begin{lemma}\label{catmodel}
There is a model structure on the category $s\Cat$ of simplicial categories, in which a morphism $f:\C \to \D$ is 
\begin{enumerate}
\item[(W)] a  weak equivalence whenever
\begin{enumerate} 
\item[(W1)] for any objects $a_1$ and $a_2$ in $\C$, the map
$\HHom_{\C}(a_1, a_2)\to \HHom_{\cD}(fa_1, fa_2)$ 
is a weak equivalence of simplicial sets;
\item[(W2)] the induced functor $\pi_0f : \pi_0\C \to \pi_0\cD$ is an equivalence of
categories.
\end{enumerate}

\item[(F)] a  fibration whenever
\begin{enumerate} 
\item[(F1)] for any objects $a_1$ and $a_2$ in $\C$, the map
$\HHom_{\C}(a_1, a_2)\to \HHom_{\cD}(fa_1, fa_2)$ 
is a fibration of simplicial sets;
\item[(F2)] for any objects $a_1 \in\C$, $b \in \cD$, and homotopy equivalence $e :
fa_1 \to b$ in $\cD$, there is an object $a_2 \in \C$ and a homotopy equivalence
$d : a_1 \to a_2$ in $\C$ such that $fd = e$.
\end{enumerate}
\end{enumerate}
\end{lemma}
\begin{proof}
\cite{bergner} Theorem 1.1.
\end{proof}

\subsection{Simplicial quasi-descent data}

\begin{lemma}\label{sqdato}
For a fixed set $\cO$, there is a cofibrantly generated simplicial model category structure on $sQ\Dat_{\cO}$ for which a morphism $f:\cD \to \cD'$ is a fibration or a weak equivalence if and only if for all $a,b \in \cO$, the map
$$
f:\cHHom_{\cD}(a,b) \to \cHHom_{\cD'}(a,b)
$$
is a Reedy fibration or a levelwise weak equivalence in $\bS^{\Delta_{**}}$.
\end{lemma}
\begin{proof}
Applying \cite{Hirschhorn} Theorem 11.3.2 (as in  the proof of Lemma \ref{qmsmodel}) to the right adjoint functor $sQ\Dat_{\cO}\to \bS^{\Delta_{**}}$ given by $\cD \mapsto \prod_{a, b \in \cO} \cHHom_{\cD}(a,b)$ shows that this is a model structure. 

The simplicial structure comes from the simplicial structure on $\bS^{\Delta_{**}}$, so $\cHHom_{\cD^K}(a,b):= \cHHom_{\cD}(a,b)^K$.
\end{proof}

We now consider the whole category $sQ\Dat$, not just the subcategories on fixed objects.

\begin{lemma}\label{sqdatadjoints}
The functor $(-)^0: sQ\Dat\to s\Cat$ is both a left and a right adjoint. The functor $\cD \mapsto \prod_{a,b \in \Ob \cD} \cHHom_{\cD}(a,b)$ from $sQ\Dat \to \bS^{\Delta_{**}}$ is a right adjoint.
\end{lemma}
\begin{proof}
$(-)^0$ has left adjoint $(\id/ \emptyset)$ and right adjoint $\alg^*$, with the same formulae and reasoning as in Lemmas \ref{idemptyset} and \ref{alg*}. The left adjoint to $\cD \mapsto \prod_{a,b \in \Ob \cD} \cHHom_{\cE}(a,b)$ is the  functor $U: \bS^{\Delta_{**}} \to sQ\Dat$ given by sending  $X$ to the category with two objects $x,y$, and morphisms
$$
\cHHom(x,x)^n=\cHHom(y,y)^n:=\left\{\begin{matrix} 1 & n=0 \\ \emptyset & n>0, \end{matrix}\right. 
$$
 $\cHHom(x,y):= X$ and $\cHHom(y,x):=\emptyset$.
\end{proof}

\begin{proposition}\label{sqdatmodel}
There is a cofibrantly generated model structure on $sQ\Dat$  for which a morphism $f:\cD \to \cE$ is 
\begin{enumerate}

\item[$(W)$]  a weak equivalence if and only if 
\begin{enumerate}
\item[$(W1)$] for all $a,b \in \Ob \cD$, the map
$$
f:\cHHom_{\cD}(a,b) \to \cHHom_{\cE}(fa,fb)
$$
is a weak equivalence  in $\bS^{\Delta_{**}}$,  and

\item[$(W2)$] the morphism $\pi_0(f^0): \pi_0(\cD^0) \to \pi_0(\cE^0)$ is an equivalence of categories (for $\pi_0\C$ as in Definition \ref{pi0C});

\end{enumerate}

\item[$(F)$] a fibration  if and only if 
\begin{enumerate}
\item[$(F1)$] for all $a,b \in \Ob \cD$, the map
$$
f:\cHHom_{\cD}(a,b) \to \cHHom_{\cE}(a,b)
$$
is a Reedy fibration  in $\bS^{\Delta_{**}}$,  and

\item[$(F2)$]  for any objects $a_1 \in \cD, b\in \cE$ and a homotopy equivalence $e:fa_1 \to b$ in $\cE^0$, there exist an object $a_2 \in \cE$ and a homotopy equivalence
$d : a_1\to a_2$ in $\cD^0$ such that $f^0d=e$.
\end{enumerate}

\end{enumerate}
\end{proposition}
\begin{proof}
Note that these conditions are equivalent to saying that $f$ is a weak equivalence or fibration provided that both $f^0:\cD^0 \to \cE^0$ and the maps 
$\cHHom_{\cD}(a, b)\to \cHHom_{\cE}(fa, fb)$ are weak equivalences or fibrations. This follows because  the functor $\bS^{\Delta_{**}} \to \bS$ given by $X \mapsto X^0$ preserves both weak equivalences and fibrations (by definition of the Reedy model structure).

For $U$ as in the proof of Lemma \ref{sqdatadjoints}, define $(I1)$  to be the class consisting of the images under $U$ of the generating cofibrations from Proposition \ref{reedymodel}, and let $(I2)$ be the single morphism $\emptyset \to \oO$ from the category with no objects to the category with one object and no non-identity morphisms. Define $(I):= (I1) \cup (I2)$.

Define  $(J1)$  to be the class consisting of the images under $U$ of the generating trivial cofibrations from Proposition \ref{reedymodel}, and let $(J2)$ be the image of the class $(A2)$ from \cite{bergner} under the functor $(\id/\emptyset):s\Cat \to sQ\Dat$ of Lemma \ref{sqdatadjoints}. Define $(J):=(J1) \cup (J2)$.

For a class $C$ of morphisms, say that a morphism $f$ is $S$-injective if it has the left lifting property (LLP) with respect to S. 
From the adjoint property of the functor $U$, it follows that 
a morphism  $f: \cD \to \cE$ is  $(J1)$-injective, (resp. $(I1)$-injective) if and only if $f$ satisfies $(F1)$ (resp. $(F1)$ and $(W1)$). The  morphism $f$ is  $(I2)$-injective if and only if it is surjective on objects. By Lemma \ref{sqdatadjoints}, $f$ is $(J2)$-injective if and only if $f^0: \cD^0 \to \cE^0$ is $(A2)$-injective.

Since cofibrations in $\bS$, concentrated in degree $0$ in $\bS^{\Delta_{**}}$, become cofibrations in the model structure of Proposition \ref{reedymodel},  the images  under $(\id/\emptyset)$ of the classes $(C1)$, $(C2)$, $(A1)$ and $(A2)$ from \cite{bergner} lie in $(I1)$, $(I2)$, $(J1)$ and $(J2)$ respectively. It thus follows from \cite{bergner} Theorem 1.1 that for any  $J$-injective (resp. $I$-injective) morphism $f: \cD \to \cE$, the morphism  $f^0: \cD^0 \to \cE^0$ is a fibration (resp. a trivial fibration) in $s\Cat$. Looking at the LLP in $\bS^{\Delta_{**}}$, we then deduce that $J$-injectives  (resp. $I$-injectives) are precisely $(F)$  (resp. $(F) \cap (W)$ ) in $sQ\Dat$.

We now verify the conditions of \cite{hovey} Theorem 2.1.19. It is immediate that the class $(W)$ has the two out of three property and is closed under retracts. The domains of $(I2)$ and $(J2)$ are small, and similarly to the proof of \cite{bergner} Theorem 1.1, the smallness of the generating (trivial) cofibrations in $\bS^{\Delta_{**}}$ means that the domains of $(I1)$ and $(J1)$ are small relative to $(I1)$-cells and $(J1)$-cells, respectively. Thus the domains of  $(I)$ and $(J)$  are small relative to $(I)$-cells and $(J)$-cells, respectively. We have shown that the class of $(I)$-injectives is the intersection of $(W)$ with the class of $(J)$-injectives.

It remains only to show that all $(J)$-cells are in $(W)$ and are $(I)$-cofibrations. Since cofibrations in $\bS$, concentrated in degree $0$ in $\bS^{\Delta_{**}}$, become cofibrations in the model structure of Proposition \ref{reedymodel}, the functor $(\id/\emptyset)$ maps cofibrations in $s\Cat$ to $(I)$-cofibrations. Thus $(J2)$-cells are $(I)$-cofibrations, since $(\id/\emptyset)$ preserves all colimits. Likewise, $U$  preserves all colimits and maps cofibrations to $(I)$-cofibrations, so   $(J1)$-cells are $(I)$-cofibrations, and therefore $(J)$-cells are  $(I)$-cofibrations.
Since $(J) \subset (W)$, it follows immediately from the definitions and the corresponding properties in $s\Cat$ and $\bS^{\Delta_{**}}$ that all $(J)$-cells also lie in $(W)$. 
\end{proof}

Note that the functor $(-)^0: sQ\Dat\to s\Cat$ of Lemma \ref{sqdatadjoints} is then both left and right Quillen, while the functor $\cD \mapsto \prod_{a,b \in \Ob \cD} \cHHom_{\cD}(a,b)$ from $sQ\Dat \to \bS^{\Delta_{**}}$ is right Quillen.

\begin{definition}\label{dpdmodel}
We define model structures on $sQ\Dpd, sQ\Dpd_{\cO}$ by requiring that a morphism $f$ is a weak equivalence or a fibration whenever the underlying map in $sQ\Dat, sQ\Dat_{\cO}$ is so. \cite{Hirschhorn} Theorem 11.3.2 shows that these are indeed  model structures, since the forgetful functors preserve filtered colimits, and have  left adjoints (denoted by $\cD \mapsto \cD^{\gpd}$), given by formally inverting morphisms in $\cD^0$.
\end{definition}

\subsection{Simplicial bicomonoids}

\begin{lemma}
There is a cofibrantly generated  simplicial model structure on $QMM^*(\bS)$ for which a morphism $f$ is a fibration or weak equivalence whenever the underlying map in the Reedy model category $\bS^{\Delta_{**}\by \Delta_{**}}$ is so.
\end{lemma}
\begin{proof}
The proof of Lemma \ref{qmsmodel} carries over.
\end{proof}

\begin{lemma}\label{diagquillen}
The diagonal functor $\diag: QMM^*(\bS) \to QM^*(\bS)$ is right Quillen.
\end{lemma}
\begin{proof}
The Special Adjoint Functor Theorem (\cite{mac} Theorem V.8.2) implies that $\diag: QMM^*(\bS) \to QM^*(\bS)$ has a  left adjoint $\diag^*$. It therefore suffices to show that for any (trivial) fibration $f$   in  $\bS^{\Delta_{**}\by \Delta_{**}}$, the map $\diag f$ is a (trivial) fibration in $\bS^{\Delta_{**}}$. 

Now, let $\Theta_n, \pd\Theta_n \in \Set^{\Delta_{**}}$ be given by $\Hom_{\Set^{\Delta_{**}}}(\Theta_n, X)= X^n$ and $\Hom_{\Set^{\Delta_{**}}}(\pd\Theta_n, X)= M^nX$, and similarly let $\Theta_{ij}, \pd\Theta_{ij} \in \Set^{\Delta_{**}\by\Delta_{**}}$ be given by $\Hom_{\Set^{\Delta_{**}\by\Delta_{**}}}(\Theta_{ij}, X)= X^{i,j}$ and $\Hom_{\Set^{\Delta_{**}\by\Delta_{**}}}(\pd\Theta_{ij}, X)= M^{i,j}X$. Latching object arguments (adapting \cite{sht} Proposition VII.1.7) show that the maps $\pd\Theta_n \subset \Theta_n$ generate all  monomorphisms in $\Set^{\Delta_{**}}$, and likewise the maps $\pd\Theta_{ij} \subset \Theta_{ij}$  generate all  monomorphisms in $\Set^{\Delta_{**}\by\Delta_{**}}$.

The diagonal functor $\diag: \Set^{\Delta_{**}\by\Delta_{**}} \to \Set^{\Delta_{**}}$ has a left adjoint $\diag^{*}$, and we just observe that this preserves monomorphisms (much like the case of bisimplicial sets considered in \cite{sht} Theorem IV.3.15). Therefore the functor $\diag^{*}: \bS^{\Delta_{**}}\to \bS^{\Delta_{**}\by \Delta_{**}}$ preserves Reedy (trivial) cofibrations, so is left Quillen, making $\diag$ right Quillen.
\end{proof}

Note that we may regard $QM^*(\bS)$ and $QMM*(\bS)$ as being simplicial diagrams in $QM^*(\Set)$ and $QMM^*(\Set)$, respectively. This will allow us to extend many of the constructions of \S \ref{leinrev} to the simplicial case.

\subsection{Nerves}\label{snerves}
In Remark \ref{shob}, we saw how $\mmc$ enables us to define the space of s.h. bialgebras over a fixed object. However, as was first noted in  \cite{loop}, there is no satisfactory general way to define morphisms of s.h. algebras. The bar construction of \cite{loop} gives a definition when the monad is an operad, but does not generalise to s.h. bialgebras. Instead, we will now introduce a space of s.h. $\bI$-diagrams of s.h. bialgebras for any small category $\bI$, allowing us to mimic the nerve construction and thus to construct a simplicial space of s.h. bialgebras.

\subsubsection{$\cMC$}

\begin{definition}\label{pkdefn}
Given $K \in \bS$, define the functor $P_K : sQ\Dat_{K_0} \to QMM^*(\bS)$ by extending Definition \ref{pnset} to simplicial sets. Let $P_n:= P_{\Delta^n}$; as in Lemma \ref{pnleft2}, $P_n$ has a left adjoint $P_n^*$. 
\end{definition}

\begin{proposition}\label{pnquillen}
The functor $P_K: sQ\Dat_{K_0} \to QMM^*(\bS)$ is right Quillen.
\end{proposition}
\begin{proof}
Since $P_K$ is defined as a limit, it preserves arbitrary limits, so we just need to show that it preserves (trivial) fibrations. We may regard an object $X$ of  $\bS^{\Delta_{**}\by \Delta_{**}}$ as a $\Delta_{**}$ diagram in $\bS^{\Delta_{**}}$, by $i \mapsto X^{i, \bt}$. Denote the associated Reedy matching objects by $M^i_{\hor}X  \in \bS^{\Delta_{**}}$. Similarly, there is  a diagram $j \mapsto X^{ \bt, j}$, and we denote the associated matching objects by $M^j_{\ver}X \in \bS^{\Delta_{**}}$. Note that the Reedy matching objects in $\bS^{\Delta_{**}\by \Delta_{**}}$ are then given by 
$$
M^{ij}X = (M^i_{\hor}X)^j\by_{M^i_{\hor}M^j_{\ver}X}(M^j_{\ver}X)^i,
$$
as an immediate consequence of the characterisation of matching objects in \cite{hovey} Definition 5.2.2.

Now, for $\cD \in sQ\Dat_{K_0}$, the horizontal matching object $M^{i}_{\hor}P_K(\cD)$ in $\bS^{\Delta_{**}}$ is given by 
$$
(M^i_{\hor}P_K(\cD))^j= M^i(P_K(\cD)^j)=\prod_{x \in K_j} M^i\cD((\pd_0)^{j}x, (\pd_1)^{j}x).
$$
Next, observe that for $X,K \in \bS$ the object $S \in \bS^{\Delta_{**}}$ given by $S^n= X^{K_n}$, the matching object $M^nS$ is given by 
$$
M^nS=\{(f_0,f_1,\ldots,f_{n-1}) \in (X^{K_{n-1}})^{n}\,|\, \sigma^if_j=\sigma^{j-1}f_i \in X^{K_{n-2}}\, \forall i<j \}\cong  X^{L_nK}, 
$$  
where $L_nK$ is the $n$th simplicial latching object of $K$ (of \cite{sht} \S VII.1). 
Thus  the  vertical matching object
$M^{j}_{\ver}P_K(\cD)$ is given by
$$
(M^{j}_{\ver}P_K(\cD))^i= \prod_{x \in L_jK} \cD((\pd_0)^{j}x, (\pd_1)^{j}x)^i.
$$

Therefore,  since $L_jK \to K_j$ is always injective, 
$$
M^{ij}P_K(\cD) \cong (\prod_{x \in L_jK}\cD((\pd_0)^{j}x, (\pd_1)^{j}x)^i)  \by (\prod_{x \in K_j - L_jK} M^i\cD((\pd_0)^{j}x, (\pd_1)^{j}x)),
$$
which yields the required result. 

In fact, we may adapt this further to say that for any cofibration $i:J \into K$ in $\bS$  and any (trivial) fibration $\cD \to \cE$ in $sQ\Dat_{K_0}$, the map
$$
P_K(\cD) \to P_K(\cE)\by_{P_J(i_0^{-1}\cE)}P_J(i_0^{-1}\cD)
$$
is a (trivial) fibration.
\end{proof}

\begin{definition}\label{MCdef}
Define a functor $\cMC: sQ\Dat \to s\bS$ to bisimplicial sets  by 
$$
\cMC(\cD)_{(n)}:= \coprod_{f: [0,n] \to \Ob \cD} \mmc(\diag P_n(f^{-1}\cD)) \in \bS.
$$ 
\end{definition}

\begin{definition}
Given a simplicial set $X$, we define $X \in s\bS$ to be the constant space $X_{(n)} := X$ for all $n$. By contrast, we define $X^{\hor}$ by $X^{\hor}_{(n)}:=X_n$. 
\end{definition}

\begin{remark}\label{shdiagram}
Assume $\cD$ comes from a distributive monad-comonad pair $(\top, \bot)$ on a simplicial category $\cB$, as in Proposition \ref{enrichtopbot}. Lemma \ref{pknerve}  then allows us   to think of $\cMC(\cD)_{(n)}$ as being the space of s.h. $\on$-diagrams of s.h. $(\top, \bot)$-bialgebras over $\cB$, so $\cMC$ is a kind of nerve construction. More generally, for any small category $\bI$, we think of 
$ \coprod_{f: \Ob \bI \to \Ob \cD} \mmc(\diag P_{B\bI}(f^{-1}\cD))$ as the  space of s.h. $\bI$-diagrams of s.h. $(\top, \bot)$-bialgebras over $\cB$, noting that this is just $\HHom_{s\bS}((B\bI)^{\hor}, \cMC(\cD))$.
\end{remark}

\begin{lemma}\label{MCfib}
Given a (trivial) fibration $f:\cD \to \cE$ in $sQ\Dat$, the morphism 
$$
\cMC(\cD) \to \cMC(\cE)\by_{\cosk_0(\Ob \cE)^{\hor}}\cosk_0(\Ob \cD)^{\hor}
$$
is a (trivial) fibration in the Reedy category  $s\bS= \bS^{\Delta^{\op}}$, where $\cosk_0: \Set \to \bS$ denotes the $0$-coskeleton (\cite{sht} \S IV.3).
\end{lemma}
\begin{proof}
For any simplicial set $K$, $\Hom_{\bS}(K, \cosk_0S) = S^{K_0}$; since $(\pd\Delta^n)_0= (\Delta^n)_0= [0,n]$ (for $\pd\Delta^n \subset \Delta^n$ the boundary),  
the $n$th Reedy matching map of the morphism above is given by taking the coproduct over all $g:[0, n] \to \Ob \cD$ of
$$
\mmc(\diag P_n g^{-1}\cD) \to \mmc(\diag P_n g^{-1}(\Ob f)^{-1}\cE)\by_{\mmc(\diag P_{\pd \Delta^n} g^{-1}(\Ob f)^{-1}\cE)}\mmc(\diag P_{\pd \Delta^n} g^{-1}\cD).
$$

Assume that $f$ is a (trivial) fibration. By  the proof of Proposition \ref{pnquillen}, 
$$
P_n g^{-1}\cD \to  (P_n g^{-1}(\Ob f)^{-1}\cE)\by_{ P_{\pd \Delta^n} g^{-1}(\Ob f)^{-1}\cE) }( P_{\pd \Delta^n} g^{-1}\cD) 
$$
is a (trivial) fibration in $QMM^*(\bS)$, so  
Lemma \ref{diagquillen} shows that
$$
\diag P_n g^{-1}\cD\to (\diag P_n g^{-1}(\Ob f)^{-1}\cE)\by_{\diag P_{\pd \Delta^n} g^{-1}(\Ob f)^{-1}\cE) }(\diag P_{\pd \Delta^n} g^{-1}\cD) 
$$
is  a (trivial) fibration in $QM^*(\bS)$,
 so Corollary \ref{mcquillen} shows that the map above is a (trivial) fibration whenever $f$ is so.
\end{proof}

Note that considering level $0$ shows  that $\cMC: sQ\Dat\to s\bS$ is not right Quillen, since a fibration only maps to a fibration when it is surjective on objects, and a trivial fibration only maps to a trivial fibration when it is an isomorphism on objects.

\subsubsection{$\ddel$}

\begin{definition}
Let $s\Gpd$ be the category of simplicial groupoids, i.e. the full subcategory of $(\Gpd)^{\Delta^{\op}}$ consisting of those $\Gamma$ for which the simplicial set $\Ob \Gamma$ of objects is constant. As in \cite{sht} \S V.7, this has a model structure in which a morphism is a weak equivalence or fibration whenever the corresponding morphism in $s\Cat$ is so (for the model structure of Lemma \ref{catmodel}), although the description simplifies considerably, since all morphisms  in $\Gamma$ (and in particular homotopy equivalences) are isomorphisms.
\end{definition}

\begin{definition}
Given $\Gamma \in s\Gpd$, define $\bS(\Gamma)$ to be the category of simplicial $\Gamma$-representations. An object $X \in \bS(\Gamma)$ consists of $X(a) \in \bS$ for all objects $a$ of $\Gamma$, and distributive morphisms
$$
\Gamma(a,b) \by X(b) \to X(a)
$$
in $\bS$.
\end{definition}

\begin{lemma}
For $\Gamma \in sQ\Dpd$, there is a natural $\Gamma^0$-representation in $\bS$, given by mapping $a \in \Ob \Gamma$ to $\mmc(\Gamma(a,a))$. Denote this representation by $\mmc(\Gamma) \in \bS(\Gamma^0)$.
\end{lemma}
\begin{proof}
We just need to define an associative action  $\mmc(a)\by\Gamma^0(a,b)\to \mmc(b)$. As in Definition \ref{ddefdefsub}, the adjoint action  $(\omega,g)\mapsto g^{-1}*\omega*g$ suffices.
\end{proof}

\begin{definition}
Given a simplicial groupoid $\Gamma$, define the $\Gamma$-representation $W\Gamma$ by 
$$
(W\Gamma)_n(a) := \coprod_{x_0, x_1, \ldots , x_n}\Gamma(a,x_n)_n\by\Gamma(x_{n}, x_{n-1})_{n-1}\by \ldots \by \Gamma(x_1, x_0)_0,
$$
for $a \in \Ob \Gamma$. As in \cite{sht} \S V.4 (which considered only simplicial groups),  this has  operations:
\begin{eqnarray*}
\pd_i(v_n,v_{n-1}, \ldots,v_0)&=& \left\{ \begin{matrix} (\pd_iv_n,\pd_{i-1}v_{n-1},\ldots, (\pd_0v_{n-i})v_{n-i-1}, v_{n-i-2}, \ldots, v_0) & i<n,\\ (\pd_nv_n,\pd_{n-1}v_{n-1}, \ldots, \pd_1v_1) & i=n, \end{matrix} \right.\\
\sigma_i(v_n,v_{n-1}, \ldots , v_0)&=& (\sigma_iv_n,\sigma_{i-1}v_{n-1},\ldots, \sigma_0v_{n-i}, 1 ,  v_{n-i-1}, \ldots, v_0),
\end{eqnarray*}
For $h \in \Gamma(a,b)$, the action is given by 
$$
h(v_n, v_{n-1}, \ldots , v_0)=(hv_n,v_{n-1}, \ldots ,v_0).
$$
\end{definition}

\begin{definition}
Given $\Gamma \in s\Gpd$, define $\ho\LLim_{\Gamma}: \bS(\Gamma) \to (\bS \da \bar{W}\Gamma)$ by $X \mapsto X\by_{\Gamma}W\Gamma$, where $\bar{W}\Gamma:= \bt\by_{\Gamma}W\Gamma$ is a model for the classifying space of $\Gamma$ (\cite{sht} \S V.7).
\end{definition}

\begin{lemma}\label{holimquillen}
The functor $\ho\LLim_{\Gamma}: \bS(\Gamma) \to (\bS \da \bar{W}\Gamma)$ is right Quillen, where fibrations and weak equivalences in $\bS(\Gamma)$ are defined objectwise.
\end{lemma}
\begin{proof}
The proof of \cite{sht} Lemma  VI.4.6, which takes the case when $\Gamma$ is a discrete groupoid, carries over to this generality. The left adjoint is given by $X \mapsto   X\by_{\bar{W}\Gamma}W\Gamma$.
\end{proof}

\begin{definition}
Define a functor $\ddel: sQ\Dpd \to \bS$ by 
$$
\ddel(\Gamma):= \ho\LLim_{\Gamma^0}\mmc(\Gamma),
$$
making use of the forgetful functor $\bS \da  \bar{W}\Gamma \to \bS$.
Note that if $\Gamma$ has one object, then it may be regarded as an object of $QM^*(\bS)$, and this definition is consistent with Definition \ref{ddefdefsub} in this case. 
\end{definition}

\begin{proposition}\label{ddefquillen}
The functor $\ddel: sQ\Dpd \to \bS$ is right Quillen.
\end{proposition}
\begin{proof}
Since $\ddel$  clearly preserves limits, we need only show that it preserves (trivial) fibrations. Given a  (trivial) fibration $f:\cD \to \cE$ in $sQ\Dpd$, set $\cF:= \cE\by_{\alg^*\cE^0}\alg^*\cD^0$, and observe that $f$ factors as the composition of the (trivial) fibrations $g: \cD \to \cF$, $h: \cF \to \cE$. 

Now, $\mmc(\cD) \to \mmc(\cF)$ is a morphism in $\bS(\cD^0)$ which is a (trivial) fibration by Corollary \ref{mcquillen}. Lemma \ref{holimquillen} then implies that $\ddel(\cD) \to \ddel(\cF)$ is a (trivial) fibration in $\bS$, since $\Ob(\cD) = \Ob(cF)$. The morphism $\ddel(\cF) \to \ddel(\cE)$ is a pullback of $\ddel(\alg^*\cD^0) \to \ddel(\alg^*\cE^0)$, so it remains only to show that the latter is a (trivial) fibration.  

Given $E \in QM^*(\bS)$ with $E^n=E^0$ for all $n$, studying the degeneracy operations shows that $\mmc(E) =\{1\}$. Therefore $\mmc(\alg^*\cD^0)$ is the constant $\cD^0$-representation on the one-point set $\bt$, so 
$$
\ddel(\alg^*\cD^0)= \bt \by_{\cD^0}W\cD^0=  \bar{W}\cD^0,
$$ 
and similarly for $\ddel(\alg^*\cE^0)=\bar{W}\cE^0 $. The morphism $\cD^0 \to \cE^0$ is a (trivial) fibration in $s\Gpd$, so $\bar{W}\cD^0 \to \bar{W}\cE^0$ is a (trivial) fibration in $\bS$, by  \cite{sht} Theorem V.7.8.
\end{proof}
 
\begin{definition}\label{hKdef}
Given $K \in \bS$ and $\cD \in sQ\Dat$, define $h(K,\cD)\in sQ\Dat$ by $\Ob(h(K, \cD))= \Hom_{\Set}(K_0, \Ob\cD)$ and
$$
h(K,\cD)(a,b)^n:=\prod_{x \in K_n}\C(a((\pd_0)^{n}x), b((\pd_1)^{n}x)), 
$$
and note that $\cMC(\cD)_n= \mmc(h(\Delta^n, \cD))$.
\end{definition}

\begin{definition}\label{Defdef}
Define $\cDel: sQ\Dpd \to s\bS$ by $\cDel(\Gamma)_{(n)}:= \ddel(h(\Delta^n, \Gamma))$.
\end{definition}

\begin{corollary}\label{Defquillen}
The functor $\cDel: sQ\Dpd \to s\bS$ is right Quillen.
\end{corollary}
\begin{proof}
This  just combines Proposition \ref{ddefquillen} with the observation that for any (trivial) fibration $f:\cD \to \cE$ in  $sQ\Dpd$, the morphism 
$$
h(\Delta^n, \cD)\to h(\Delta^n,\cE)\by_{h(\pd\Delta^n, \cE)}h(\pd\Delta^n, \cD)
$$
is a (trivial) fibration in $sQ\Dat  $  for all $n \ge 0$, which follows by combining Lemma \ref{diagquillen} with the proof of Proposition \ref{pnquillen}.
\end{proof}

\section{Abelian groups and cohomology}\label{absn}
In this section, we will investigate quasi-comonoids in abelian groups and in groupoids. The main motivation for this is that we can detect whether a  simplicial set $X$ is contractible  just by looking at $\pi_fX$ and $\H^*(X, \Z)$, and we will now develop the corresponding notions for $QM^*(\bS)$.

\subsection{Cosimplicial abelian groups}
\begin{lemma}\label{abqm}
There is an equivalence between the category $QM^*(\Ab, \by)$ of quasi-comonoids in $(\Ab, \by)$, and the category $c\Ab$ of cosimplicial complexes of abelian groups.
\end{lemma}
\begin{proof}
Take $A \in QM^*(\Ab, \by)$. 
The operations $\pd^i, \sigma^i$ on $A$ are necessarily $\Z$-linear, and we enhance this to a cosimplicial structure by setting $\pd^0 a = 0_1*a, \, \pd^{m+1}a= a*0_1$, for $a \in A^m$ and $0_1$ the group identity  in $A^1$. To see that this satisfies the  cosimplicial axioms, note that the properties of $*$ give that 
\begin{eqnarray*}
\pd^{i+1}\pd^0a &=& \pd^0\pd^ia \quad \forall 1 \le i \le m\\
\pd^{i}\pd^{m+1}a &=& \pd^{m+2}\pd^ia \quad \forall 1 \le i \le m\\
\sigma^{i+1}\pd^0a&=& \pd^0\sigma^ia \quad \forall 0 \le i < m\\
\sigma^{i}\pd^{m+1}a&=& \pd^{m}\sigma^ia \quad \forall 0 \le i < m.
\end{eqnarray*}
We also have $\pd^0\pd^{m+1} a= 0_1*a*0_1= \pd^{m+2}\pd^0a$, so it only remains to show that $\sigma^0\pd^0=\id$ and $\sigma^m\pd^{m+1}=\id$, and that $\pd^0\pd^0= \pd^1\pd^0$ and $\pd^{m+1}\pd^{m+1}=\pd^{m+2}\pd^{m+1}$. The first two conditions follow because $\sigma^00_1=0_0$, which is the identity for $*$, and the second two follow because $0_1*0_1=0_2= \pd^10_1$. 

Since $A^m \by A^n \xra{*} A^{m+n}$ is linear, for   $a, a' \in A^m$,  $b,b' \in A^n$ we then have
$$
a*b'+ a'*b= (a+a')*(b'+b),
$$  
so setting $a'=0_m, b'=0_n$ gives $a*b = a*0_n +0_m*b$, and $0_n= 0_1^{*n}$, so the product is necessarily the Alexander-Whitney cup product
$$
a*b= (\pd^{m+1})^na + (\pd^0)^m b,
$$
which is uniquely determined by the cosimplicial structure.
\end{proof}

\begin{definition}
Let $s\Ab$ be the category of simplicial abelian groups, and $cs\Ab$ be the category of cosimplicial simplicial abelian groups.
\end{definition}

\begin{definition}\label{cotdef}
Denote the left adjoint to the inclusion functor $cs\Ab \to QM^*(\bS)$ by $\cot$. This is  left  Quillen, and we denote the associated left-derived functor on homotopy categories by $\oL \cot$.
\end{definition}

\begin{lemma}\label{modelqmsab}
There is a cofibrantly generated simplicial model structure on $QM^*(s\Ab,\by)$ in which a morphism is a fibration or a weak equivalence whenever the underlying map in $QM^*(\bS)$ is so. 
\end{lemma} 
\begin{proof}
We may  apply \cite{Hirschhorn} Theorem 11.3.2 to the forgetful functor $QM^*(s\Ab,\by)\to  QM^*(\bS)$.  This functor satisfies the Special Adjoint Functor Theorem (\cite{mac} Theorem V.8.2), so  has a left adjoint (analogous to the free module generated by a set). It also preserves filtered direct limits, so admits the small object argument where necessary.  
\end{proof}

\begin{lemma}
There is an equivalence
$$
QM^*(s\Ab,\by) \simeq cs\Ab
$$
of model categories, where $cs\Ab$ is given the Reedy model structure for cosimplicial objects in $s\Ab$ (with its standard model structure). 
\end{lemma}
\begin{proof}
 Lemma \ref{abqm} gives the  equivalence of categories, by passing to simplicial diagrams. Now, $f$ is a weak equivalence in $QM^*(s\Ab,\by)$  whenever each $f^n$ is  a weak equivalence, and a fibration whenever $f$ is a  Reedy  fibration in $\bS^{\Delta_{**}}$. Since the matching objects for $\Delta$ and $\Delta_{**}$ are the same, this means that the  model structure of Lemma \ref{modelqmsab} is just the
Reedy model structure on $(s\Ab)^{\Delta}$, as they have the same weak equivalences and fibrations.
\end{proof}

\begin{definition}\label{doldkan}
Let $N^s$ denote simplicial normalisation from simplicial abelian groups to non-negatively graded chain complexes, given by  $N^s(V)_n:=\bigcap_{i>0}\ker (\pd_i: V_n \to V_{n-1})$, with differential $\pd_0$.
Let $N_c$ denote cosimplicial conormalisation from cosimplicial abelian groups to  non-negatively graded cochain complexes, given by $N_c(V)^n :=\bigcap_{i\ge 0}\ker (\sigma^i: V^n \to V^{n-1})$, with differential $\sum_i (-1)^i \pd^i$. By the Dold-Kan correspondence (\cite{W} Theorem 8.4.1, passing to opposite categories and using \cite{W} Lemma 8.3.7 in the cosimplicial case), these functors are both equivalences;  
 let $D_c$ be the  cosimplicial denormalisation functor, inverse to $N_c$.
\end{definition}

\begin{definition}
Set $I=\Delta^1$, and for $n \ge 0$, let 
$$
\gimel^n:= (\{1\}\by I^{n-1})\cup \bigcup_{j>0}(I^j\by\{0,1\}\by I^{n-1-j}) \subset I^n \in \bS;
$$ 
for $n\ge 2$ this is given 
by removing the interior of $0\by I^{n-1}$ from  the boundary $\pd I^n$, while $\gimel^1=\{1\}$ and $\gimel^0=\emptyset$.

Let
$$
\Z(I^n/\gimel^n):=\Z(I^n) /\Z(\gimel^n) \in s\Ab,
$$
where $\Z(S)$ is the free $\Z$-module generated by the set $S$, 
and
let  $\delta$ be  the canonical map $ \Z(I^{n-1}/\gimel^{n-1})\to \Z(I^n/\gimel^n) $  arising from the map $I^{n-1} \to I^n$ given by $x \mapsto (0,x)$.

For any simplicial abelian group $W$, write
$$
W^{I^n/\gimel^n}:=\ker (W^{I^n} \to W^{\gimel^n}) = \HHom_{s\Ab}(\Z(I^n/\gimel^n), W)
$$
and
let  $\delta$ be  the canonical map $W^{I^n/\gimel^n} \to W^{I^{n-1}/\gimel^{n-1}}$ dual to the map $\delta$ above.
\end{definition}

\begin{proposition}\label{mctansub} Given an abelian group object $E$  in $QM^*(\bS)$,  corresponding under Lemma \ref{abqm} to the cosimplicial simplicial abelian group $\CC(E)$, there is an isomorphism
$$
\mmc(E) \cong \{\eta \in \prod_{n=0}^{\infty} N_c^{n+1} \CC(E)^{I^{n}/\gimel^{n}}: d_c\eta_{n-1}=\delta \eta_{n}\}.
$$
\end{proposition}
\begin{proof} 
As in Lemma \ref{mcqsub}, write $\mmc(E)= \Lim_n \mmc(E)^n$, and
assume that we are given an element 
$$
(\omega_0, \ldots,\omega_{n-1}) \in \mmc(E)^{n-1}.
$$
The proof of Lemma \ref{mcqsub}  then gives rise to the data  
$$
\beta_{n-1} \in M^{n+1}\CC(E)^{I^n}, \quad \alpha_{n-1} \in \CC^{n+1}(E)^{\pd I^n}
$$ 
(in the notation of Definition \ref{sdcmatchsub}).
By Lemma \ref{mcqsub}, the fibre of $\mmc(E)^{n}\to \mmc(E)^{n-1}$ over $(\omega_0, \ldots,\omega_{n-1})$ is given by $\omega_n \in \CC^{n+1}(E)^{I^n}$ compatibly lifting  $\alpha_{n-1}, \beta_{n-1}$ in the following diagram: 
$$
\xymatrix{\CC^{n+1}(E)^{I^n}   \ar[r] \ar[d]& \CC^{n+1}(E)^{\pd I^n}  \ar[d]\\
 M^{n+1}\CC(E)^{I^n}\ar[r]  &   M^{n+1}\CC(E)^{\pd I^n}. }
$$
 
For any abelian cosimplicial abelian group $C^{\bullet}$, dualising \cite{W} Lemma 8.3.7 gives  a decomposition of the associated cochain complex as $C^n=N_c^n(C) \oplus (M')^n(C)$, where $N_c^n(C)=\cap_{i=0}^{n-1}\ker \sigma^i$, and $(M')^n(C)=\sum_{i=1}^n\pd^i C^{n-1}$, so the commutative diagram becomes
$$
\xymatrix{(N_c^{n+1}\CC(E)^{I^n})\oplus ((M')^{n+1}\CC(E)^{I^n})   \ar[r] \ar[d]& (N_c^{n+1}\CC(E)^{\pd I^n})\oplus ((M')^{n+1}\CC(E)^{\pd I^n})  \ar[d]\\
 M^{n+1}\CC(E)^{I^n}\ar[r]  &   M^{n+1}\CC(E)^{\pd I^n}. }
$$
 Moreover, $\underline{\sigma}:(M')^{n}C \to M^{n}C$ is an isomorphism, and we will denote the inverse by $a \mapsto \tilde{a}$. Thus the problem of constructing $\omega_n$
reduces to seeking
an element  $\eta_{n} =\omega_{n}-\tilde{\beta}_{n-1} \in N_c^{n+1}\CC(E)^{I^n}$ lifting $\pr_N(\alpha_{n-1}) \in  N_c^{n+1}\CC(E)^{\pd I^n}$, where $\pr_N:C \to N_c(C) $ is the projection given by annihilating $(M')^C$ .

Now, $\alpha_{n-1}$ is defined by
\begin{eqnarray*}
\alpha_{n-1}(t_1,\ldots, t_{i-1},0,t_{i+1},\ldots, t_n)&=&(\pd^{i+1})^{n-i+1}\omega_{i-1}(t_1,\ldots, t_{i-1})+(\pd^0)^i\omega_{n-i}(t_{i+1},\ldots, t_n);\\ 
\alpha_{n-1}(t_1,\ldots, t_{i-1},1,t_{i+1},\ldots, t_n)&=&\pd^i\omega_{n-1}(t_1,\ldots,t_{i-1},t_{i+1},\ldots,t_n).
\end{eqnarray*}
Therefore
\begin{eqnarray*}
\pr_N\alpha_{n-1}(0,t_2,\ldots, t_n)&=&\pr_N\pd^0\eta_{n-1}(t_2,\ldots, t_n);\\ 
\pr_N\alpha_{n-1}(t_1,\ldots, t_{i-1},0,t_{i+1},\ldots, t_n)&=&0\quad \text{for } i>1;\\ 
\pr_N\alpha_{n-1}(t_1,\ldots, t_{i-1},1,t_{i+1},\ldots, t_n)&=&0,
\end{eqnarray*}
since all other terms lie in $(M')^{n+1}$ (the span of  $\{\pd^i\,:\,i>0\}$). 

Since $\pr_N\pd^0= \pr_N d_c$, it follows from   \cite{W} Lemma 8.3.7 that  on $N_c^n$, 
$$
\pr_N\pd^0=d_c= \sum_{i=0}^{n+1}(-1)^i \pd^i.
$$
This implies that 
$$
\eta_n \in  N_c^{n+1}\CC(E)^{I^n/ \gimel^n}, 
$$
and the condition for $\eta_n$ to lift $\pr_N(\alpha_{n-1})$  is precisely that  
$\delta(\eta_{n})=d_c\eta_{n-1}$.
\end{proof} 

\begin{corollary}\label{cot2} A representative for $\oL\cot (\iota\bt) $ is given by
 $\iota(\Z) \cong  \cE D_c(\Z^{[-1]})$, for $\cE: cs\Ab \to  QM^*(s\Ab,\by)$ as in Lemma \ref{abqm}.
\end{corollary}
\begin{proof}
First observe that, for  $\Xi$ from Definition \ref{Xidef}, $\oL \cot (\iota\bt)\simeq \oL \cot(\Xi)$, since  $\Xi \to \iota\bt$ is a cofibrant replacement.
 
 Since $\Xi$ represents $\mmc$,  for any simplicial cochain complex $V$ (in non-negative cochain degrees) we have
$$
\Hom(N_c\CC(\cot \Xi), V) \cong \mc(\cE(D_cV)),
$$
where $\Hom$ is taken in the category of simplicial cochain complexes.
Since $\CC$ is inverse to $\cE$, and $N_c$ is inverse to $D_c$,  Proposition \ref{mctansub} can then be rephrased to say that
$$
N_c^n\CC(\cot(\Xi))\cong \left\{ \begin{matrix} \Z (I^{n-1}/\gimel^{n-1}) & n \ge 1\\ 0 & n=0, \end{matrix} \right.
$$
with differential $d_c= \delta$.

Thus  the  bicomplex $N^sN_c\CC(\cot \Xi)$ is  weakly equivalent (in the Reedy model category of cochain diagrams in chain complexes) to the  bicomplex $\Z^{[-1]}$, consisting of $\Z$ concentrated in cochain degree $1$, chain degree $0$. This means that $\cot \Xi$ is weakly equivalent to $\cE D_c(\Z^{[-1]})$ (with constant simplicial structure), but this is isomorphic to $\iota(\Z)$ (having $n$ copies of $\Z$ in level $n$).
\end{proof}

\begin{definition}
Given a cochain complex $V$, denote the brutal truncation in degrees $\ge n$ by $\sigma^{\ge n} V$, so
$$
(\sigma^{\ge n} V)^i= \left\{ \begin{matrix} V^i & i \ge n \\ 0 & i<n.  \end{matrix} \right. 
$$
\end{definition}

\begin{definition}\label{totprod}
Define the total complex  functor $\Tot^{\Pi}$ from chain cochain complexes (i.e. bicomplexes) to chain complexes by
$$
(\Tot^{\Pi} V)_n := \prod_{a-b=n} V^b_a,
$$
with differential $d:=d^s +(-1)^a d_c$ on $V^b_a$.
\end{definition}

\begin{proposition}\label{cotcoho}
For $A \in cs\Ab$,
$$
\pi_n\mmc(A)  \cong \H_{n-1}(\Tot^{\Pi}\sigma^{\ge 1}N^sN_cA).
$$
\end{proposition}
\begin{proof}
First, note that
$$
\mmc(A) = \HHom_{QM^*(\bS)}(\Xi, A),
$$ 
so the Dold-Kan correspondences give 
$$
\mmc(A)= 
 \HHom_{cs\Ab}(\cot \Xi, A)
\simeq \HHom_{DG^{\ge 0}dg_{\ge 0}\Ab}(N^sN_c\CC(\cot \Xi), N^sN_cA),
$$
where $DG^{\ge 0}dg_{\ge 0}\Ab$ is the Reedy category of non-negatively graded cochain complexes of non-negatively graded chain complexes.

Now,  Corollary \ref{cot2} implies that $N^sN_c\CC(\cot \Xi)$ is a cofibrant replacement for $\Z^{[-1]}$ in $DG^{\ge 0}dg_{\ge 0}\Ab$, so 
$$
\mmc(A) \simeq \oR\HHom_{DG^{\ge 0}dg_{\ge 0}\Ab}(\Z^{[-1]} ,N^sN_cA)
$$ 

A simpler cofibrant replacement is the object $L$ given by
$$
L^n_i= \left\{\begin{matrix} \Z & n=i+1, i+2,\\ 0 & \text{otherwise}, \end{matrix} \right.
$$ 
with chain and cochain differentials the identity whenever possible. Thus
$$
\mmc(A) \simeq \HHom_{DG^{\ge 0}dg_{\ge 0}\Ab}(L, N^sN_cA).
$$ 

Now, a map $L \to B$ is determined by the images $b_n$ of the elements $1 \in L_{n-1}^{n}$, subject to the conditions that $d_cb_n= d^s b_{n+1}$. Thus  
$$
\Hom_{DG^{\ge 0}dg_{\ge 0}\Ab}(L, B) \cong   \z_{-1}(\Tot^{\Pi}\sigma^{\ge 1}B), 
$$
where $\z_i(V)= \ker(d:V_i \to V_{i-1})$,
and  the description of $\pi_n\mmc(A)$ follows.
\end{proof}

\begin{proposition}\label{cot3}
For $A \in cs\Ab$,
$$
\pi_n\ddel_A  \cong \H_{n-1}(\Tot^{\Pi}N^sN_cA). 
$$
\end{proposition}
\begin{proof}
We need to understand the adjoint action of $A^0$ on $\mmc(A)$. For $g \in A^0$ and $\omega \in \mmc(A)$, we  have $g^{-1}*\omega*g = g^{-1}*0*g + \omega$,  so the problem reduces to understanding the morphism $D:A^0 \to \mmc(A)$ given by $D(g)= g^{-1}*0*g$. 

If we now consider the simplicial normalisation of $[X/^hG]$, for $X$ and $G$ abelian and $[-/^h-]$ as in Definition \ref{ddefdefsub}, we see that 
\begin{eqnarray*}
 N^s_n[X/^hG] &\cong&  N^s_nX \oplus N^s_{n-1}G\\
(x,g,d^sg,0,0,\ldots,0)  &\mapsfrom&  (x,g).
\end{eqnarray*}
The corresponding differential is  given by $d^s(x,g)=(d^sx+g\cdot 0, dg)$, so $N^s[X/^hG]$ is isomorphic to the mapping cone of the morphism $N^sG \xra{\cdot 0} N^sX$.

If $X= \mmc(A)$ and $G=A^0$, then the gauge action on $0$ is given by $g\cdot 0= D g$,  so $N^s\ddel_A$ is isomorphic to the cone complex of  the morphism 
$$
N^sA^0 \xra{N^sD} N^s\mmc(A).
$$

Now, 
$$
\HHom_{cs\Ab}(\iota(\Z),A) \cong \z^1N_cA:= \ker(d_c: N_c^1A \to N_c^2A),
$$ 
with the map $\cot(\Xi) \to \iota(\Z)$ corresponding to the inclusion $f:\z^1N_cA \into \mmc(A)$ given by $f(a)_n = (\pd^1)^n(a) \in A^{n+1}\subset (A^{n+1})^{I^n}$. The key observation is that this subset is closed under the adjoint action, so in particular $D:  A^0 \to \mmc(A)$ factors through $\z^1N_cA$, via the map $d_c:A^0 \to \z^1N_cA$.

The map $L \to \Z^{[-1]}$ corresponds to the natural inclusion $(\z^1N_cA)_n \into \z_{-1}(\Tot^{\Pi}\sigma^{\ge 1}N_cA^{\Delta^n} )$. The proof of Proposition \ref{cotcoho} gives an  equivalence 
$$
N^s\mmc(A) \simeq \tau_{\ge 0}(\Tot^{\Pi}(\sigma^{\ge 1}N_cA)[-1])
$$ 
of chain complexes, where $\tau_{\ge 0}$ is good truncation in non-negative degrees  (so 
$$
\tau_{\ge 0}(V[-1])_n = \left \{ \begin{matrix}V_{n-1} & n >0\\  \z_{-1}V & n=0\\ 0 & n<0). \end{matrix} \right.
$$

Since this equivalence preserves the image of $N^s\HHom_{cs\Ab}(\iota(\Z),A)$, it also preserves the image of $A^0$, giving an equivalence between the cone complexes of 
$$
N^sA^0 \xra{N^sD} N^s\mmc(A)\quad \text{ and } \quad N^sA^0 \xra{d_c} \tau_{\ge 0}(\Tot^{\Pi}(\sigma^{\ge 1}N_cA)[-1]).
$$
The latter cone complex  is just 
$$
\tau_{\ge 0}(\Tot^{\Pi}(N^sN_cA)[-1]),
$$ 
which has the required cohomology.
\end{proof}

\subsection{Cohomology}

\begin{definition}
Given $E \in QM^*(\bS)$ and  $A \in c\Ab$, define cohomology groups of $E$ with coefficients in $A$ by
$$
\H^i(E, A):= \Hom_{\Ho(QM^*(\bS))}(E, N_s^{-1}A_{[-i]}),
$$
where $N_s^{-1}A_{[-i]}$ is the simplicial abelian group whose simplicial normalisation has $A$ concentrated in chain degree $i$.
\end{definition}

\begin{proposition}\label{btcoho} 
For $A \in c\Ab$,
$$
\H^i(\iota\bt, A) \cong \left\{ \begin{matrix} \H^{i+1}(N_cA) & i>0 \\ \z^1N_cA & i=0. \end{matrix} \right.
$$
\end{proposition}
\begin{proof}
By Proposition \ref{rhommc}, 
$
\H^i(\iota\bt, A)\cong \pi_0\mmc(N_s^{-1}A_{[-i]}),
$
which by Proposition \ref{cotcoho} is isomorphic to 
$
\H_{-1}(\Tot^{\Pi}\sigma^{\ge 1}N_cA_{[-i]}).
$
Since $\Tot^{\Pi}\sigma^{\ge 1}N_cA_{[-i]})_n= (\sigma^{\ge 1}N_cA)^{i-n}$, this is just
 $\H^{i+1}(\sigma^{\ge 1}N_cA)$,
as required.
\end{proof}

\subsection{Groupoids  }

Given $\Gamma \in QM^*(\Gpd, \by)$ (or even in $QM^*(\Cat, \by$), we have $B\Gamma \in QM^*(\bS)$, for $B:\Gpd \to \bS$ the nerve functor, and we now seek to describe the set $\mc(B\Gamma)$ (and hence the space $\mmc(B\Gamma)$).

\begin{proposition}\label{gpdcalc}
For $\C \in QM^*(\Cat, \by)$, the set $\mc(B\C)$ is isomorphic to the set of pairs $(x, a)$, for $x \in \Ob \C^1, a \in \C^2(x*x, \pd^1x)$ satisfying the following conditions:
\begin{eqnarray*}
\sigma^0x&=& 1\\
\sigma^0a=\sigma^1a &=& \id_x \in \C^1(x,x)\\
(\pd^2a)\circ (x*a)&=&(\pd^1a)\circ (a*x) \in \C^3(x*x*x, \pd^2\pd^1x),
\end{eqnarray*}
where $a*x:= a*\id_x$ and $x*a= \id_x*a$, for $\id_x \in \C^1(x,x)$ the identity morphism.
\end{proposition}
\begin{proof}
First, observe that 
$$
\mc(B\C) = \Hom_{QM^*(\bS)}(\Xi, B\C)\cong \Hom_{QM^*(\Cat)}(\tau_1\Xi, \C),
$$
where $\tau_1\Xi$ is the fundamental category of $\Xi$ (in the sense of \cite{joyaltierney} \S 1).

It is  therefore equivalent to show that  $\tau_1\Xi$  is the quasi-comonoid in categories generated by an object $\xi \in \Xi_0^1$ and a morphism $\alpha \in \tau_1(\Xi^2)(\xi*\xi, \pd^1\xi)$ satisfying the conditions for $(x,a)$ above.  

Note that $\Xi_0$ is the free quasi-comonoid (in sets) generated by the unique element $\xi \in \Xi_0^1$, subject to the condition that $\sigma^0\xi =1$. Recall that $\Xi^n_0= [0,1]^{n-1}$ for $n \ge 1$. 
We may then describe $\tau_1\Xi^n$ as the category associated to the poset $[0,1]^{n-1}$. Thus $\tau_1\Xi^2$ has objects $\xi*\xi$ and $\pd^1\xi$ (corresponding to $0$ and $1$ in $[0,1]$ respectively), with a unique isomorphism $\alpha: \xi*\xi \to\pd^1\xi$. Since there is a unique morphism from $(0,0)$ to $(1,1)$ in $\tau_1\Xi^3$, the morphisms 
$$
(\pd^2\alpha) \circ (\xi*\alpha), (\pd^1\alpha)\circ (\alpha*\xi) \in \C^3(\xi*\xi*\xi, \pd^2\pd^1\xi)
$$
must be equal.

It therefore only remains to show that $\tau_1\Xi $ is isomorphic to the quasi-comonoid $\cD$ defined to have  objects $\Ob \Xi$, with morphisms freely generated by $\alpha \in \cD(\xi*\xi, \pd^1\xi)$, subject to the condition $(\pd^2\alpha) \circ (\xi*\alpha)= (\pd^1\alpha)\circ (\alpha*\xi)$. Since the condition is satisfied by $\tau_1\Xi$, there is a natural map $\cD \to \tau_1\Xi$ in $QM^*(\Cat)$.

Now, for $u \in \cD^m(x,y)$ and $v \in \cD^n(x',y')$, the fact that $\cD^m \by \cD^n \xra{*} \cD^{m+n}$ is a  functor  implies that 
$$
 (\id_y*v)\circ (u*\id_{x'})=u*v=  (u*\id_{y'})\circ (\id_x*v).
$$
Therefore every morphism in $\cD$ can be generated from $\alpha$  by the operations $\pd^i$, $\xi*$, $*\xi$ and composition. Thus every morphism
is a composition of morphisms of the form 
$$
\pd^{i_r}\cdots \pd^{i_1}(\xi^s * \alpha * \xi^t).
$$
 There are $2^{n-2}(n-1)$ such morphisms in $\cD^n$, corresponding to  edges in the $(n-1)$-cube  $\Ob \cD^n=[0,1]^{n-1}$, so we call these the edge morphisms. Since $\tau_1\Xi^{n}$ is generated by edge morphisms, this implies that the functor $\cD^{n} \to \tau_1\Xi^{n}$ is full. 

Finally, the condition $(\pd^2\alpha) \circ (\xi*\alpha)= (\pd^1\alpha)\circ (\alpha*\xi)$ implies that any square of edge morphisms in $\cD^n$ commutes. Thus $\cD^n$ is the category associated to the poset $[0,1]^{n-1}$, so $\cD \to \tau_1\Xi$ is an isomorphism, as required.
 \end{proof}

\begin{corollary}\label{btcohogpd}
For $\Gamma \in QM^*(\Gpd)$, 
$$
\Hom_{\Ho(QM^*(\bS))}(\iota\bt, B\Gamma) \cong \mc(B\Gamma)/\sim,
$$
where for  pairs  $(x, a)$ as in Proposition \ref{gpdcalc}, the   equivalence relation  $\sim$ is given by saying that  $(x,a)\sim (x', a')$ if there exists $\lambda \in \Gamma^1(x,x')$ such that 
$$
(\pd^1\lambda) \circ a = a' \circ (\lambda*\lambda), \quad \sigma^0\lambda= \id_1,
$$
where $\id_1$ is the identity morphism in $\Gamma^0(1,1)$. 
\end{corollary}
\begin{proof}
Let $\Ar(\Gamma)\in QM^*(\Gpd)$ be the groupoid of arrows in $\Gamma$, defined levelwise, and observe that 
$$
B\Gamma \xra{\id} B\Ar(\Gamma) \to B(\Gamma\by \Gamma)  
$$
is a path object for $B\Gamma$. Since $\Xi$ is a cofibrant replacement for $\iota \bt$, this gives
$$
\Hom_{\Ho(QM^*(\bS))}(\iota\bt, B\Gamma) \cong \Hom_{QM^*(\bS)}(\Xi, B\Gamma)/\Hom_{QM^*(\bS)}(\Xi, B\Ar(\Gamma)),
$$
which is just $\mc(B\Gamma)/ \mc(B\Ar(\Gamma))$.

Applying  Proposition \ref{gpdcalc} to   $\Ar(\Gamma)$ shows that elements of $\mc(B\Ar(\Gamma))$ correspond to  pairs
$$
\left(x\xra{\lambda} x', \quad \begin{CD} x*x @>{\lambda*\lambda}>> x'*x'\\ @VaVV @VV{a'}V \\ \pd^1x @>{\pd^1\lambda}>> \pd^1x'\end{CD}\right), 
$$
with $(x,a), (x', a') \in \mc(B\Gamma)$ and $(\pd^1\lambda) \circ a = a' \circ (\lambda*\lambda)$, $\sigma^0\lambda= \id_1$. This gives the required description.

\end{proof}

\subsection{Linear quasi-comonoids}

Given $E \in QM^*(\bS)$,  we now wish to describe the homology groups $\H_*(E^n,\Z)$, for all $n$ --- by definition, these are just  homology groups of the simplicial abelian group $\Z\ten E$ freely generated by $E$. Applying $\Z \ten$ levelwise gives a   functor 
$$
\Z \ten \co QM^*(\bS) \to QM^*(s\Ab, \ten),
$$
where the latter category is not to be confused with the category  $QM^*(s\Ab, \by)$ considered earlier. There is a forgetful  functor, right adjoint to $\Z \ten$, and \cite{Hirschhorn} Theorem 11.3.2 allows us to put a model structure on $QM^*(s\Ab, \ten)$ for which a morphism is a weak equivalence or a fibration if the underlying map in  $QM^*(\bS)$ is so.

This means that a morphism in $QM^*(s\Ab, \ten)$  is a weak equivalence or a fibration whenever the underlying map in the Reedy category $\bS^{\Delta_{**}}$ is so. Thus the forgetful functor $QM^*(s\Ab, \ten) \to (s\Ab)^{\Delta_{**}}$ is right Quillen, where the latter category has the Reedy model structure.

Alternatively, we can forget the operations $\pd^i, \sigma^i$ and retain the multiplication. This gives us a  forgetful functor $U: QM^*(\Ab, \ten)\to G\Ring$ to simplicial (not necessarily commutative) $\N_0$-graded  rings with unit, given by $U(R)^n= R^n$.

\begin{lemma}\label{u*}
The forgetful functor $U: QM^*(s\Ab, \ten)\to sG\Ring$ is left Quillen.
\end{lemma}
\begin{proof}
We first note that the model structure which we will use for $sG\Ring$ is defined by saying that $f: R \to S$ is a weak equivalence (resp. a fibration) whenever all the maps $f^n:R^n \to S^n$ are weak equivalences (resp. fibrations) in $\bS$. That this defines a model structure follows from  \cite{Hirschhorn} Theorem 11.3.2.

We may explicitly describe the right adjoint $U_*$ by
$$
(U_*R)^n:= \prod_{\substack{m \in \N_0 \\ f: \on \to \om \text{ in } \Delta}} R^m,
$$
where for any morphism $g:n\to n'$ in $\Delta$, the map $g: (U_*R)^n \to (U_*R)^{n'}$ is given by $g(r)_f= r_{g\circ f}$, for $f:\on' \to \om$.

The matching maps of this all have canonical sections, so it follows immediately that $U_*$ preserves fibrations and trivial fibrations, hence is right Quillen.
\end{proof}

\begin{definition}
Given a (not necessarily commutative) ring $R$, and an $R$-bimodule $M$, define the set $\Der(R,M)$ of derivations  to be the set of ring homomorphisms $R \to R\oplus M\eps$ over $R$, where $\eps^2=0$ and $\eps$ commutes with everything. Equivalently, a derivation  is a morphism $f: R \to M$ such that $f(ab)= af(b) +f(a)b$.

Then (as in \cite{duboisviolette})  define the  $R$-bimodule $\Omega(R)$ to be the kernel of the multiplication map $R\ten R \to R$; this
 has the  universal property that $\Hom_{R-R}(\cot(R),M) \cong \Der(R,M)$, with the universal derivation $R \to \Omega(R)$ given by $r \mapsto r\ten 1-1\ten r$.
\end{definition}

We now also denote the forgetful functor  $QM^*(\Ab, \by) \to G\Ab $ to  $\N_0$-graded abelian groups by $U$. 
\begin{lemma}\label{cotcalc}
For $E \in QM^*(\Set)$, there is a natural isomorphism
$$
U\cot(E) \cong   U(\Z\ten\iota\bt)\ten_{U(\Z\ten E)} \Omega(U(\Z\ten E))\ten_{ U(\Z\ten E)}U(\Z\ten\iota\bt).
$$
\end{lemma}
\begin{proof}
Given a graded $U(\Z\ten\iota\bt)$-bimodule $M$, we have
$$
\Hom_{U(\Z\ten E)-U(\Z\ten E)}(\Omega(U(\Z\ten E)), M)= \Der(U(\Z\ten E),M),
$$
making use of the augmentation $U(\Z\ten E) \to U(\Z\ten\iota\bt)$ coming from the canonical map $E \to \iota\bt$ to the final object.

Now, 
$$
\Der(U(\Z\ten E),M) = \Hom_{G\Ring\da U(\Z\ten\iota\bt )}(U(\Z\ten E),U(\Z\ten\iota\bt)\oplus M\eps), 
$$
for $\eps^2=0$, and this is isomorphic to $\Hom_{QM^*(\Set)\da U_* U(\Z\ten \iota\bt ) }(E, U_* (U(\Z\ten \iota\bt)\oplus M\eps))$.

Now, observe that $c\Ab$ is equivalent to the category of $\Z \ten \iota\bt$-bimodules in $\Ab^{\Delta_{**}}$, with operations  $\pd^0$ and $\pd^{n+1}$ on level $n$ corresponding to left and right multiplication by the unique element of $(\iota\bt)^1$. Given a $\Z \ten \iota\bt$-bimodule $N$ in $\Ab^{\Delta_{**}}$, the equivalence $c\Ab \simeq QM^*(\Ab, \by)$ then combines with the forgetful functor $U_{\Ab}: QM^*(\Ab, \by) \to QM^*(\Set)$ to give rise to an object of $QM^*(\Set)$. Explicitly, this is 
$$
((\Z\ten\iota\bt)\oplus N\eps)\by_{(\Z\ten\iota\bt)} \iota\bt
$$
(noting that the underlying object in $\Set^{\Delta_{**}} $ is just $N$).

Next, observe that the forgetful functor $\Ab^{\Delta_{**}} \to G\Ab$ has right adjoint $U_*$, defined by the same formulae as the functor $U_*$   of Lemma \ref{u*}. Thus we get  $U_*M \in\Ab^{\Delta_{**}} $, which has a natural $\Z \ten \iota\bt$-bimodule structure, allowing us to regard it as a cosimplicial complex. Moreover, 
\begin{eqnarray*}
U_{\Ab}U_*M&=& ((\Z\ten\iota\bt)\oplus U_*M\eps)\by_{(\Z\ten\iota\bt)} \iota\bt \\
&=& ((U_*U(\Z\ten \iota\bt))\oplus U_*M\eps)\by_{U_* U(\Z\ten \iota\bt )}\iota\bt \\
&=& U_* (U(\Z\ten \iota\bt)\oplus M\eps)\by_{U_* U(\Z\ten \iota\bt )}\iota\bt 
\end{eqnarray*}
in $QM^*(\Set)$.
Since $\cot$ is left adjoint to $U_{\ab}$,
$$
\Hom_{c\Ab}(\cot E, U_*M) \cong \Hom_{QM^*(\Set)}(E, U_* (U(\Z\ten \iota\bt)\oplus M\eps)\by_{U_* U(\Z\ten \iota\bt )}\iota\bt),
$$
which we have already seen is isomorphic to $ \Der(U(\Z\ten E),M) $. 

Thus
\begin{eqnarray*}
\Hom_{U(\Z\ten E)-U(\Z\ten E)}(\Omega(U(\Z\ten E)), M) &\cong& \Hom_{c\Ab}(\cot E, U_*M)\\
&\cong& \Hom_{U(\Z\ten\iota\bt )-U(\Z\ten\iota\bt )}(U\cot E, M),
\end{eqnarray*}
as required.
\end{proof}

\begin{corollary}\label{btcot}
 The morphism $\oL\cot\iota\bt \simeq \cot(\iota\bt)$ is a weak equivalence.
\end{corollary}
\begin{proof}
We could use our explicit cofibrant replacement $\Xi$ to calculate $\oL\cot$, but instead we give an argument which will generalise more widely. Although $\iota\bt$ is not cofibrant, the underlying graded ring $U(\Z\ten \iota\bt)$ is freely generated by the unique element of $(\iota\bt)^1$, so is cofibrant in $sG\Ring$. If we took a cofibrant replacement $E$ of $\iota\bt$, we would then have a weak equivalence $U(\Z \ten E) \to U(\Z\ten \iota\bt)$ of cofibrant objects. Therefore 
$$
U(\Z\ten \iota\bt)\ten_{U(\Z \ten E)}\Omega(U(\Z \ten E))\ten_{U(\Z \ten E)}U(\Z\ten \iota\bt)\to \Omega(U(\Z\ten \iota\bt))
$$
would be a weak equivalence of simplicial $U(\Z\ten \iota\bt)$-modules, so Lemma \ref{cotcalc} gives a weak equivalence
$$
\cot(E) \to \cot(\iota\bt),
$$
as required.
\end{proof}

\begin{proposition}\label{adams}
Given a cofibrant object $R \in sG\Ring$, equipped with an augmentation $ R \to \Z$,   there is a spectral sequence
$$
F(\H_*(\Z\ten_R\Omega(R)\ten_R\Z)) \abuts \H_*(R), 
$$
where $F$ is the free graded (non-commutative) $\Z$-algebra functor on a graded module. This converges whenever $R^0=\Z$. 
\end{proposition}
\begin{proof}
By \cite{QHA} \S II.6, cofibrant simplicial rings are   retracts of free simplicial rings, where a simplicial ring $R_{\bt}$ is said to be free if there are free generators $C_q \subset R_q$ for all $q$, closed under the degeneracy operations of $R_{\bt}$.

Since the proposition is unchanged by taking retractions, we may assume that $R$ is free.   If $I= \ker(R \to \Z)$ is the augmentation ideal, first observe that $R= \Z \oplus I$, and that freeness gives   $\Z\ten_R\Omega(R)\ten_R\Z = I/(I \cdot I)$. There is a filtration on $R$ by powers of  $I$, with associated graded algebra $\gr_IR= \bigoplus_{n \ge 0} I^{\cdot n}/I^{\cdot n+1}$. Since $R$ is free, the canonical map $F(I/(I\cdot I)) \to \gr_IR$ is an isomorphism. The spectral sequence of the proposition is then just the spectral sequence associated to this filtration. 

Finally, if $R^0=\Z$, then $I^0=0$, so $I^{\cdot n+1}\cap R^n=0$. We may regard the spectral sequence as a direct sum of spectral sequences, using the graded decomposition. In degree $n$, this gives the spectral sequence associated to the filtration  $I^{\cdot p}\cap R^n=0$ of $R^n$. Since this filtration is bounded, the spectral sequence converges.
\end{proof}

\begin{corollary}\label{trivchar}
If $f:E \to F$ in $QM^*(\Set)$ is a morphism  with $E^0=F^0=1$, $\pi_fE \simeq \pi_fF\simeq \bt$ and $\oL\cot(E) \simeq \oL\cot(F)$, then $f$ is a weak equivalence. Moreover,  if  $\oL\cot (E) \simeq \cot(\iota\bt)$, then $E \to \iota\bt$ is a weak equivalence.
\end{corollary}
\begin{proof}

We may choose cofibrant replacements $\tilde{E}, \tilde{F}$ of $E,F$, with $\tilde{E}^0=\tilde{F}^0=1$. Since the objects $\tilde{E}^n,\tilde{F}^n$ are simply connected for all $n$, we only need to prove that the map $\H_*(\tilde{E}^n, \Z)\to\H_*(\tilde{F}^n, \Z)$ of homology groups is an isomorphism. Now, $\H_*( \tilde{E}^n, \Z)= \H_*(\Z\ten \tilde{E}^n)$, so we may apply Lemma \ref{cotcalc} and Proposition \ref{adams} to give the required isomorphism.

For the final part, note that Corollary \ref{btcot} gives $\oL\cot(\iota\bt) \simeq \cot(\iota\bt)$.
\end{proof}

\subsection{Diagonals}\label{diagsn}

In this section, we will study properties of the diagonal functor $\diag: QMM^*(\bS) \to QM^*(\bS)$, with a view to characterising $\mc(\diag E)$, thereby extending Lemma \ref{diagset} to the simplicial case.

\subsubsection{Groupoids}

Our first step is to  establish a diagonal version of Proposition \ref{gpdcalc}. 

\begin{definition}
Given a $\Delta_{**}\by \Delta_{**}$-diagram $S^{\bt, \bt}$ of sets, and a distinguished point $1 \in S^{n-1,i}$,  define $(N_h^nS)^i= S^{n,i}\cap\bigcap_{j}\ker (\sigma_h^j)$, where $\ker$ denotes the inverse image of $1$. Similarly, given $1 \in S^{i, n-1}$, define $(N_v^nS)^i= S^{i,n}\cap\bigcap_{j}\ker (\sigma_v^j)$.
\end{definition}

\begin{lemma}\label{diaggpd}
Given $\Gamma \in QMM^*(\Gpd)$, 
the set  $\mc(\diag  B\Gamma)$ consists of data $(x,s,t, a,b)$, where 
$x \in \Ob \Gamma^{11}$, and for $x_h:=\sigma^0_vx, x_v:= \sigma^0_hx$,
$$
s \in \Gamma^{20}(x_h*x_h, \pd^1_hx_h),\quad t \in \Gamma^{02}(x_v*x_v, \pd^1_vx_v)
$$
and
$$
a \in N_hN_v\Gamma^{11}(x_h * x_v, x) \quad b \in N_hN_v\Gamma^{11}(x_v*x_h, x).
$$

These data satisfy the additional conditions that 
 $$
(x_h,s) \in \mc(B\Gamma^{\bt 0}),\quad (x_v,t) \in \mc(B\Gamma^{0 \bt})
$$
(as in Lemma \ref{gpdcalc}), and that
if we set 
$\gamma=b^{-1}\circ a:x_h * x_v \to x_v*x_h$, then 
\begin{eqnarray*}
\pd^1_h\gamma &=& (x_v*s)\circ (\gamma*x_h)\circ (x_h*\gamma)\circ (s*x_v)^{-1}\\
\pd^1_v\gamma^{-1} &=& (x_h*t)\circ (\gamma^{-1}*x_v)\circ (x_v*\gamma^{-1})\circ (t*x_h)^{-1}.
\end{eqnarray*}
\end{lemma}
\begin{proof}[Proof (sketch).]
By  Proposition \ref{gpdcalc}, we know that $\mc(\diag B\Gamma)$ consists of pairs $(x, \alpha)$, with  $x \in \Ob \Gamma^{11}$, $\alpha \in \Gamma^{22}(x*x, \pd^1_h\pd^1_vx)$, satisfying various conditions. 
If we  look at $a:= \sigma^0_v\sigma^1_h\alpha,\, b:= \sigma^0_h\sigma^1_v\alpha$, then we have $a: \sigma^0_vx * \sigma^0_hx\to x$ and $b: \sigma^0_hx*\sigma^0_vx \to x$. We also set $s= (\sigma^0_v)^2\alpha \in \Gamma^{20}, t= (\sigma^0_h)^2\alpha \in \Gamma^{02}$. 

Note that
$$
(\sigma^0_vx,s) \in \mc(B\Gamma^{\bt 0}),\quad  (\sigma^0_hx,t) \in \mc(B\Gamma^{0\bt}),\quad a,b \in N_hN_v\Gamma^{11},
$$
by applying powers of $\sigma^0_h$ or $\sigma^0_v$ to the equations for $\alpha$.

Applying the operations $\sigma^i_h\sigma^j_v\sigma^k_v$ to the equation 
$$
(\pd^2_h\pd^2_v\alpha)\circ (x*\alpha)=(\pd^1_h\pd^1_v\alpha)\circ (\alpha*x)
$$
 for $\{i,j,k\}=\{1,2,3\}$ and $j<k$ 
gives us the following  equations in  $\sigma^0_v\alpha, \sigma^1_v\alpha$:
\begin{eqnarray*}
\sigma^1_v\alpha &=&   (\pd^1_hb) \circ (b * x_h)^{-1}\\
(\sigma^0_v\alpha)^{-1}\circ (\sigma^1_v\alpha)&=& (x_h *b) \circ (a*x_h)^{-1}\\
\sigma^0_v\alpha &=& (\pd^1_h a) \circ (x_h * a)^{-1},
\end{eqnarray*}
which reduce to the first condition for $\gamma$. Interchanging horizontal and vertical structures does the same for  $\sigma^0_h\alpha, \sigma^1_h\alpha$, giving the second condition (for $\gamma^{-1}$).

It remains  to show that we can recover $\alpha$ from $s,t,a,b$. It is the composition
$$
(\pd^1_h\pd^1_va)\circ (s*t) \circ (x_h*\gamma *x_v) \circ (a^{-1}*a^{-1}).
$$
\end{proof}

\begin{proposition}\label{btcohogpddiag}
For $\Gamma \in QMM^*(\Gpd)$, elements of
$
\Hom_{\Ho(QM^*(\bS))}(\iota\bt, B\diag\Gamma)
$
are represented by  data of the form $(x_h,x_v,s,t, \gamma)$, where
$$
(x_h,s) \in \mc(B\Gamma^{\bt, 0}),\quad (x_v,t) \in \mc(B\Gamma^{0,\bt}),
$$
and $\gamma:x_h * x_v \to x_v*x_h$ satisfies the conditions of Lemma \ref{diaggpd}. Two systems $(x_h,x_v,s,t, \gamma), (x_h',x_v',s',t', \gamma')$ represent the same element if and only if there exist 
$$
\lambda_h \in \Gamma^{1,0}(x_h,x'_h), \quad\lambda_v \in \Gamma^{0,1}(x_v, x_v')
$$ 
such that
\begin{eqnarray*}
(\pd^1_h\lambda_h)\circ s &=& s' \circ (\lambda_h*\lambda_h) \quad \sigma^0_h\lambda_h=1,\\
(\pd^1_v\lambda_v)\circ t &=& t' \circ (\lambda_v*\lambda_v) \quad \sigma_v^0\lambda_v=1,\\
\gamma' \circ (\lambda_h*\lambda_v) &=& (\lambda_v *\lambda_h) \circ \gamma.
\end{eqnarray*}
\end{proposition}
\begin{proof}[Proof (sketch).]
Take $(x,\alpha)$ as in the proof of Lemma \ref{diaggpd}. By Corollary \ref{btcohogpd}, $(x, \alpha)\sim (x', \alpha')$ whenever
 there exists $\lambda \in \Gamma^{11}(x,x')$ such that
$$
(\pd^1_h\pd^1_v\lambda)\circ \alpha = \alpha' \circ (\lambda*\lambda) \quad \sigma^0_h\sigma_v^0\lambda=1.
$$
Given $(x, \alpha)$, we may therefore set 
$$
x'=x_h * x_v, \quad\alpha'= (\pd^1_h\pd^1_va^{-1})\circ \alpha \circ (a*a),
$$
and define a transformation $\lambda: (x,\alpha) \to (x', \alpha')$ by $\lambda=a^{-1}$.

Therefore every element of $\Hom_{\Ho(QM^*(\bS))}(\iota\bt, B\diag\Gamma)$ has a representative with $x=x_h*x_v$, and $a=1$, giving data $(x_h,x_v,s,t, \gamma)$ as above. Two such systems are equivalent if there exists a  transformation $\lambda\in \Gamma^{11}(x_h*x_v,x'_h*x_v')$ satisfying the conditions of Lemma \ref{btcohogpd}. Since $a=1$, we recover that $\lambda= \lambda_h*\lambda_v$ for $\lambda_h:= \sigma^0_v\lambda$ and $\lambda_v:= \sigma^0_h\lambda$. The conditions for $\lambda$ then reduce to the conditions for $\lambda_h,\lambda_v$ above.  
\end{proof}

\begin{corollary}\label{diagcohogpd}
The object  $\diag^*(\Xi) \in QMM^*(\bS)$ is simply connected in every level.
\end{corollary}
\begin{proof}
By Lemma \ref{diaggpd}, the fundamental groupoid $\Gamma:=\pi_f(\diag^*\Xi) \in QMM^*(\Gpd)$ is generated by an object $x \in \Ob \Gamma^{1,1}$, together with isomorphisms $(s,t,a,b)$ satisfying the conditions of that Lemma. Consider  $\Upsilon\in QMM^*(\Gpd)$ generated by objects $x_h \in  \Ob \Gamma^{1,0}, \,x_v \in  \Ob \Gamma^{0,1}$ and isomorphisms $(s,t, \gamma)$ satisfying the conditions of  Lemma \ref{btcohogpddiag}. That Lemma implies that the canonical inclusion $\Upsilon \to \pi_f\diag^*\Xi$ is a  levelwise equivalence.

Now, the objects of $\Upsilon$ are words in $(\pd^1_h)^ix_h, (\pd^1_v)^jx_v$, and the conditions on $(s,t,a,b)$ ensure that there is a unique isomorphism between any two such words in the same level, so $\Upsilon^{mn}$ is simply connected. Thus the maps $\Upsilon \to \pi_f\diag^*\Xi\to \iota(\bt, \bt)$ are equivalences in every level, as required.
\end{proof}

\subsubsection{Abelian groups}

\begin{lemma}
The category $QMM^*(\Ab, \by)$ is equivalent to the category $cc\Ab$ of bicosimplicial abelian groups.
\end{lemma}
\begin{proof}
The proof of Lemma \ref{abqm} carries over to this context.
\end{proof}

\begin{definition}
Define $\cot: QMM^*(\Set) \to QMM^*(\Ab, \by)\simeq cc\Ab$ to be left adjoint to the forgetful  functor $ QMM^*(\Ab, \by)\to QMM^*(\Set)$.
\end{definition}

The following results have the same proofs as Corollary \ref{btcot}, Proposition \ref{adams} and Corollary \ref{trivchar}, replacing the category $G\Ring$ of graded rings with the category $GG\Ring$ of bigraded rings.

\begin{lemma}\label{btcotd}
 The morphism $\oL\cot\iota(\bt,\bt) \to \cot\iota(\bt,\bt)$ is a weak equivalence.
\end{lemma}

\begin{proposition}\label{adamsd}
Given a cofibrant object $R \in sGG\Ring$, equipped with an augmentation $ R \to \Z$   there is a spectral sequence
$$
F(\H_*(\Z\ten_R\cot(R)\ten_R\Z)) \abuts \H_*(R), 
$$
where $F$ is the free bigraded (non-commutative) $\Z$-algebra functor on a bigraded module. This converges whenever $R^{00}=\Z$. 
\end{proposition}

\begin{corollary}\label{trivchard}
If $f:E \to F$ in $QMM^*(\Set)$ is a morphism  with $E^{00}=F^{00}=1$, $\pi_fE \simeq \pi_fF\simeq \bt$ and $\oL\cot(E) \simeq \oL\cot(F)$, then $f$ is a weak equivalence. Moreover,  if  $\oL\cot (E) \simeq \cot\iota(\bt,\bt)$, then $E \to \iota(\bt,\bt)$ is a weak equivalence.
\end{corollary}

\begin{proposition}\label{diagbt}
The map $\oL \diag^*(\iota\bt) \to \iota(\bt,\bt)$ is a weak equivalence.
\end{proposition}
\begin{proof}
Proposition \ref{btcohogpddiag} shows that this gives an equivalence on fundamental groupoids, so by Corollary  \ref{trivchard} it suffices to show that
$$
\oL \cot (\oL \diag^*(\iota\bt)) \simeq \oL \cot(\iota(\bt,\bt)).
$$

Now the diagonal functor $\diag:QMM^*(\Ab,\by) \to QM^*(\Ab, \by)$ also has a left adjoint  $\diag^*_{\Ab}:QM^*(\Ab, \by) \to QMM^*(\Ab,\by)$, which  can be calculated in a similar way to the functor $d^*$ in \cite{sht} \S IV.3.3,  by studying the associated cosimplicial and bicosimplicial complexes. It follows from this description that $\diag^*_{\Ab}$ preserves weak equivalences. Moreover, the functors $\diag^*_{\Ab}\cot$ and $\cot\diag^*$ from $QM^*(\Set)$ to $QMM^*(\Ab, \by)$ are naturally isomorphic, since their right adjoints are.

By Corollary \ref{btcotd}, $\oL\cot\iota(\bt,\bt) \simeq \cot\iota(\bt,\bt) $, so it suffices  to show that
$$
\diag^*_{\Ab}\cot\Xi \to \cot\iota(\bt,\bt) = \diag^*_{\Ab}\cot\iota\bt
$$
is a weak equivalence. Since $\diag^*_{\Ab}$ preserves weak equivalences in $QM^*(\Ab, \by)$,  we need only observe that
$$
\cot\Xi \to \cot\iota\bt 
$$
is a weak equivalence by Corollary \ref{btcot}.
\end{proof}

\section{Mapping spaces}\label{mapsn}

\begin{lemma}\label{spnleft}
For $\on$ as in Lemma \ref{pnnerve}, $P_n^*\diag^*\Xi$ is a cofibrant replacement for $\alg^*\on$ in $sQ\Dat_n$ (with the model structure of Lemma \ref{sqdato}), and hence also in $sQ\Dat$.
\end{lemma}
\begin{proof}
By Proposition \ref{diagbt}, the morphism $\diag^*\Xi \to \iota(\bt,\bt)$ in $QMM^*(\bS)$ is a weak equivalence. From the description of $P_n^*$ in Lemma \ref{pnleft2}, it follows that $P_n^*$ from Definition \ref{pkdefn} preserves weak equivalences, so $ P_n^*\diag^*\Xi \to P_n^*\iota(\bt,\bt)$ is also a weak equivalence. Since $P_n^*: QMM^*(\bS) \to sQ\Dat_n$ is left Quillen, $P_n^*\diag^*\Xi$ is cofibrant (and a  realisation of $P_n^*\oL \diag^*\iota(\bt)$). Now we need only recall from  the proof of Lemma \ref{pndiag} that
$
P_n^*\iota(\bt,\bt) \cong \alg^*\on.
$
\end{proof}

\subsection{$\cMC$}

\begin{proposition}\label{MCmap1}
For $\cD \in sQ\Dat$ fibrant,
$$
\cMC(\cD)_n \simeq \coprod_{f: [0,n] \to \Ob \cD} \Map^h_{sQ\Dat_n}(\alg^*\on, f^{-1}\cD),
$$
where $\Map^h$ denotes the derived function complex $\oR \Map$ of \cite{hovey} \S 5.
\end{proposition}
\begin{proof}
For $\cE \in sQ\Dat_n$, 
$$
\mmc(\diag P_n\cE)= \HHom_{sQ\Dat_n}(P_n^*\diag^*\Xi, \cE),
$$
recalling that (unlike $sQ\Dat$)  the model category $sQ\Dat_n$ has a simplicial structure. Since $P_n^*\diag^*\Xi$ is cofibrant, this is equivalent to the derived function complex $\Map^h_{sQ\Dat_n}(P_n^*\diag^*\Xi, \cE)$ whenever $\cE$ is fibrant (\cite{hovey} Theorem 5.4.9). Since 
$P_n^*\diag^*\Xi$ is weakly equivalent to $\alg^*\on$, this function complex is also weakly equivalent to $\Map^h_{sQ\Dat_n}(\alg^*\on, \cE)$. The description now follows immediately from definition \ref{MCdef}.
\end{proof}

\begin{corollary}\label{MCmap2}
For $\cD \in sQ\Dat$ fibrant,
$$
\cMC(\cD)_n \simeq  \Map^h_{sQ\Dat}(\alg^*\on, \cD)\by^h_{\Map^h_{s\Cat}(\oO, \cD^0)^{[0,n]}}(\Ob\cD)^{[0,n]}.
$$
\end{corollary}
\begin{proof}
 Comparing the model structures of Lemma \ref{sqdato} and Proposition \ref{sqdatmodel}, it follows immediately that the functor $\cO \da sQ\Dat\to sQ\Dat_{\cO}$, given by mapping $f: \cO \to \cD$ to $f^{-1}\cD$, preserves (trivial) fibrations.  Its left adjoint is the inclusion  functor $sQ\Dat_{\cO}\to \cO \da sQ\Dat$, so these form a Quillen pair. Hence 
\begin{eqnarray*}
\Map^h_{sQ\Dat_n}(\alg^*\mathbf{n}, f^{-1}\cD) &\simeq& \Map_{[0,n] \da sQ\Dat}^h( \alg^*\mathbf{n}, [0,n] \xra{f} \cD)\\
&=& \Map_{sQ\Dat}^h(\alg^*\mathbf{n}, \cD)\by_{\Map_{sQ\Dat}^h([0,n], \cD)}^h\{f\}.
\end{eqnarray*}
Thus we have  that $\cMC(\cD)_n = \Map^h(\alg^*\mathbf{n}, \cD)\by^h_{\Map^h([0,n], \cD)}\Hom([0,n], \cD)$,
and
$$
\Map^h_{sQ\Dat}([0,n], \cD)= \Map^h_{s\Cat}([0.n], \cD^0)= \Map^h_{s\Cat}(\oO, \cD^0)^{[0,n]}.
$$
\end{proof}

\begin{definition}
For $n \in \N_0$, define $\Xi \by \alg^*\on \in sQ\Dat$ to have objects $[0,n]$ and morphisms $(\Xi \by \alg^*\on)(i,j)= \Xi$ for $i\le j$ and $\emptyset$ otherwise. This can be characterised as the coproduct $(\Xi \by \alg^*(\Ob \on))\sqcup_{ (\id / \emptyset)[0,n]}(\id / \emptyset)(\on)$ in the category $sQ\Dat$, or equivalently in the category $sQ\Dat_n$. 
\end{definition}

\begin{lemma}\label{workmor}
The natural morphism $ f:\Xi \by \alg^*\on \to P_n^*\diag^*\Xi$ is a trivial cofibration in $sQ\Dat_n$.
\end{lemma}
\begin{proof}
For $\cD \in sQ\Dat_n$, 
$$
\Hom_{sQ\Dat_n}(\Xi \by \alg^*\on, \cD) =  \mc( (P_n\cD)^{\bt, 0}\by_{(P_n\cD)^{0, 0}}(P_n\cD)^{0,\bt}).
$$

The morphism $ f$ then corresponds to the map
$$
((\sigma^0_v)^{\bt}, (\sigma^0_h)^{\bt}):\diag(P_n\cD) \to (P_n\cD)^{\bt, 0}\by_{(P_n\cD)^{0, 0}}(P_n\cD)^{0,\bt}
$$
in $QM^*(\bS)$. 

Given a trivial fibration $\cD \to \cE$ in $sQ\Dat_n$, observe that the  conditions (W1) and (F1) from Proposition \ref{sqdatmodel} ensure that 
$$
\diag(P_n\cD) \to \diag(P_n\cE)\by_{ (P_n\cE)^{\bt, 0}\by_{(P_n\cE)^{0, 0}}(P_n\cE)^{0,\bt}} (P_n\cD)^{\bt, 0}\by_{(P_n\cD)^{0, 0}}(P_n\cD)^{0,\bt}
$$
is a trivial fibration in $QM^*(\bS)$, so Lemma \ref{mcqsub} implies that the functor $\mc$ applied to this map is surjective. Therefore
$f $
is a cofibration. 

Now,  $\Xi \by \alg^*\on \simeq (\iota \bt) \by \alg^*\on= \alg^*\on$ and  $P_n^*\diag^*\Xi\simeq \alg^*\on$ by Lemma \ref{spnleft}. Thus $f$ is a trivial cofibration in $sQ\Dat_n$.
\end{proof}

\begin{definition}
Recall from \cite{rezk} 4.1 that a Segal space is defined to be a bisimplicial set $W \in s\bS$ which is Reedy fibrant, and for which the natural maps 
$$
W_k \to \overbrace{W_1\by_{\pd_0, W_0\, \pd_1}\ldots \by_{\pd_0, W_0, \pd_1}W_1}^k
$$ 
are  weak equivalences in $\bS$ for all $k$.
\end{definition}

\begin{proposition}\label{MCsegal}
For $\cD \in sQ\Dat$ fibrant, $\cMC(\cD)$ is a Segal space.
\end{proposition}
\begin{proof}
By applying Lemma \ref{MCfib}  to the morphism $\cD \to \alg^*\oO$, we know that $\cMC(\cD)$ is Reedy fibrant, since $\cosk_0(\Ob\cD)^{\hor}$ is Reedy fibrant, and $\cMC(\alg^*\oO)=\bt$. 

Letting $W:=\cMC(\cD)$, we have 
\begin{eqnarray*}
(W_n)_i &=& \coprod_{f: [0,n] \to \Ob \cD} \HHom_{sQ\Dat_n}( P_n^*\diag^*\Xi  , f^{-1}\cD)_i \\
&=& \coprod_{f: [0,n] \to \Ob \cD}  \Hom_{sQ\Dat_n}( (P_n^*\diag^*\Xi)  , (f^{-1}\cD)^{\Delta^i})\\
&=& \Hom_{sQ\Dat}( (P_n^*\diag^*\Xi)  , \cD^{\Delta^i}),
\end{eqnarray*}
where  $\cD^K$ is defined by $\Ob \cD^K= \Ob \cD$ and $\cHHom_{\cD^K}(a,b):= \cHHom_{\cD}(a,b)^K$ (note that although this makes $sQ\Dat$ into a simplicial category, it does not satisfy axiom (SM7) of a simplicial model category).

Now,
\begin{eqnarray*}
\overbrace{
(W_1\by_{\pd_0, W_0, \pd_1}\ldots \by_{\pd_0, W_0, \pd_1}W_1}^k)_i &\cong& \Hom_{sQ\Dat}((P_1^*\diag^*\Xi) {\cup}_{P_0^*\Xi}\ldots {\cup}_{P_0^*\Xi} (P_1^*\diag^*\Xi)    , \cD^{\Delta^i})  \\
&\cong& \coprod_{f: [0,n] \to \Ob \cD}\HHom_{sQ\Dat_n}( (P_1^*\diag^*\Xi) {\cup}_{\Xi}\ldots {\cup}_{\Xi} (P_1^*\diag^*\Xi) , f^{-1}\cD)_i,
\end{eqnarray*}
since $[0,1]\cup_{\{1\}}[1,2]\cup_{\{2\}} \ldots \cup_{\{n-1\}}[n-1,n]= [0,n]$.

Now, Lemma \ref{workmor} implies that
$$
(\Xi \by \alg^*\oI){\cup}_{\Xi}\ldots {\cup}_{\Xi}(\Xi \by \alg^*\oI) \to  P_1^*\diag^*\Xi {\cup}_{\Xi}\ldots {\cup}_{\Xi} P_1^*\diag^*\Xi 
$$
is a trivial cofibration (being a pushout of trivial cofibrations), and the left-hand side is just $\Xi \by \alg^*(\oI{\cup}_{\oO} \ldots {\cup}_{\oO}\oI)= \Xi \by \alg^*\ok$. This is weakly equivalent to $\alg^*\ok$, so 
$$
\overbrace{W_1\by_{\pd_0, W_0, \pd_1}\by \ldots \by_{\pd_0, W_0, \pd_1}W_1}^k \simeq \coprod_{f: [0,n] \to \Ob \cD}\Map_{sQ\Dat_n}^h(\alg^*\ok, f^{-1}\cD),
$$
 which is equivalent to $W_k$ by Proposition \ref{MCmap1}.
\end{proof}

\subsubsection{Morphism spaces}\label{morphism}

\begin{lemma}\label{cosimplicialmor}
Given $\cD \in Q\Dat$, and $x, y \in  \Ob \alg \cD$, the object $\cD(\bar{x},\bar{y}) \in \Set^{\Delta_{**}}$ has the natural structure of a cosimplicial set, where $\bar{x}, \bar{y} \in \Ob \cD$ are the objects underlying $x,y$. 
\end{lemma}
\begin{proof}
The objects $x,y$ correspond to elements $\omega_x \in \mc(\cD(x,x)), \,\omega_y \in \mc(\cD(y,y))$. In order to enhance the structure of $\cD(\bar{x},\bar{y})$, we define operations $\pd^0:= \omega_x*: \cD(x,y)^n \to \cD(x,y)^{n+1}$, and $\pd^{n+1}:= *\omega_y: \cD(x,y)^n \to \cD(x,y)^{n+1}$. The Maurer-Cartan equations ensure that these operations satisfy the necessary conditions for a cosimplicial set.
\end{proof}

\begin{definition}
Recall from  \cite{sht} \S VIII.1 that  the functor  $\Tot: c\bS \to \bS$ from cosimplicial simplicial sets to simplicial sets is given by
$$
\Tot X^{\bt} =\{ x \in \prod_n (X^n)^{\Delta^n}\,:\, \pd^i_Xx_n = (\pd^i_{\Delta})^*x_{n+1},\,\sigma^i_Xx_n = (\sigma^i_{\Delta})^*x_{n-1}\}. 
$$ 
When  $X$ is Reedy fibrant, homotopy groups of the total space are related to homotopy groups of the spaces $X^n$ via a spectral sequence given in \cite{sht} Proposition VIII.1.15.
\end{definition}

\begin{proposition}\label{loopmorsub}
Given $\cD \in sQ\Dat$ fibrant, and $x, y \in  \Ob \alg (\cD_0)$, there is a natural weak equivalence
$$
\cMC(\cD)_1\by_{\mc(\cD)\by \mc(\cD)}\{(x,y)\} \simeq \Tot \cD(\bar{x},\bar{y}).
$$
\end{proposition}
\begin{proof}
Define $f: [0,1] \to \Ob \cD$ by $f(0)=\bar{x}, f(1)=\bar{y}$, and set $\cE:= f^{-1}\cD \in sQ\Dat_1$ (for $sQ\Dat_n$ as in Lemma \ref{pnleft2}). Then  
\begin{eqnarray*}
\cMC(\cD)_1\by_{\mc(\cD)\by \mc(\cD)}\{(x,y)\} &=& \HHom_{(\Xi\sqcup \Xi)\da  sQ\Dat_1}(P_1^*\diag^*\Xi, \cE)\\
&=& \HHom_{\alg^*(\oO \sqcup \oO)\da sQ\Dat_1}((P_1^*\diag^*\Xi){\cup}_{(\Xi\sqcup \Xi)}\alg^*(\oO \sqcup \oO), \cE).
\end{eqnarray*}

By Lemma \ref{workmor}, 
$$
(P_1^*\diag^*\Xi){\cup}_{(\Xi\sqcup \Xi)}\alg^*(\oO \sqcup \oO) \simeq (\Xi \by \alg^*\oI) {\cup}_{(\Xi\sqcup \Xi)}\alg^*(\oO \sqcup \oO)= \alg^*\oI, 
$$
so our expression reduces to $\oR\HHom_{\alg^*(\oO \sqcup \oO)\da sQ\Dat_1}(\alg^*\oI, \cE)$.

Now, in the simplicial category $\alg^*(\oO \sqcup \oO)\da sQ\Dat_1$, a cofibrant replacement for $\alg^*\oI$ is given by the object $\C$ with $\C(0,0)= \C(1,1)= \iota\bt$, $\C(1,0)=\emptyset$ and $\C(0,1)^{\bt}= \Delta^{\bt}$. The multiplicative structure is determined determined by setting  $ \omega_0*a=\pd^0a$ and $a*\omega_1=  \pd^{n+1}a$, for $a \in \Delta^n$ and $\omega_0, \omega_1$ the unique elements of $\C(0,0)^1$ and $\C(1,1)^1$.
 
Now,
$$
\HHom_{\alg^*(\oO \sqcup \oO)\da sQ\Dat_1}(\C, \cE)= \{ e \in \prod_n (\cE(0,1)^n)^{\Delta^n}\,:\, \pd^i_{\cE} e_n = (\pd^i_{\Delta})^*e_{n+1},\,\sigma^i_{\cE}e_n = (\sigma^i_{\Delta})^*e_{n-1}\}
$$ 
(which preserves (trivial) fibrations, proving that $\C$ is cofibrant).

This expression is just $\Tot \cE(0,1)$, where $\cE(0,1)$ is given the cosimplicial structure of Lemma \ref{cosimplicialmor}, but $\cE(0,1)=  \cD(\bar{x},\bar{y})$, giving the  required description.
\end{proof}

\subsection{$\cN$}

\begin{definition}
Given $\cD \in sQ\Dat$, define $\cN(\cD) \in s\bS$ by $\cN(\cD)_n:= \Map^h_{sQ\Dat}(\alg^*\mathbf{n}, \cD)$.
\end{definition}

\begin{proposition}\label{NSS}
Any Reedy fibrant replacement for $\cN(\cD)$ is a Segal space.
\end{proposition}
\begin{proof}
If $W:= \cN(\cD)$, we need to show that the maps
$$
W_k \to \overbrace{W_1\by_{\pd_0, W_0\, \pd_1}^h \ldots \by_{\pd_0, W_0, \pd_1}^hW_1}^k
$$
are  weak equivalences  for all $k$.

Now, the right hand side is given by
$$
 \Map^h_{sQ\Dat}(\alg^*\oI{\cup}^{\oL}_{\alg^*\oO} \ldots {\cup}^{\oL}_{\alg^*\oO}\alg^*\oI , \cD),
$$
so we need only show that $\alg^*\oI{\cup}^{\oL}_{\alg^*\oO} \ldots {\cup}^{\oL}_{\alg^*\oO}\alg^*\oI \simeq \alg^*\ok$.

A cofibrant replacement of the diagram for this homotopy colimit is given by taking $\Xi$ instead of $\alg^*\oO$, and $P_1^*\diag^*\Xi$ instead of $\alg^*\oI$. Thus the calculation of Proposition \ref{MCsegal} can be interpreted as saying that 
$$
\overbrace{\alg^*\oI{\cup}^{\oL}_{\alg^*\oO} \ldots {\cup}^{\oL}_{\alg^*\oO}\alg^*\oI}^k \simeq \alg^*\ok,
$$
which gives the required result.
\end{proof}

\begin{definition}
Given a Segal space $W$, define $\Ob W:= (W_0)_0$. For $x, y \in \Ob W$, define $\map_W(x,y):=  \{x\}\by_{W_0, \pd_1} W_1\by_{\pd_0}\{y\} \in \bS$. There is a natural category $\Ho(W)$ with objects $\Ob W$ and morphisms $\pi_0\map_W(x,y)$.
\end{definition}

\begin{definition}
Recall from \cite{rezk} 7.4 that a morphism $f: U \to V$ of Segal spaces is said to be a  Dwyer-Kan equivalence if  
\begin{enumerate}
\item
$\Ho(f): \Ho(U) \to \Ho(V)$ is an equivalence of categories, and 
\item 
 for all $x,y \in \Ob U$, the map $\map_U(x,y) \to \map_V(fx,fy)$ is a weak equivalence in $\bS$. 
\end{enumerate}
\end{definition}

\begin{proposition}\label{MCNequiv}
Given a fibrant object $\cD \in sQ\Dat$, the natural transformation $\cMC(\cD) \to  \cN(\cD)^f$ is a Dwyer-Kan equivalence of Segal spaces, where $(-)^f$ denotes Reedy fibrant replacement.
\end{proposition}
\begin{proof}
Corollary \ref{MCmap2}  amounts to saying that $\cMC(\cD)$ is weakly equivalent to the homotopy fibre product of 
$$
 \cN(\cD)\to {\cosk_0\Map^h_{sQ\Dat}(\oO,\cD^0)^{\hor}}\la \cosk_0(\Ob \cD)^{\hor}.
$$
Thus, for $x, y \in \Ob \cMC(\cD)$, the space $\map_{\cMC(\cD)}(x,y)$ is is weakly equivalent to the homotopy fibre product of
$$
 \map_{\cN(\cD)^f}(x,y) \to{ \map_{\cosk_0\Map^h_{sQ\Dat}(\oO,\cD^0)^{\hor}}(x,y)}\la \map_{\cosk_0(\Ob\cD)^{\hor}}(x,y).
$$
Now, for any $S \in \bS$ and $x,y \in S_0$, $\map_{\cosk_0(S)^{\hor}}(x,y)= \{(x,y)\}$, so
$$
\map_{\cMC(\cD)}(x,y) \simeq \map_{\cN(\cD)^f}(x,y).
$$

It therefore remains only to show that the morphism $\Ho(\cMC(\cD)) \to \Ho(\cN(\cD)^f)$ of categories is essentially surjective. Since $\oO$ is cofibrant, the map $\Ob \cD \to \Map^h_{sQ\Dat}(\oO,\cD^0)$ is surjective on $\pi_0$. Any fibrant replacement $(\Ob \cD)^f \to \Map^h_{sQ\Dat}(\oO,\cD^0) $ will also be surjective on $\pi_0$, and hence surjective on  level $0$ (by path-lifting). Therefore the map $\Ob \cMC(\cD) \to  \Ob \cN(\cD)^f$ is surjective on objects, so \emph{a fortiori} essentially surjective.
\end{proof}

\begin{definition}
Given a Segal space $W$, let $W_{\mathrm{hoequiv}} \subset W_1$ consist of components whose images in $\pi_0\map_W(x,y)$ are equivalences in $\Ho(W)$. $W$ is said to be a complete Segal space if the map $\sigma_0: W_0 \to W_{\mathrm{hoequiv}}$ is a weak equivalence.
\end{definition}

\begin{lemma}\label{endmono}
Given a levelwise trivial cofibration $E \to F$ in  $\bS^{\Delta_{**}}$, and $X \in \Set^{\Delta_{**}^{\op}}$, the map
$$
X\uleft{\by} E \to X\uleft{\by} F
$$
is a trivial cofibration in $\bS$, for $\uleft{\by}$ as in Definition \ref{uleft}.
\end{lemma}
\begin{proof}
 Since $\Set^{\Delta_{**}}$ is equivalent to the category of augmented simplicial objects by Remark \ref{cflein1},  a morphism $E \to F$ is  a levelwise trivial cofibration precisely when $E^0 \to F^0$ is a trivial cofibration and $E^{\ge 1} \to F^{\ge 1}$ corresponds to a levelwise trivial  cofibration of bisimplicial sets. By \cite{sht} Theorem IV.3.9, this second condition is the same as being a Reedy trivial cofibration of bisimplicial sets.

Let $L^nE$ be the Reedy latching object of $E$ in $\Set^{\Delta_{**}}$ (as in \cite{hovey} Definition 5.2.2). Explicitly, this is the quotient of $\coprod_{i=1}^{n-1}E^{n-1}$ by the relations $(\pd^je)_i \sim (\pd^{i-1}e)_j$ for $e \in E^{n-2}$ and $i \le j$. Note that $L^0E= L^1E= \emptyset$. For $E^{\ge 1} \to F^{\ge 1}$ to be a Reedy trivial cofibration says that the latching maps $E^n\cup_{L^nE}L^nF \to F^n$ are trivial cofibrations for all $n \ge 1$. Thus a morphism $E \to F$ is a levelwise trivial cofibration in $\bS^{\Delta_{**}}$ precisely when the latching maps  are trivial cofibrations for all $n \ge 0$. 

Let $(X\uleft{\by} E)^{(n)} \subset X\uleft{\by} E$ be the subspace generated by $X_i \by E^i$ for $i \le n$, and let $N_nX =X_n -(\bigcup_{0 \le r \le n-1} \sigma_rX_{n-1})$. Then 
$$
(X\uleft{\by} E)^{(n)}= (X\uleft{\by} E)^{(n-1)}\cup_{(N_nX \by L^nE)}(N_nX \by E^n),
$$
so the map 
$$
(X\uleft{\by} E)^{(n)}\cup_{(X\uleft{\by} E)^{(n-1)}}(X\uleft{\by} F)^{(n-1)}\to (X\uleft{\by} F)^{(n)}
$$ 
is a trivial cofibration. 

We then proceed inductively to show that $(X\uleft{\by} E)^{(n)} \to (X\uleft{\by} F)^{(n)} $ is a trivial cofibration. The case $n=0$ is immediate, and assuming the  $n-1$ case gives a trivial cofibration 
$$
(X\uleft{\by} E)^{(n)} \to  (X\uleft{\by} E)^{(n)}\cup_{(X\uleft{\by} E)^{(n-1)}}(X\uleft{\by} F)^{(n-1)}.
$$
The result above then implies that 
$$
(X\uleft{\by} E)^{(n)}\to (X\uleft{\by} F)^{(n)}
$$
is a trivial cofibration, which completes the induction.

Thus $X\uleft{\by} E \to X\uleft{\by} F$ is a transfinite composition of trivial cofibrations, so must be a trivial cofibration, as required.
\end{proof}

\begin{proposition}\label{catswork1}
If $\bI$ is a small category and $\cD \in sQ\Dat$, then there  are canonical equivalences
$$
\Map_{s\bS}^h((B\bI)^{\hor}, \cN(\cD))\simeq \Map_{sQ\Dat}^h(\alg^*\bI , \cD).
$$
\end{proposition}
\begin{proof}
Lemma \ref{spnleft} shows that $P_n^*\diag^*\Xi$ is a cofibrant replacement for $\alg^*\on$, so for $K \in \bS$, 
\begin{eqnarray*}
\Map_{s\bS}^h(K^{\hor}, \cN(\cD))&\simeq& \holim_{\substack{ \lla \\n \in \Delta\da K}} \Map_{s\bS}^h((\Delta^n)^{\hor}, \cN(\cD))  \\
&\simeq& \holim_{\substack{ \lla \\n \in \Delta\da K}} \Map^h_{sQ\Dat}(P_n^*\diag^*\Xi, \cD)\\
&\simeq& \Map^h_{sQ\Dat}(\holim_{\substack{ \lra \\ n \in (\Delta\da K)^{\op}}} P_n^*\diag^*\Xi, \cD)\\
& =& \Map_{sQ\Dat}^h(P_K^*\diag^*\Xi, \cD).
\end{eqnarray*}

Now,  Lemma \ref{pkleft2} gives a method for calculating the functor $P_{B\bI}^*$. Since $(\diag^*\Xi)^{n, \bt} \in \bS^{\Delta_{**}}$ is levelwise contractible for all $n$  (by Proposition \ref{diagbt}), a choice of point $x \in (\diag^*\Xi)^{n, 1}$ determines a levelwise trivial cofibration $\iota\bt \to (\diag^*\Xi)^{n, \bt}$ in $\bS^{\Delta_{**}}$. Therefore Lemma \ref{endmono} implies that
$$
K(a,b)\uleft{\by}(\iota\bt) \to K(a,b) \uleft{\by}(\diag^*\Xi)^{n, \bt} 
$$
is a trivial cofibration for all $a, b \in K_0$.

By Lemma \ref{pkleft}, $(P_{B\bI}^* \diag^*\Xi)(a,b)^n= (B\bI)(a,b) \uleft{\by}(\diag^*\Xi)^{n, \bt}$, so the calculation above shows that the canonical map
$$
(P_{B\bI}^* \diag^*\Xi)(a,b)^n\to (B\bI)(a,b) \uleft{\by}(\iota\bt)
$$
is a weak equivalence, so
$$
P_{B\bI}^* \diag^*\Xi\to P_{B\bI}^*\iota(\bt,\bt) 
$$
is a weak equivalence, and $P_{B\bI}^*\iota(\bt,\bt) \cong \alg^*\bI$ as in Lemma \ref{sdcdiagramsub}.
\end{proof}

\begin{corollary}\label{NCSS}
Any Reedy fibrant replacement  of $\cN(\cD)$ is a complete Segal space.
\end{corollary}
\begin{proof}
We know that $\cN(\cD)^f$ is a Segal space by Proposition \ref{NSS}. If we let $J$ be the simply connected groupoid on two objects,  \cite{rezk} Proposition 6.4 implies that it suffices to show that the map
$$
\cN(\cD)^f_0 \to \Map_{s\bS}^h((BJ)^{\hor}, \cN(\cD)^f)
$$
is a weak equivalence.

By Proposition \ref{catswork1},
$$
\Map_{s\bS}^h((B\bI)^{\hor}, \cN(\cD))\simeq \Map_{sQ\Dat}^h(\alg^*\bI , \cD),  
$$
and $\alg^*J \simeq \alg^*\oO$, since $J$ is simply connected. The result then follows by considering this equivalence in the cases $\bI=J$ and $\bI=\oO$.
\end{proof}

\begin{definition}
Given a simplicial category $\bI$, let $B\bI \in s\bS$ be the nerve of $\bI$, as considered in \cite{bergner3} \S 8 (where it was denoted $R$). Explicitly, $(B\bI)_n \in \bS$ is given by
$$
(B\bI)_n = \coprod_{x_0, \ldots, x_n \in \Ob \bI} \HHom_{\bI}(x_0, x_1)\by \ldots \by\HHom_{\bI}(x_{n-1}, x_n).
$$ 
\end{definition}

In \cite{rezk} Theorem 7.2, Rezk introduced the complete Segal space model structure $\CSS$ on the category $s\bS$, whose fibrant objects are complete Segal spaces. 
In \cite{bergner3}, Bergner showed that there is a chain of Quillen equivalences between the model categories $\CSS$ and $s\Cat$. Therefore the following theorem can be interpreted as saying that $\cN$ and $\cMC$ are derived right adjoints to $\alg^*: s\Cat \to sQ\Dat$.

\begin{theorem}\label{catswork}
If $\bI$ is a simplicial category and $\cD \in sQ\Dat$, then there  are canonical equivalences
$$
 \Map_{sQ\Dat}^h(\alg^*\bI , \cD)\simeq\Map_{s\bS}^h(B\bI, \cN(\cD))\simeq \Map_{\CSS}^h(B\bI, \cN(\cD))\simeq \Map_{\CSS}^h(B\bI, \cMC(\cD^f)),
$$
where $\cD^f$ is a Reedy fibrant replacement of $\cD$.
\end{theorem}
\begin{proof}
First note that we can regard $K \in s\bS$ as a diagram $\Delta^{\op} \to \bS$, and evaluate $K$ in the simplicial category $s\bS$ as   the coend $K=K\ten_{\Delta^{\op}}(\Delta)^{\hor} = \int^n K_n^{\ver} \by (\Delta^n)^{\hor}$ (in the notation of \cite{Hirschhorn} Definition 18.3.2).  Since $K$ and $\Delta: \Delta^{\op} \to \bS$ are both Reedy cofibrant, this means that
$
\Map_{s\bS}^h(K, \cN(\cD))
$ 
is (assuming $\cN(D)$ Reedy fibrant)  the end $\int_n \Map_{s\bS}^h(K_n^{\ver} \by (\Delta^n)^{\hor}, \cN(\cD))  = \int_n \cN(\cD)_n^{K_n}$.  
The proof of Proposition \ref{catswork1} now adapts to prove  the first equivalence, replacing homotopy limits with ends.

For any $K \in s\bS$, we may define $P_K^*:QMM^*(\bS) \to sQ\Dat(\bS)$, specialising to the functor $P_K^*$ of Lemma \ref{pkleft2} when $K \in s\Set$. If $\bI$ is a simplicial category, then for all $m$ the simplicial set $[n] \mapsto ((B\bI)_n)_m$ is the nerve of the category $\bI_m$, given by $\Ob \bI_m= \Ob \bI$ and $\Hom_{\bI_m}(x,y):= \HHom_{\bI}(x,y)_i$. Thus the description of $P_{B\bI}^*$ in Lemma \ref{pkleft2} generalises to simplicial categories $\bI$, and as in Proposition \ref{catswork1}, the morphism $P_{B\bI}^*\diag^*\Xi \to P_{B\bI}^*\iota(\bt, \bt) = \alg^*\bI$ is a weak equivalence. 

By Corollary \ref{NCSS}, $\cN(\cD)^f$ is a complete Segal space, and since the identity functor $\CSS \to s\bS$ is right  Quillen for the Reedy  model structure, the second  equivalence follows. Proposition \ref{MCNequiv} gives the third equivalence, since \cite{rezk} Theorem 7.7 shows that a morphism of Segal spaces becomes a weak equivalence in $\CSS$ if and only if it is a Dwyer-Kan equivalence.
\end{proof}

\subsection{$\cDel$}

\begin{definition}
Let $\cD \mapsto \cD^{\oL \gpd}$ be the left-derived functor $\Ho(sQ\Dat) \to \Ho(sQ\Dpd)$ of the functor $\cD \mapsto \cD^{\gpd}$ from Definition \ref{dpdmodel}.
 \end{definition}

\begin{lemma}\label{gpdn}
There is a natural equivalence $ (\alg^*\on)^{\oL\gpd} \simeq \alg^*(\on^{\gpd})$.
\end{lemma}
\begin{proof}
By Lemma \ref{sdcdiagramsub}, $P_n^*\diag^*\Xi$ is a cofibrant resolution of $\alg^*\on$, so we need to show that  
$$
(P_n^*\diag^*\Xi)^{\gpd} \simeq \alg^*(\on^{\gpd}).
$$

By Lemma \ref{workmor}, the natural morphism $ f:\Xi \by \alg^*\on \to P_n^*\diag^*\Xi$ is a trivial cofibration in $sQ\Dat_n$, and hence also in $sQ\Dat$.
Since $(-)^{\gpd}: sQ\Dat \to sQ\Dpd$ is left Quillen, this implies that $f^{\gpd}$ is a trivial cofibration, so it suffices to show that $ (\Xi \by \alg^*\on)^{\gpd} \simeq \alg^*(\on^{\gpd})$. Now, $(\Xi \by \alg^*\on)$ has objects $[0,n]$ and $(\Xi \by \alg^*\on)^{\gpd}(i,j)= \Xi$ for all $i,j$. Since $\Xi \simeq \iota\bt$, this gives a weak equivalence $(\Xi \by \alg^*\on)^{\gpd} \to \alg^*(\on^{\gpd}) $, as required.
 \end{proof}

\begin{proposition}\label{ddefmap}
For $\cD \in sQ\Dpd$ fibrant, 
$$
\ddel(\cD)\simeq \Map_{sQ\Dat}^h(\alg^*\on, \cD)
$$
for all $n\ge 0$.
\end{proposition}
\begin{proof}
Let $\ddel^*$ be the left adjoint to $\ddel: sQ\Dpd \to \bS$, making it a left Quillen functor by Proposition \ref{ddefquillen}. Since $\ddel(\cD)_0= \coprod_{x \in \Ob \cD}\mc(\cD(x,x))$, it follows that $\ddel^*(\Delta^0)= \Xi$, so $\ddel^*(\Delta^0) \simeq \alg^*\oO$. Since $\Delta^{\bt}$ is a cofibrant cosimplicial resolution of $\Delta^0$ in $\bS$, $\ddel^*(\Delta^{\bt})$ is thus a  cofibrant cosimplicial resolution of $\alg^*\oO$, so  
$$
\Map_{sQ\Dpd}^h(\alg^*\oO,\cD) \simeq ([m] \mapsto \Hom_{sQ\Dpd}(\ddel^*(\Delta^m), \cD)) = \ddel(\cD),
$$ 
giving the case $n=0$.

For the general case, note that  the inclusion functor $sQ\Dpd \to sQ\Dat$ is right Quillen, with left adjoint given by groupoid completion, and so 
$$
\Map_{sQ\Dat}^h(\alg^*\on,\cD) \simeq \Map_{sQ\Dpd}^h((\alg^*\on)^{\oL \gpd},\cD)\simeq \Map_{sQ\Dpd}^h(\alg^*(\on^{\gpd}),\cD),
$$
by Lemma \ref{gpdn}.
Since $\on^{\gpd} \simeq \oO$, this is just $\Map_{sQ\Dpd}^h(\alg^*\oO,\cD)$, which completes the proof.
\end{proof}

\begin{proposition}\label{DefCSS}
The functor $\cDel: sQ\Dpd\to s\bS$ of Definition \ref{Defdef} is right  Quillen for the complete Segal space model structure $\CSS$ of \cite{rezk} Theorem 7.2. In particular, for any $\cD \in sQ\Dpd$ fibrant, the simplicial space $\cDel(\cD)$ is a complete Segal space. In fact, $\cDel(\cD)$ is then equivalent to the constant simplicial space $\ddel(\cD)$.
 \end{proposition}
\begin{proof}
By Corollary \ref{Defquillen}, we know that $\cDel$ is right  Quillen for the Reedy model structure.
In order to prove the remaining statements, it suffices that for any fibration $f:\cD \to \cE$, the fibrations
$$
(f,(\pd_i)^n):\cDel(\cD)_n \to \cDel(\cE)_n\by_{\ddel(\cE)}\ddel(\cD)
$$
are trivial fibrations for $i=0,1$ and for all $n$. This is equivalent to showing that the cofibrations
$$
\ddel^*(\Delta^m){\cup}_{\ddel^*(\pd\Delta^m)}(\cDel_n)^*(\pd\Delta^m)\to (\cDel_n)^*(\Delta^m)
$$
are trivial cofibrations for all $m \ge 0$  (and both choices of map $(\pd_i)^{n*}:\ddel^* \to \cDel_n^*$), where $(\cDel_n)^*: \bS \to sQ\Dpd$ is left adjoint to $\cDel_n$.  

Since $\cDel$ is right Quillen, the functor $\cDel_n$ is right Quillen, so the proof of Proposition \ref{ddefmap} adapts to show that the cosimplicial object $ i \mapsto (\cDel_n)^*(\Delta^m)$ is a cofibrant cosimplicial resolution of $(P_n^*\diag^*\Xi)^{\gpd}$ in $sQ\Dpd$, which by Lemma \ref{gpdn} is equivalent to $\alg^*(\on^{\gpd})$, and hence to $\alg^*\oO$. Therefore the cofibrations $\ddel^*(\Delta^m) \to (\cDel_n)^*(\Delta^m)$ are all weak equivalences. Moreover, $\ddel^*(\pd\Delta^m) \simeq \ho\LLim_{\Delta_{m-1,+}}\alg^*\oO$ (for $\Delta_{m-1,+}$ the  subcategory of   $\Delta$ on objects $\le m-1$ with only injective morphisms)
and similarly for $(\cDel_n)^*(\pd\Delta^m)$.  Thus the cofibrations $\ddel^*(\pd\Delta^m) \to (\cDel_n)^*(\pd\Delta^m)$ are also weak equivalences, giving a weak equivalence
$$
\ddel^*(\Delta^m)\to \ddel^*(\Delta^m){\cup}_{\ddel^*(\pd\Delta^m)}(\cDel_n)^*(\pd\Delta^m).
$$
Since $\ddel^*(\Delta^m)\to (\cDel_n)^*(\Delta^m) $ is a trivial cofibration (coming  the trivial cofibration $(\pd^i)^n: \Delta^0 \to \Delta^n$), this gives the required result.
\end{proof}

\begin{corollary}
For $\cD \in sQ\Dpd$ fibrant, the morphism $\cMC(\cD) \to \cDel(\cD)$ is a Dwyer-Kan equivalence, and $\cDel(\cD) \simeq \cN(\cD)$.
\end{corollary}
\begin{proof}
Proposition \ref{ddefmap} implies that the constant simplicial space $\ddel$ is a model for $\cN$, so Proposition \ref{DefCSS} shows that $\cDel$ must also be a model for $\cN$.  By Proposition \ref{MCNequiv}, the morphism $\cMC(\cD) \to \cN(\cD)^f$ is a Dwyer-Kan equivalence. 
\end{proof}

\section{Maurer-Cartan and classifying spaces}\label{class}

In many cases, a simplicial quasi-descent datum $E$  has additional structure, and in this section we show how this simplifies the Segal spaces $\cMC(E)$ and $\cDel(E)$ (where appropriate). This will give descriptions which are not only simpler, but are related to more familiar  formulae.

\subsection{Groups}

\subsubsection{Cosimplicial simplicial groups}

\begin{definition}
Let $cs\Gp$ be the category of cosimplicial simplicial groups, equipped with its Reedy simplicial model structure over simplicial groups.
\end{definition}

\begin{example}\label{chains}
Given a simplicial set $X$ and a simplicial group $G$, our main motivating example of a  cosimplicial simplicial group is the complex $\CC^{\bt}(X,G)$ given by
$$
\CC^n(X,G)_m:=G_m^{X_n}, 
$$
with cosimplicial operations $\pd^i:=G_m^{\pd_i}, \sigma^i:=G_m^{\sigma_i}$, and simplicial operations $\pd_i=\pd_i^G, \sigma_i= \sigma_i^G$.
\end{example}

\begin{definition}\label{mcdefgp1}
 Define $\mc:c\Gp \to \Set$ by  
$$
\mc(G):= \z^1(G)= \{ \omega \in G^1 \,:\, \sigma^0\omega =1\,\, \pd^1 \omega = \pd^2\omega \cdot \pd^0\omega\}.
$$ 
\end{definition}

\begin{definition}\label{mcdefgp}
Define $\mmc: cs\Gp \to \bS$ by $\mmc(G) \subset \prod_{n\ge 0} (G^{n+1})^{\Delta^n}$, satisfying the conditions of \cite{htpy} Lemma \ref{htpy-maurercartan}, i.e.
 the elements $\omega_n \in(G^{n+1})^{\Delta^n}$ satisfy 
\begin{eqnarray*}
\pd_i\omega_n &=& \left\{\begin{matrix} \pd^{i+1}\omega_{n-1}  & i>0 \\ (\pd^1\omega_{n-1})\cdot(\pd^0\omega_{n-1})^{-1} & i=0,\end{matrix} \right.\\
\sigma_i\omega_n &=& \sigma^{i+1}\omega_{n+1},\\
\sigma^0\omega_n&=& 1.
\end{eqnarray*}
Define $\mc :cs\Gp \to \Set$ by $\mc(G)=\mmc(G)_0$, noting that this agrees with Definition \ref{mcdefgp1} when $G \in c\Gp$. 
\end{definition}

\begin{remark}\label{pathspace}
If $G \in cs\Gp$ is of the form $G^n= \CC^{\bt}(X,H)$, for $X \in \bS$ and $H \in s\Gp$ as in Example \ref{chains}, then \cite{htpy} Lemma \ref{htpy-maurercartan} gives a canonical isomorphism $\mc(G) \cong \Hom_{\bS}(X, \bar{W}H)$, where $\bar{W}$ is the classifying space functor of \cite{sht} Ch.V.4.
\end{remark}

\begin{definition}\label{cEdef}
Define  $\cE: cs\Gp \to QM^*(\bS)$ by $\cE(G)^n= G^n$, with identity $1 \in G^0$, operations $\pd^i_{\cE(G)}=\pd^i_G, \sigma^i_{\cE(G)}=\sigma^i_G$, and Alexander-Whitney product
$$
g*h= ((\pd_G^{m+1})^ng)\cdot ((\pd^0_G)^mh),
$$  
for $g \in G^m,\, h \in G^n$.

Observe that $\cE$  is right Quillen and preserves weak equivalences. Denote its left adjoint by $\cE^*$. Note that $\mc(G)=\mc(\cE(G))$.
\end{definition}

Note  that  the equivalence of Lemma \ref{abqm} is just given by $\cE: c\Ab \to QM^*(\Ab,\by)$.

\begin{proposition}\label{rhommc2}
For $G \in cs\Gp$ fibrant, there is a natural isomorphism
$$
\oR \HHom_{cs\Gp}(\cE^*(\iota\bt), G) \simeq \mmc(G)
$$
in $\Ho(\bS)$.
\end{proposition}
\begin{proof}
This is essentially the same as Proposition \ref{rhommc}. First, note that $\cE^*(\iota\bt)^n$ is the free group on $n$ generators, with constant simplicial structure. If $s$ is the unique element of $(\iota\bt)^1$, then the generators in level $n$ are given by $\pd^{i_{n-1}}\ldots \pd^{i_2}\pd^{i_1}s$, for $0< i_1 <i_2< \ldots < i_{n-1} \le n$.  We then define $\Phi \in cs\Gp$ by 
$$
\Phi^n:=\Fr(\coprod_{j < n}\Delta^{j}),
$$
where $\Fr$ denotes the free group functor. We give this the operations dual to those on $\bar{W}G$ in \cite{sht} Ch.V, i.e. for $x \in \Delta^j \subset \Phi^n$, we set
\begin{eqnarray*}
\pd^i_{\Phi}(x)&=& \left\{ \begin{matrix}  \pd^{i-n+j}_{\Delta}(x)   & i> n-j\\  
(\pd^0_{\Delta}x)\cdot x & i=n-j \\
x  & i < n-j.\end{matrix}\right.\\
\sigma^i_{\Phi}(x) &=& \left\{ \begin{matrix}  \sigma^{i-n+j}_{\Delta}(x)   & i\ge n-j\\  
1 & i=n-j \\
x  & i < n-j.\end{matrix}\right.
\end{eqnarray*}

Thus $\Hom_{s\Gp}(\Phi^n, G^m)= (\bar{W}G^m)_n$, with cosimplicial operations on $\Phi$ corresponding to simplicial operations on $\bar{W}G$,   so  we have
$$
\Hom_{cs\Gp}(\Phi, G)= \{f\in \prod_n (\bar{W}G^n)_n\,:\, \pd^if_n= \pd_if_{n+1} \, \sigma^if_n= \sigma_i f_{n-1}\} =  \Hom_{c\bS}(\Delta, \bar{W}G),
$$ 
where $\Delta$ is the cosimplicial space given by $\Delta^n$ in level $n$.

The proof of \cite{sht} Lemma V.5.3 adapts to show that this is isomorphic to $\mc(G)$. Explicitly,  $\underline{\omega} \in \mc(G)$ corresponds to the maps $\Delta^n \to\bar{W}G^n$ given by the element $(\omega_{n-1}, \pd^0 \omega_{n-2}, \ldots,(\pd^0)^{n-1}\omega_0) \in (\bar{W}G^n)_n$. (This also implies that $\Phi \cong G(\Delta)$, for $G$ the  loop group functor of  \cite{sht} \S V.5.) Hence
$$
\Hom_{cs\Gp}(\Phi, G)\cong \mmc(G).
$$

 The inclusion $\mc(G_0) \into \mmc(G)$ corresponds to a map $ \Phi \to \iota\bt$. Since $\coprod_{i < n}\Delta^{i}$ is weakly equivalent to a disjoint union of $n-1$ points, this map must be a weak equivalence.

Now, we may show (similarly to  Lemma \ref{mcqsub}, replacing $I^n$ by $\Delta^n$) that for all trivial fibrations $G \to H$, 
$$
\mc(G)\to {\mc}(H)
$$
is a surjection, so $\Phi$ is cofibrant. 

Thus, for $G$ fibrant,
$$
\oR \HHom(\cE^*(\iota\bt), G) \simeq \HHom(\Phi, G) \simeq \mmc(G).
$$
\end{proof}

\subsubsection{Equivalence of Maurer-Cartan spaces}

\begin{proposition}\label{cfmc}
There are equivalences
$$
\mmc(\cE(G))\simeq \mmc(G)
$$
in $\bS$, functorial in fibrant objects $G \in cs\Gp$. Here, the functors $\mmc$ on the left and right are those from Definitions \ref{mcdefqm} and \ref{mcdefgp} respectively.
 
\end{proposition}
\begin{proof}
Since $G$ is fibrant, $\cE(G)= \oR \cE(G)$.  By Proposition \ref{rhommc},
$$
\mmc(\cE(G))\cong \oR\Hom_{QM^*(\bS)}(\iota\bt, \oR\cE(G))\cong \HHom_{cs\Gp}(\oL\cE^*(\iota\bt), G).
$$
Thus, by Proposition \ref{rhommc2}, we need only show that
$$
\oL\cE^*(\iota\bt) \to \cE^*(\iota\bt)
$$
is a weak equivalence in $cs\Gp$. A model for $\oL\cE^*(\iota\bt)$ is given by $\cE^*\Xi$, for $\Xi$ from Definition \ref{Xidef}.

Since $\cE^*$ is a left adjoint, it commutes with coequalisers, so  $ \pi_0(\cE^*\Xi)^n = \cE^*(\pi_0\Xi^n)= \cE^*((\iota\bt)^n)$. Since  $\cE^*(\iota\bt)^n$ is a discrete group,   we need only show that the components of $(\cE^*\Xi)^n$ are contractible for all $n$. This is equivalent to saying that the universal cover $W$  of $B\cE^*(\iota\bt)^n$ is contractible, which will follow from the Hurewicz theorem if $\H_i(W, \Z) = 0$ for all $i>1$. This is the same as saying  that the homology groups $\H_i((\cE^*\Xi)^n, \Z\cE^*(\iota\bt)^n)$ are zero for all $i>1$. 

Let $cs\Rep(\cE^*(\iota\bt))  $ be the category of  abelian cosimplicial simplicial $\cE^*(\iota\bt)$-representations, and consider the functor $cs\Rep(\cE^*(\iota\bt)) \to cs\Gp\da \cE^*(\iota\bt)$ given by $V \mapsto V\rtimes \cE^*(\iota\bt)$. This is a right adjoint, and is right Quillen for the Reedy model structure on $cs\Rep(\cE^*(\iota\bt))$. Denote the derived left adjoint by $\oL \cot$, and observe that for $V \in s\Rep(\cE^*(\iota\bt)^n)$,
$$
\Hom_{\Ho(s\Rep(\cE^*(\iota\bt)^n))}((\oL \cot G)^n,V) \cong \Hom_{\Ho(s\Gp\da \cE^*(\iota\bt)^n)}(G^n, V[-i]\rtimes\cE^*(\iota\bt)^n) \cong  \bH^{1}(B(G^n), \bt; V),
$$
where the final group is hypercohomology (with coefficients in a complex of local systems). 
Thus
$$
\H_i(\oL \cot G)^n\cong \H_{i+1}(B(G^n), \bt; \Z \cE^*(\iota\bt)^n).
$$ 
It therefore suffices to show that $\H_i(\oL \cot \cE^*\Xi)=0$ for all $i \ge 1$.

Now, for $W \in cs\Rep(\cE^*(\iota\bt))$,
\begin{eqnarray*}
&&
\HHom_{cs\Gp\da \cE^*(\iota\bt)}(\cE^*\Xi, W \rtimes \cE^*(\iota\bt))\\
&&\cong
\HHom_{QM^*(\bS)\da \cE\cE^*(\iota\bt)}(\Xi, \cE(W \rtimes \cE^*(\iota\bt)) )\\
&&\cong
\HHom_{QM^*(\bS)}(\Xi,  \cE(W \rtimes \cE^*(\iota\bt))\by_{\cE\cE^*(\iota\bt)}\iota\bt   )\\
&&\cong
\HHom_{QM^*(\Ab,\by)}(\cot \Xi, \cE(W \rtimes \cE^*(\iota\bt))\by_{\cE\cE^*(\iota\bt)}\iota\bt  ).
\end{eqnarray*}

This leads us to consider the functor $\jmath: cs\Rep(\cE^*(\iota\bt)) \to cs\Ab $, given by 
$$
W \mapsto \CC( \cE(W \rtimes \cE^*(\iota\bt))\by_{\cE\cE^*(\iota\bt)}\iota\bt).
$$
Explicitly, we see that this has objects $w \cdot \varpi_m$ in level $m$, for $\varpi_m$ the image of $ (\iota\bt)^m \to  \cE\cE^*(\iota\bt)^m$. The operations are $\pd^i (w \cdot \varpi_m)= \pd^i(w) \cdot\varpi_{m+1}$ for $0 < i \le m$, $\sigma^i(w \cdot \varpi_m)= \sigma^i(w) \cdot \varpi_{m-1}$, and
\begin{eqnarray*}
\pd^0_{\jmath W}(w \cdot \varpi_m)&=& \varpi_1 *(w \cdot \varpi_m)= (\pd^2)^{m}(\varpi_1)\cdot \pd^0_W(w)\cdot  \pd^0(\varpi_{m})\\
&=& \ad_{(\pd^2)^{m}(\varpi_1)}(\pd^0_Ww)\cdot \varpi_{m+1},
\end{eqnarray*}
and
$$
\pd_{\jmath W}^{m+1}(w \cdot \varpi_m)=(w \cdot \varpi_m)*\varpi_1 =\pd^{m+1}_Ww \cdot \varpi_{m+1}.
$$

Thus $(\jmath W)^n \cong W^n$, with the same operations as $W$, except for 
$$
\pd^0_{\jmath W}(w)= ((\pd^2)^{m}\varpi_1)\star \pd^0_W(w),
$$
where $\star$ denotes the group action of $\cE^*(\iota\bt)^{m+1}$ on $W^{m+1}$.

It therefore follows that the left adjoint $\jmath^*$ is given by 
$$
(\jmath^*U)^n = U^n[\cE^*(\iota\bt)^{n}], 
$$
the free representation on generators $U^n$. This will have the same operations on generators as $U$, except that $\pd^0_{\jmath^*U }(u) = ((\pd^2)^{m}\varpi_1)^{-1}\star \pd^0_U(u)$. 

For the natural Reedy model structure on $cs\Rep(\cE^*(\iota\bt))$, $\jmath^*$ is clearly left Quillen, so it follows that $\jmath$ is a right Quillen functor. Moreover, the descriptions above show that $\jmath$ and $\jmath^*$ both preserve weak equivalences.
Thus
$$
\oL \cot \cE^*\Xi \simeq \jmath^*\oL\cot (\iota\bt),
$$
so $\H_i(\oL \cot \cE^*\Xi) \cong \H_iD(\Z[-1])$ by Lemma \ref{cot2}. In particular, $\H_i(\oL \cot \cE^*\Xi)=0$ for all $i>0$, as required.
\end{proof}

\subsubsection{$\ddel$}

\begin{definition}\label{ddefdefgp}
For $G \in cs\Gp$, there is an adjoint action  of $G^0$ on $\mmc(G)$, given by
$$
(g*\omega)_n= (\pd_0 (\pd^1)^{n+1}(\sigma_0)^{n+1}g) \cdot \omega_n \cdot (\pd^0 (\pd^1)^n(\sigma_0)^ng^{-1}),
$$
as in \cite{htpy} Definition \ref{htpy-defdef}.

We then define  $\ddel(G)$ to be the homotopy quotient $\ddel(G)= [\mmc(G)/^hG^0]:= \mmc(G)\by_{G^0}WG^0\in \bS$.
\end{definition}

\begin{remark}\label{pathdef}
If $G= \CC^{\bt}(X, H)$ for $X \in \bS, H \in s\Gp$, then $\ddel(G)\simeq \Map_{\bS}(X, \bar{H})$. This is because $\ddel(G)_n= \Hom(X, \bar{W}(G^{\Delta^n}) \by (\cosk_0 G_0)^n)$, and $[n] \mapsto \bar{W}(G^{\Delta^n}) \by (\cosk_0 G_0)^n$ gives a fibrant simplicial resolution of $\bar{W}G$ in $\bS$.
\end{remark}

\begin{proposition}\label{cfdef}
The equivalence of Proposition \ref{cfmc} is $G^0$-equivariant, giving  isomorphisms 
$$
\ddel(\cE(G))\simeq \ddel(G)
$$
in $\Ho(\bS)$, functorial in fibrant objects $G \in cs\Gp$.
\end{proposition}
\begin{proof}
If we let $\Z \in QM^*(\bS)$ be the object given by $\Z$ in degree $0$ and $\emptyset$ in higher degrees, then 
$$
E^0 \cong \HHom_{QM^*(\bS)}(\Z, E),
$$
so the $E^0$-action on $\mmc(E)$ corresponds to a $\Z$-coaction on $\Xi$, i.e. a map
$$
\Xi \to \Xi\coprod \Z
$$ 
in $QM^*(\bS)$, where $\coprod$ denotes coproduct in the category $QM^*(\bS)$, satisfying a coassociativity condition.

Since $\cE^*$ is a left adjoint, it preserves coproducts, so we get  a coaction of $\cE^*(\Z)$ on $\cE^*(\Xi)$, i.e. a coassociative map
$$
\cE^*(\Xi) \to \cE^*(\Xi)\star \cE^*(\Z),
$$
 where $\star$ denotes free product.   Of course, $\HHom_{cs\Gp}(\cE^*(\Z), G )\cong \cE(G)^0=G^0$. 

Now, the morphisms
$$
\mc(\cE(G)) \la \mc(\cE(G)_0) =\mc(G_0) \to \mc(G)
$$
are equivariant with respect to the action of $G^0_0$. Since $\Hom_{cs\Gp}(\cE^*(\Z), G )\cong G^0_0$, this amounts to saying that the weak equivalences
$$
\cE^*(\Xi) \to  \cE^*(\iota\bt) \la \Phi
$$
are $\cE^*(\Z)$-coequivariant, with the $\cE^*(\Z)$-coaction on $\Phi$ corresponding to the adjoint action of 
Definition \ref{ddefdefgp}. 

We now make use of the fact that for $H \in s\Gp$ acting on $Y \in \bS$, one model for the homotopy quotient $[Y/^hH]$ is given by first forming the simplicial object $[Y/H]$ in the category of groupoids, then forming the nerve $B[Y/H]$ (which is a bisimplicial set), giving $ [Y/^hH] \simeq \diag B[Y/H]$.

Given $C \in cs\Gp$ equipped with a $\cE^*(\Z)$-coaction, we now define the cosimplicial diagram $\beta(C) \in (cs\Gp)^{\Delta}$ by the property that $\HHom_{cs\Gp}(\beta(D), G)= B[\HHom_{cs\Gp}(D,G)/G^0] \in (\bS)^{\Delta^{\op}}$. Explicitly, $\beta(D)^n= D\star \overbrace{\cE^*(\Z)\star\cE^*(\Z) \star \ldots \star \cE^*(\Z)}^n$. 

Now $\cE^*(\Z)$ is just the cosimplicial group $\cE^*(\Z)^n = \overbrace{\Z\star\Z \star \ldots \star \Z}^{n+1}$, so $\cE^*(\Z)$ is levelwise cofibrant.
Since $\cE^*(\Z) $, $\cE^*(\Xi) $,  $\cE^*(\iota\bt)$ and  $\Phi$ are all levelwise cofibrant, the morphisms
$$
\beta(\cE^*(\Xi)) \to  \beta(\cE^*(\iota\bt)) \la \beta(\Phi)
$$
are (levelwise) weak equivalences.

In the Reedy model category $(cs\Gp)^{\Delta}$, we now choose a cofibrant replacement $C$ for $\beta(\Phi)$, and a factorisation  $\beta(\cE^*(\Xi)) \to F \to \beta(\cE^*(\iota\bt))$ with $\beta(\cE^*(\Xi)) \to F $ a trivial cofibration and $F \to \beta(\cE^*(\iota\bt))$ a trivial fibration. Since $C$ and $F$ are weakly equivalent in the Reedy model category $(cs\Gp)^{\Delta}\da \beta(\cE^*(\iota\bt))$, there is an explicit weak equivalence $f: C \to F$ in this category. Now, Reedy cofibrations are \emph{a fortiori} levelwise cofibrations, so the objects $C^n,F^n$ are cofibrant in $cs\Gp$. We therefore have levelwise weak equivalences
$$
\HHom_{cs\Gp}( \beta(\Phi) , G)\to \HHom_{cs\Gp}( C , G), \text{ and } \HHom_{cs\Gp}( F, G)\to \HHom_{cs\Gp}( \beta(\cE^*(\Xi)) , G)
$$
in $(\bS)^{\Delta^{\op}}$, for all $G \in cs\Gp$. Combining these with $f$  gives weak equivalences
$$
B[\mmc(G)/G^0] \to  \HHom_{cs\Gp}( C , G) \xla{f^*} \HHom_{cs\Gp}( F, G) \to B[\mmc(\cE(G))/G^0],
$$
in $(\bS)^{\Delta^{\op}}$, and taking diagonals gives the required result that $\ddef(G) \simeq \ddef(\cE(G))$.
\end{proof}

\subsection{Lie algebras}\label{liesn}

\subsubsection{Nilpotent DGLAs}

\begin{definition}
Let $DG_{\Z}\hat{N}LA$ denote the category  of pro-nilpotent differential graded Lie algebras (DGLAs) over $k$.  

Explicitly, a DGLA is a  graded vector space $L=\bigoplus_{i \in \Z} L^i$ over $k$, equipped with operators $[,]:L \by L \ra L$ bilinear and $d:L \ra L$ linear,  satisfying:

\begin{enumerate}
\item $[L^i,L^j] \subset L^{i+j}.$

\item $[a,b]+(-1)^{\bar{a}\bar{b}}[b,a]=0$.

\item $(-1)^{\bar{c}\bar{a}}[a,[b,c]]+ (-1)^{\bar{a}\bar{b}}[b,[c,a]]+ (-1)^{\bar{b}\bar{c}}[c,[a,b]]=0$.

\item $d(L^i) \subset L^{i+1}$.

\item $d \circ d =0.$

\item $d[a,b] = [da,b] +(-1)^{\bar{a}}[a,db]$
\end{enumerate}

Here $\bar{a}$ denotes the degree of $a$, mod $ 2$, for $a$ homogeneous.

A DGLA is said to be nilpotent if the lower central series $\Gamma_nL$ (given by $\Gamma_1L=L$, $\Gamma_{n+1}L= [L, \Gamma_nL]$) vanishes for $n\gg 0$.

Thus $DG_{\Z}\hat{N}LA$ is the category of pro-objects in the category of nilpotent DGLAs.
\end{definition}

\begin{definition}
Given a pro-nilpotent Lie algebra $\g$, define $\hat{\cU}(\g)$ to be the pro-unipotent completion of the universal enveloping algebra of $\g$, regarded as a pro-object in the category of algebras. As in \cite{QRat} Appendix A, this is a pro-Hopf algebra, and we define $\exp(\g) \subset  \hat{\cU}(\g)$ to consist of elements $g$ with $\vareps(g)=1$ and $\Delta(g)= g\ten g$, for $\vareps: \hat{\cU}(\g) \to k$ the augmentation (sending $\g$ to $0$), and $\Delta: \hat{\cU}(\g) \to \hat{\cU}(\g)\ten \hat{\cU}(\g)$ the comultiplication.

Since $k$ is assumed to have characteristic $0$, exponentiation gives an isomorphism from $\g$ to $\exp(\g)$, so we may regard $\exp(\g)$ as having the same elements as $\g$, but with multiplication given by the Campbell--Baker--Hausdorff formula. 
\end{definition}

\begin{definition}
Given a $\Z$-graded pro-nilpotent DGLA $L^{\bt}$, define the Maurer-Cartan set by 
$$
\mc(L):= \{\omega \in  L^{1}\ \,|\, d\omega + \half[\omega,\omega]=0 \in  \bigoplus_n L^{2}\}
$$
Define the gauge group $\Gg(L)$ by $\Gg(L):= \exp(L^0)$, which acts on $\mc(L)$ by the gauge action 
 $$
g(\omega)=   g\cdot \omega \cdot g^{-1} -dg\cdot g^{-1},
$$
where $\cdot$ denotes multiplication in the universal enveloping algebra of $L$. 
That $g(\omega) \in \mc(L)$ is a standard calculation (see \cite{Kon} or \cite{Man}).

Let $\pi^1(L):= \mc(L)/\Gg(L)$ be the quotient set.
\end{definition}

\begin{lemma}\label{obsdgla}
If $e:L \onto M$ has kernel $K$, with $[K,L]=0$, then there is an obstruction map $o_e:\pi^1(M) \to \H^2(K)$, with $o_e^{-1}(0)$ being the image of $\pi^1(L)$. Moreover, $\pi^1(L)$ is a principal $\H^1(K)$-bundle over $ o_e^{-1}(0)$ 
\end{lemma}
\begin{proof}
This is well-known (see \cite{Man} \S 3,  for instance). 
\end{proof}

\begin{definition}
Let $O(\mc)$ be the pro-nilpotent DGLA representing $\mc$, so $\Hom(O(\mc), L) \cong \mc(L)$. Explicitly, $O(\mc)$ is the free  pro-nilpotent graded Lie algebra on one generator $x$ in degree $1$, with differential determined by $dx=\half[x, x]$. 

Similarly, define $O(\Gg)$ to represent $\Gg$; this is freely generated by $y, dy$, for $y \in O(\Gg)^0$. Note that this has a 
cogroup structure in $DG_{\Z}\hat{N}LA$, coming from the group structure on $\Gg$.

Define $T$ and $O(T)$ by $T(L)=\Hom(O(T),L):= \exp(\z^0L)$; this is freely generated by $z \in O(T)^0$, with $dz=0$. The embedding $T \into \Gg$ corresponds to the projection $y \mapsto z, dy \mapsto 0$, and $O(T)$ is a quotient cogroup of $O(\Gg)$ in $DG_{\Z}\hat{N}LA$.
\end{definition}

We may therefore regard $\mc, \Gg, T$ as being objects of the opposite category $(DG_{\Z}\hat{N}LA)^{\op}$, which is a full subcategory of the category of presheaves on $DG_{\Z}\hat{N}LA$, via the Yoneda embedding.

\begin{proposition}\label{psidef2}
The map $q:\Gg \to \mc$ given by $g \mapsto g(0)$ gives an isomorphism between $\mc$ and the right quotient of $\Gg$ by $T$ in the category $(DG_{\Z}\hat{N}LA)^{\op}$. 
\end{proposition}
\begin{proof}
It is immediate that for $h \in T(L)$ we have $h(0)=0$, since $dh=0$. This gives a morphism $\Gg(L)/T(L) \to \mc(L)$, functorial in $L$, and we need to show that $\mc$ is the universal representable presheaf with this property, in other words that
$$
\xymatrix@1{   \Gg\by T    \ar@<0.5ex>[r]^-{\pr_1} \ar@<-0.5ex>[r]_-{\mu} & \Gg\ar[r]^{q} & \mc}
$$
is a coequaliser in $(DG_{\Z}\hat{N}LA)^{\op}$, where $\mu(g,t) = g\cdot t$.

The forgetful functor $U^{\op}$ from  $DG_{\Z}\hat{N}LA$ to the category $G_{\Z}\hat{N}LA$ of    pro-nilpotent $\Z$-graded Lie algebras preserves and reflects all equalisers, so $U:(DG_{\Z}\hat{N}LA)^{\op} \to (G_{\Z}\hat{N}LA)^{\op} $ preserves and reflects all coequalisers. It therefore suffices to show that $U\mc = (U\Gg)/(UT)$.

Now, the forgetful functor $U^{\op}$ has a right adjoint $R$, given by $(R L)^n= L^n\by L^{n+1}$, with $[(a,a'),(b,b')]= ([a, b], [a', b] +(-1)^{\bar{a}}[a, b'])$ and $d(a,a')= (a',0)$. This gives
$$
(U\mc)(L)/ (U\Gg)(L)= \mc(RL)/\Gg(RL) =\pi^1(RL). 
$$
Applying $R$ the tower  $   \ldots \to L/[L,[L,L]] \to   L/[L,L]$ gives a tower of surjections satisfying the conditions of Lemma \ref{obsdgla}. Since $\H^*(R M)= 0$ for all $M$, this  implies by induction that $\pi^1(RL)=0$ for all $L$. Thus $U\Gg$ acts transitively on $U\mc$. 

In particular, this means that the canonical element in $(U\mc)(U^{\op}O(\mc))$ is of the form $Uq(s)$ for some $ s \in (U\Gg)(U^{\op}O(\mc))$. Via the Yoneda embedding, $s$ is equivalent to the data of a section of $Uq: U\Gg \to U\mc$. Since the fibres of $Uq$ are precisely the $UT$-orbits, the map  $ U\mc \by UT \to U\Gg $ given by $(\omega, t)\mapsto s(\omega)\cdot t$ is an isomorphism. It follows immediately that the fork is a coequaliser, since $(U\mc \by UT)/UT = U\mc$.
\end{proof}

\subsubsection{Cosimplicial simplicial groups}

\begin{definition}\label{gauge}
Given $G \in  cs\Gp$, define the gauge group $\Gg(G)$ to  be the subgroup of  $\prod_n G_n^n$ consisting of those $\underline{g}$ satisfying
\begin{eqnarray*}
\pd_ig_n &=& \pd^{i}g_{n-1}  \quad \forall i>0, \\
\sigma_ig_n &=& \sigma^{i}g_{n+1} \quad \forall i,
\end{eqnarray*}
similarly to \cite{htpy} Definition \ref{htpy-gauge}.
Note that $G_0^0$ can be regarded as a subgroup of $\Gg(G)$, setting $g_n= (\pd^1)^n(\sigma_0)^ng$, for $g\in G_0^0$. The group $T(G):=\Tot(G)_0$ is the subgroup of $\Gg(G)$ consisting of those $\underline{g}$ for which $\pd_0g_n=\pd^0g_{n-1}$.

The action of Definition \ref{ddefdefgp} extends to an action of $\Gg(G)$ on $\mc(G)$, given by
$$
(g*\omega)_n= (\pd_0g_{n+1}) \cdot \omega_n \cdot (\pd^0g_n^{-1}),
$$ 
as in \cite{htpy} Definition \ref{htpy-defdef}.
\end{definition}

\begin{remark}\label{pathgauge}
If $G \in cs\Gp$ is of the form $G= \CC^{\bt}(X,H)$, for $X \in \bS$ and $H \in s\Gp$, then \cite{htpy} (Lemma \ref{htpy-maurercartan} and Proposition \ref{htpy-psidef}) gives a canonical isomorphism $\Gg(G) \cong \Hom_{\bS}(X, WH)$ for $W$  the canonical $G$-torsor on $\bar{W}G$, as in \cite{sht} Ch.V. Moreover,  $T(H) \cong \Hom_{\bS}(X, H)$, and the map $q:\Gg(G)\to \mc(G)$ corresponding to the gauge action on $1$ comes from the identification $\bar{W}G = G \backslash WG$ of \cite{sht} Ch.V. 
\end{remark}

\begin{lemma}
The functors $\mc$, $\Gg$ and $T$ are all representable in $cs\Gp$.
\end{lemma}
\begin{proof}
$\mc$ is represented by  the object $\Phi$ defined in Proposition \ref{rhommc2}, which we will now denote by $O(\mc)$, so $\Hom_{cs\Gp}(O(\mc), G) \cong \mc(G)$. 

Define $O(T)$  by  $O(T)^n= \Fr(\Delta^n)$ (where $\Fr$ denotes the free group functor), with cosimplicial operations coming from those on the cosimplicial space $\Delta^{\bt}$; this gives $\Hom_{cs\Gp}(O(T),G)\cong T(G)$.

Similarly, define $O(\Gg)$ to represent $\Gg$; this  is given by $O(\Gg)^n= \Fr(\coprod_{j \le n}\Delta^{j})$, with operations given on $x \in \Delta^j \subset O(\Gg)^n$ by 

\begin{eqnarray*}
\pd^i_{O(\Gg)}(x)&=& \left\{ \begin{matrix}  \pd^{i-n+j}_{\Delta}(x)   & i> n-j\\  x  & i \le n-j.\end{matrix}\right.\\
\sigma^i_{O(\Gg)}(x) &=& \left\{ \begin{matrix}  \sigma^{i-n+j}_{\Delta}(x)   & i> n-j\\  x  & i \le n-j.\end{matrix}\right.
\end{eqnarray*}
The isomorphism $\Gg(G)\cong \Hom_{cs\Gp}(O(\Gg),G)$ is given by  $\underline{g} \in \Gg(G)$ mapping $ \Delta^j \subset O(\Gg)^n$ to $G^n$ via the element  $(\pd_0)^{n-j}g_n \in G^n_j$. 

Note that $O(\Gg)$ has a 
cogroup structure in $cs\Gp$, coming from the group structure on $\Gg$. Explicitly, this is the map $O(\Gg) \to O(\Gg)\star O(\Gg)$ given by $x \mapsto x\star x$ for $x \in \Delta^j \subset O(\Gg)^n$.  The embedding $T \into \Gg$ corresponds to the quotient map $i:O(\Gg) \to O(T)$ given by mapping $ \Delta^j \subset O(\Gg)^n$ to $\Delta^n \subset O(T)^n$ via $(\pd^0_{\Delta})^{n-j}$. 
Thus $O(T)$ is a quotient cogroup of $O(\Gg)$ in $cs\Gp$.
\end{proof}

We may therefore regard $\mc$, $\Gg$, and $T$ as being objects of the opposite category $(cs\Gp)^{\op}$.

\begin{proposition}\label{psidef}
The map $q:\Gg \to \mc$ given by $g \mapsto g*1$ gives an isomorphism between $\mc$ and the right quotient of $\Gg$ by $T$ in the category $(cs\Gp)^{\op}$. 
\end{proposition}
\begin{proof}
Define $P \in cs\Gp$ by 
$$
P^n:=\Fr(\coprod_{j \le n}\Delta^{j}).
$$
 We give this the operations dual to those on $WG$ in \cite{sht} Ch.V, i.e. for $x \in \Delta^j \subset P^n$, we set
\begin{eqnarray*}
\pd^i_{P}(x)&=& \left\{ \begin{matrix}  \pd^{i-n+j}_{\Delta}(x)   & i> n-j\\  
(\pd^0_{\Delta}x)\cdot x & i=n-j \\
x  & i < n-j.\end{matrix}\right.\\
\sigma^i_{P}(x) &=& \left\{ \begin{matrix}  \sigma^{i-n+j}_{\Delta}(x)   & i\ge n-j\\  
1 & i=n-j \\
x  & i < n-j.\end{matrix}\right.
\end{eqnarray*}
In particular, $O(\mc)^n$ is the simplicial subgroup of $P^n$ on generators $\coprod_{j < n}\Delta^{j}$, making $O(\mc)$ a subobject of $P$. 

Now, \cite{htpy} Proposition \ref{htpy-psidef} adapts to give an isomorphism   $\psi: P \to O(\Gg)$, given by mapping $x \in \Delta^j \subset P^n$ to $(\pd^0_{\Delta}x)\cdot  x^{-1}$ when $j< n$, and $x^{-1}$ when $j=n$. 

The right action of $T$ on $\Gg$ corresponds to the coaction $\mu:O(\Gg)\to O(\Gg) \star O(T)$ given by mapping  $x \in \Delta^j \subset O(\Gg)^n$ to $x_{\Gg}\cdot ((\pd^0_{\Delta})^{n-j}x)_T$, where $y_{\Gg}$ and $y_T$ denote the copies of $y$ in $O(\Gg)^n$ and in $O(T)^n$, respectively. There is an obvious isomorphism $P^n \cong O(\mc)^n\star O(T)^n$, since $O(T)^n= \Fr(\Delta^n)$, and this is equivariant for the $O(T)$-coaction (with trivial coaction on $O(\mc)$), since  $x \in \Delta^j \subset P^n$ has $\mu(\psi(x))= \psi(x) \in O(\Gg)$ for $j<n$, and   $\mu(\psi(x)) = \psi(x) \cdot i(x)$ when $j=n$.

Therefore
$$
\xymatrix@1{   \Gg\by T     \ar@<0.5ex>[r]^-{\pr_1} \ar@<-0.5ex>[r]_-{\mu} & \Gg\ar[r]^{q} & \mc}
$$
is a coequaliser in $cs\Gp^{\op}$, as required.
\end{proof}

\subsubsection{ Lie algebras to groups }

\begin{definition}
Given an $\N_0$-graded DGLA $L$, let $DL$ be its cosimplicial denormalisation. Explicitly,
$$
D^nL:= \bigoplus_{\begin{smallmatrix} m+s=n \\ 1 \le j_1 < \ldots < j_s \le n \end{smallmatrix}} \pd^{j_s}\ldots\pd^{j_1}L^m.
$$
We then  define operations $\pd^j$ and $\sigma^i$ using the cosimplicial identities, subject to the conditions that $\sigma^i L =0$ and $\pd^0v= dv -\sum_{i=1}^{n+1}(-1)^i \pd^i v$ for all $v \in L^n$.

We now have to define the Lie bracket $\llbracket,\rrbracket$ from $D^nL \ten D^nL$ to $D^n L$. Given a finite set  $I$ of distinct strictly positive integers, write $\pd^I= \pd^{i_s}\ldots\pd^{i_1}$, for $I=\{i_1, \ldots i_s\}$, with $i_1 < \ldots < i_s$. The Lie bracket is then   defined on the basis by 
$$
\llbracket \pd^Iv, \pd^J w\rrbracket:= \left\{ \begin{matrix} \pd^{I\cap J}(-1)^{(J\backslash I, I \backslash J)}[v,w] & v\in L^{|J\backslash I|}, w\in L^{|I\backslash J|},\\ 0 & \text{ otherwise},\end{matrix} \right.
$$
where for disjoint sets $S,T$ of integers, $(-1)^{(S,T)}$ is the sign of the shuffle permutation of $S \sqcup T $ which sends the first $|S|$ elements to $S$ (in order), and the remaining $|T|$ elements to $T$ (in order). 
Beware that this formula cannot be used to calculate  $\llbracket \pd^Iv, \pd^J w\rrbracket $ when $0 \in I \cup J$ (for the obvious generalisation of $\pd^I$ to finite sets $I$ of distinct non-negative integers).
\end{definition}

\begin{theorem}\label{cfexp}
Given a simplicial pro-nilpotent $\N_0$-graded DGLA $L^{\bt}_{\bt}$, there are canonical isomorphisms 
$$
\Gg(\exp(DL)) \cong \Gg(\Tot^{\Pi} N^sL), \quad \mc(\exp(DL)) \cong \mc(\Tot{\Pi} N^sL),  
$$
compatible with the respective gauge actions. Here, $D$ is cosimplicial denormalisation and $N^s$ is simplicial normalisation (as in Definition \ref{doldkan}). This isomorphism acts as the identity on the subgroups $\exp(L^0_0) \le\Gg(\exp(DL))$ and $\exp(L^0_0) \le \Gg(\Tot^{\Pi} N^sL)$
\end{theorem}
\begin{proof}
In order to keep track of the various gradings in these categories, we will write $DG_{\Z}$ for $\Z$-graded cochain complexes, $DG$ for cochain complexes in non-negative degrees, and $dg$ for chain complexes in non-negative degrees.

On the  category  $sDG\hat{N}LA$ of simplicial pro-nilpotent $\N_0$-graded DGLA, these functors are all representable, since $\exp$, $D$, $\Tot^{\Pi}$ and $N^s$ are all right adjoints, so
\begin{eqnarray*}
\Hom_{sDG\hat{N}LA}(D^*\exp^*O(\Gg), L) &\cong& \Gg(\exp(DL))\\
\Hom_{sDG\hat{N}LA}(D^*\exp^*O(\mc), L) &\cong& \mc(\exp(DL))\\
\Hom_{sDG\hat{N}LA}((N^s)^*(\Tot^{\Pi})^*O(\Gg), L) &\cong& \Gg(\Tot^{\Pi} N^sL )\\
\Hom_{sDG\hat{N}LA}((N^s)^*(\Tot^{\Pi})^*O(\mc), L) &\cong&\mc(\Tot{\Pi} N^sL),
\end{eqnarray*}
for $\exp^*: cs\Gp \to cs\hat{N}LA$, $D^*: cs\hat{N}LA\to sDG\hat{N}LA$, $(\Tot^{\Pi})^*: DG_{\Z}\hat{N}LA \to dgDG\hat{N}LA$, and $(N^s)^*:dgDG\hat{N}LA\to  sDG\hat{N}LA$ the corresponding left adjoints. 

We may therefore regard the functors $D^*\exp^*\Gg$, $D^*\exp^*\mc$, $(N^s)^*(\Tot^{\Pi})^*\Gg$ and $(N^s)^*(\Tot^{\Pi})^*\mc$ as objects of $(sDG\hat{N}LA)^{\op}$, and likewise for $D^*\exp^*T$ and $ (N^s)^*(\Tot^{\Pi})^*T$.

Now, the isomorphism $\Gg(\exp(DL)) \cong \Gg(\Tot^{\Pi} N^sL)$ simply follows from the proof of the  Dold-Kan correspondence,  which generalises to give an equivalence of categories between $\N_0$-graded complexes and ``simplicial abelian groups without $\pd_0$''. This isomorphism maps the  subgroup $T(\exp(DL))$ to $T(\Tot^{\Pi} N^sL)$ isomorphically, by the usual Dold-Kan correspondence.

The isomorphism $D^*\exp^*\mc \cong (N^s)^*(\Tot^{\Pi})^*\mc$ follows by taking the right quotients 
$$
(D^*\exp^*\Gg)/(D^*\exp^*T) \cong ((N^s)^*(\Tot^{\Pi})^*\Gg)/((N^s)^*(\Tot^{\Pi})^*T)
$$
in $(sDG\hat{N}LA)^{\op}$, by Propositions \ref{psidef} and \ref{psidef2}.
\end{proof}

\begin{remark}
Note that this gives a shorter and more conceptual proof of \cite{htpy} Theorems \ref{htpy-qs} and \ref{htpy-defqs}, and that in that context we may shorten the arguments, replacing Proposition \ref{psidef} with \cite{htpy} Proposition \ref{htpy-psidef}.
\end{remark}

\begin{definition}
Define $\underline{\mc\Tot{\Pi} N^s}: sDG\hat{N}LA \to \bS$ by 
$$ 
\underline{\mc\Tot{\Pi} N^s}(L)_n:=   \mc(\Tot{\Pi} N^s (L^{\Delta^n})),
$$ 
and define $\ddel: sDG\hat{N}LA \to \bS$ to be the homotopy quotient $\ddel := [\underline{\mc\Tot{\Pi} N^s}/^h\exp(L^0)]$, i.e.
$$
\ddel(L) := \underline{\mc\Tot{\Pi} N^s} (L)\by_{\exp(L^0)}W\exp(L^0),
$$
where $\exp(L^0)_n \subset \Gg(L^{\Delta^n})$ acts via the gauge action.
\end{definition}

\begin{remark}
This is essentially the functor used to  construct derived formal stacks in \cite{hinstack}. 
\end{remark}

\begin{corollary}\label{expmap2}
There are  canonical isomorphisms
$$
\ddel(L) \cong \ddel(\exp(DL))
$$
in $\bS$, functorial in $L \in sDG\hat{N}LA$.
\end{corollary}
\begin{proof}
By Theorem \ref{cfexp}, $\underline{\mc\Tot{\Pi} N^s}(L) \cong \mmc(L)$. This isomorphism is equivariant with respect to the action of the simplicial group $\exp(L^0)_n= \exp( (L^{\Delta^n})^0_0 \le \Gg(L^{\Delta^n})$, so gives the required isomorphism on taking homotopy quotients.
\end{proof}

\appendix
\section{Quasi-monads, quasi-comonads and quasi-distributivity}\label{qmonadsn}

To date, the main context in which s.h. $\top$-algebras have been studied is when the monad $\top$ comes from an operad. This is partly because these were the only case for which a satisfactory theory of morphisms was developed in \cite{loop}.  Since  this difficulty was resolved for general monads (and even distributive monad-comonad pairs) in \S \ref{snerves}, we now look into how related constructions adapt to this generality. 

In \cite{KS}, free resolutions of operads were exploited to study deformations of algebras. Since an objectwise weak equivalence $\top' \to \top$ of arbitrary  monads on a simplicial category gives a weak equivalence of the associated  simplicial quasi-descent data (from Proposition \ref{enrichtop}), the respective Segal spaces of s.h. $\top$-algebras are weakly equivalent. This means that many of the ideas from \cite{KS} carry over to arbitrary monads.  

However, in some settings, such as \cite{laan}, operads are too restrictive, and homotopy operads have to be used instead. 
We now introduce quasi-monads and quasi-comonads, which give a context sufficiently general to be analogous to homotopy operads, while  
providing a natural generalisation of the constructions of \S \ref{algsn}. 

\subsection{Quasi-monads and quasi-comonads}

\begin{definition} 
Define  a quasi-monad (resp. a quasi-comonad) on a category $\cB$ to be a quasi-monoid (resp. quasi-comonoid), as in Definition \ref{qmdef}, in the monoidal category $(\End(\cB), \circ)$ of endofunctors of $\cB$. 
\end{definition}

Substituting the monoidal category $(\End(\cB), \circ)$ into Lemma \ref{qmlemma} yields the following two lemmas.

\begin{lemma}\label{qmonad}
A quasi-monad consists of functors $\top_n:\cB \to \cB$, together with the following data:
$$
\begin{matrix}
\mu_i:\top_{n+1} \abuts \top_{n} & 1\le i \le n\\
\eta_i:\top_{n-1}\abuts \top_{n} &0 \le i <n,
\end{matrix}
$$
an associative coproduct $\xi_{mn}: \top_{m+n}\abuts \top_m \circ \top_n$, with coidentity $\xi_0:  \top_0\abuts \id$, satisfying:
\begin{enumerate}
\item $\mu_i\mu_j=\mu_{j-1}\mu_i\quad i<j$.
\item $\eta_i\eta_j=\eta_{j+1}\eta_i \quad i \le j$.
\item 
$
\mu_i\eta_j=\left\{\begin{matrix}
			\eta_{j-1}\mu_i & i<j \\
			\id		& i=j,\,i=j+1 \\
			\eta_j\mu_{i-1} & i >j+1
			\end{matrix} \right. .
$

\item $(\mu_i \top_n)  \xi_{m+1,n} =\xi_{mn} \mu_i$.
\item $(\top_m \mu_i) \xi_{m,n+1}=\xi_{mn}\mu_{i+m} $.
\item $(\eta_i\top_n)\xi_{m-1,n}=\xi_{mn}\eta_i$.
\item $(\top_m \eta_i)\xi_{m,n-1}=\xi_{mn}\eta_{i+m}$.
\end{enumerate}
\end{lemma}

\begin{lemma}\label{qcomonad}
A quasi-comonad consists of functors $\bot^n:\cB \to \cB$, together with the following data:
$$
\begin{matrix}
\Delta^i:\bot^n \abuts \bot^{n+1} & 1\le i \le n\\
\vareps^i:\bot^{n}\abuts \bot^{n-1} &0 \le i <n,
\end{matrix}
$$
an associative product $\zeta^{mn}:\bot^m \circ \bot^n \abuts \bot^{m+n}$, with identity $\zeta^0: \id \abuts \bot^0$, satisfying:
\begin{enumerate}
\item $\Delta^j\Delta^i=\Delta^i\Delta^{j-1}\quad i<j$.
\item $\vareps^j\vareps^i=\vareps^i\vareps^{j+1} \quad i \le j$.
\item 
$
\vareps^j\Delta^i=\left\{\begin{matrix}
			\Delta^i\vareps^{j-1} & i<j \\
			\id		& i=j,\,i=j+1 \\
			\Delta^{i-1}\vareps^j & i >j+1
			\end{matrix} \right. .
$
\item $\zeta^{m+1,n}(\Delta^i\bot^n)=\Delta^i\zeta^{mn}$.
\item $\zeta^{m,n+1}(\bot^m\Delta^i)=\Delta^{i+m}\zeta^{mn}$.
\item $\zeta^{m-1,n}(\vareps^i\bot^n)=\vareps^i\zeta^{mn}$.
\item $\zeta^{m,n-1}(\bot^m\vareps^i)=\vareps^{i+m}\zeta^{mn}$.
\end{enumerate}
\end{lemma}

\begin{lemma}
A quasi-monad  on $\cB$   gives rise to a quasi-descent datum  enriching $\cB$, given by 
$$
\cHom(B,B')^n:= \Hom_{\cB}(\top_n B, B').
$$
\end{lemma}
In particular, his allows us to define  Maurer-Cartan sets.  We also have a notion of homotopy monad on a simplicial category (when the $\xi$ are all weak equivalences).  Dually, we have the same constructions for quasi-comonads.

\subsection{Distributivity for quasi-monads and quasi-comonads}

We now need to describe a distributivity transformation $\lambda$ for a quasi-monad $\top$ and a quasi-comonad $\bot$.
We wish to enrich $\cB$ to a quasi-descent datum  by  setting
$$
\cHom_{\cB}(B,B')^n:= \Hom(\top_n B, \bot^nB'),
$$
so we need natural transformations
$$
\lambda_{m}^n: \top_m\bot^n \abuts \bot^n\top_m. 
$$
with the following diagrams commuting:
$$
\xymatrix{
\top_m\bot^n \ar@{=>}[r]^{\lambda_{m}^n} \ar@{=>}[d]_{\top_m\Delta^i}&   \bot^n \top_m\ar@{=>}[d]_{\Delta^i\top_m} \\
\top_m\bot^{n+1} \ar@{=>}[r]^{\lambda_{m}^{n+1}} & \bot^{n+1}\top_m
}
\quad
\xymatrix{
 \top_m\bot^n \ar@{=>}[r]^{\lambda_{m}^n} &  \bot^n \top_m  \\
\top_{m+1}\bot^n\ar@{=>}[u]_{\mu_i\bot^n} \ar@{=>}[r]^{\lambda_{m+1}^{n}} & \bot^n\top_{m+1}\ar@{=>}[u]_{\bot^n\mu_i}
}
$$

$$
\xymatrix{
\top_m\bot^n \ar@{=>}[r]^{\lambda_{m}^n} \ar@{=>}[d]_{\top_m\vareps^i}&   \bot^n \top_m\ar@{=>}[d]_{\vareps^i\top_m} \\
\top_m\bot^{n-1} \ar@{=>}[r]^{\lambda_{m}^{n-1}} & \bot^{n-1}\top_m
}
\quad
\xymatrix{
 \top_m\bot^n \ar@{=>}[r]^{\lambda_{m}^{n}} &  \bot^n \top_m  \\
\top_{m-1}\bot^n\ar@{=>}[u]_{\eta_i\bot^n} \ar@{=>}[r]^{\lambda_{m-1}^{n}} & \bot^n\top_{m-1}\ar@{=>}[u]_{\bot^n\eta_i},
}
$$
together with unit rules
$$
(\xi_0\bot^n) = (\bot^n\xi_0) \circ \lambda^n_0 \quad \zeta^0\top_m= \lambda^0_m \circ (\top_m\zeta^0)
$$
and
associativity rules
$$
(\lambda_{p}^n\top_q) \circ (\top_p \lambda_{q}^n)\circ (\xi_{pq}\bot^n)= (\bot^n\xi_{pq})\circ \lambda_{p+q}^n,
$$
$$
(\bot^p\lambda_{m}^q) \circ (\lambda_{m}^p\bot^q)\circ (\top_m\zeta^{pq})= (\zeta^{pq}\top_m)\circ \lambda_{m}^{p+q}.
$$

\begin{lemma}
Given a category $\cB$, a quasi-monad $\top$ and a quasi-comonad $\bot$ on $\cB$, together with a quasi-distributivity transformation $\lambda$ satisfying the conditions above, there is a quasi-descent datum  on $\cB$ given by  setting
$$
\cHom_{\cB}(B,B')^n:= \Hom(\top_n B, \bot^nB').
$$
\end{lemma}
\begin{proof}
We define the structures on $\cHom_{\cB}$ by 
\begin{eqnarray*}
\pd^ix&=& \Delta^i \circ x\circ \mu_i,\\
 \sigma^i x &=& \vareps^i \circ x \circ \eta_i,\\ 
x*y&=& \zeta^{nm}\circ (\bot^nx) \circ \lambda_{m}^n\circ (\top_my)\circ \xi_{mn},
\end{eqnarray*}
for $x\in \cHom_{\cB}(B',B")^m $ and $y \in \cHom_{\cB}(B,B')^n$,  with identity $\zeta^0 \circ \xi_0$.

It follows immediately from Lemmas \ref{qmonad} and \ref{qcomonad} that $\cHom_{\cB}(B,B') \in \Set^{\Delta_{**}}$. The first four diagrams above ensure that $*$ defines a map
$$
\cHom_{\cB}(B',B")\ten \cHom_{\cB}(B,B') \to \cHom_{\cB}(B,B")
$$ 
in $\Set^{\Delta_{**}} $. The unit rules then ensure that $\zeta^0 \circ \xi_0 $ is the multiplicative identity,  and  the associativity rules make  this product associative.
\end{proof}

\section{$A_{\infty}$-algebras and homotopy operads}\label{shaa}

$A_{\infty}$-algebras are designed to model deformation retracts of differential graded associative  algebras without unit (DGAAs). They are thus an alternative candidate for the task performed in general by s.h. algebras, and are indeed often known as strong homotopy associative algebras (SHAAs). That both concepts are essentially equivalent seems to be folklore (though, if necessary, it could be inferred from the results of 
\S \ref{liesn}).

Now,  a DGAA is just a semigroup in the category of cochain complexes (i.e. it satisfies all the requirements of a monoid, except the unit axiom). A third candidate to model deformation retracts of DGAAs is therefore a homotopy semigroup in cochain complexes (defined analogously to a homotopy monoid). These were studied in \cite{leinsteroperads}, where it was conjectured that they give rise to $A_{\infty}$-algebras.

\subsection{Homotopy semigroups and semicogroups}

\begin{definition}
Define $\Delta_{**}^+$ to be the subcategory of $\Delta_{**}$ on objects $\mathbf{n} \ne \mathbf{0}$ and   containing only  injective morphisms. 
\end{definition}

\begin{definition}
Define a semigroupal  category to be a category $\C$ equipped with a bifunctor $\C\by \C \xra{\ten}\C$ satisfying the axioms of a monoidal category (but without a unit object).
\end{definition}

\begin{definition}\label{semigpdef}
Given a semigroupal  category $\C$ (in particular if $\C$ is monoidal), define a quasi-semigroup $X$ in $\C$ to be a colax semigroupal functor $X: (\Delta_{**}^+)^{\op} \to \C$. This means that we have maps
$$
\xi_{mn}: X_{m+n} \to X_m\ten X_n,
$$
 satisfying naturality and coherence. If $\C$ is a model category, we say that $X$ is a homotopy semigroup whenever the maps $\xi_{mn}$ are all weak equivalences.
\end{definition}

\begin{lemma}
Giving  a quasi-semigroup $X$ in $\C$ is equivalent to giving objects $X_n \in \C$ for $n \in \N_1$, together with morphisms
$$
\begin{matrix}
\pd_i:X_{n+1} \to X_{n} & 1\le i \le n,\\
\end{matrix}
$$
and an associative coproduct $\xi_{mn}: X_{m+n}\to X_m \ten X_n$, satisfying:
\begin{enumerate}
\item $\pd_i\pd_j=\pd_{j-1}\pd_i\quad i<j$.
\item $(\pd_i \ten \id)  \xi_{m+1,n} =\xi_{mn} \pd_i$.
\item $(\id\ten \pd_i) \xi_{m,n+1}=\xi_{mn}\pd_{i+m} $.
\end{enumerate}
\end{lemma}

\begin{definition}
Given a semigroupal category $\C$, define a quasi-semicogroup $X$ in $\C$ to be a lax semigroupal functor $X: \Delta_{**}^+ \to \C$. This means that we have maps
$$
\zeta^{mn}:  X^m\ten X^n \to X^{m+n},
$$
 satisfying naturality and coherence. If $\C$ is a model category, we say that $X$ is a homotopy semicogroup whenever the maps $\zeta^{mn}$ are all weak equivalences.
\end{definition}

\begin{lemma}
Giving  a quasi-semicogroup $X$ in $\C$ is equivalent to giving objects $X^n \in \C$ for $n \in \N_1$, together with morphisms
$$
\begin{matrix}
\pd^i:X^n \to X^{n+1} & 1\le i \le n\\
\end{matrix}
$$
an associative product $\zeta^{mn}:X^m \ten X^n \to X^{m+n}$, satisfying:
\begin{enumerate}
\item $\pd^j\pd^i=\pd^i\pd^{j-1}\quad i<j$.
\item $\zeta^{m+1,n}(\pd^i\ten \id)=\pd^i\zeta^{mn}$.
\item $\zeta^{m,n+1}(\id\ten \pd^i)=\pd^{i+m}\zeta^{mn}$.
\end{enumerate}
\end{lemma}

\subsection{Homotopy semigroups in abelian categories}.

\begin{definition}
Given a quasi-semigroup $V$ in an abelian semigroupal category $\cV$, define the chain complex $ C(V) \in \Ch(\cV)$ by
$$
C(V)_n:= \left\{ \begin{matrix} V_n & n > 0 \\ 0 & n \le 0, \end{matrix} \right.
$$
with differential $d= \sum_i (-1)^i\pd_i$.
\end{definition}

\begin{definition}
A coalgebra $C$ is said to be conilpotent if the iterated coproduct $\Delta^n: C \to C^{\ten n}$ is $0$ for $n$ sufficiently large. A coalgebra $C$ is said to be ind-conilpotent if it is a filtered colimit of  conilpotent coalgebras.
\end{definition}

\begin{lemma}\label{coalg}
$C(V)$ has the natural structure of an ind-conilpotent coassociative coalgebra without counit in $\Ch(\cV)$. 
\end{lemma}
\begin{proof}
We define the coproduct $\Delta: C(V) \to C(V) \ten C(V)$ by $\bigoplus_{i+j=n} \xi_{ij}: C(V)_n \to (C(V)\ten C(V))_n$. It is straightforward to verify that this is coassociative and a chain map, so $C(V)$ is a coassociative coalgebra without counit. 

 Observe that the brutal truncations $\sigma_{\le m}C(V)$ of $C(V)$ form conilpotent subcoalgebras of $C(V)$, since $(C(V)^{\ten n})_i=0$ for all $i <n$.   Thus $C(V)$ is ind-conilpotent, since $C(V)= \varinjlim \sigma_{\le m}C(V)$.
\end{proof}

\subsubsection{DG coalgebras}

Now assume that $\cV$ is the category of cochain complexes of  vector spaces over a field $k$. If $V$ is a quasi-semigroup in $\cV$, then Lemma \ref{coalg} implies that the cochain complex
$
\Tot C(V)
$ 
(given by $(\Tot C(V))^n:= \bigoplus_{i} C(V)_i^{i+n}$) is an  ind-conilpotent coassociative DG coalgebra without counit over $k$. 

\begin{definition}
Let $DG\C_k$ be the category of all ind-conilpotent coassociative   DG coalgebras without counit over $k$. Let $DG\cA_k$ be the category of associative DG algebras without unit over $k$.
\end{definition}

\begin{definition}
Define the cobar functor $\beta^*: DG\C_k \to DG\cA_k$ by letting $\beta^*C$ be the free graded associative $k$-algebra $\bigoplus_{n>0} C[-1]^{\ten n}$  on generators $C[-1]$, with differential defined on generators by $d_C+ \Delta$, for $\Delta: C[-1]\to (C\ten C)[-1]= (C[-1]\ten C[-1])[1]$ the comultiplication.

It has right adjoint $\beta:  DG\cA_k\to DG\C_k$ given by letting $\beta A$ be the cofree graded coassociative  ind-conilpotent $k$-algebra $\bigoplus_{n>0} A[1]^{\ten n}$  on cogenerators $A[1]$, with differential defined on cogenerators by $d_A+ m : A\oplus  (A\ten A)[1]  \to A[1] $, for $m : A\ten A \to A$ the multiplication.
\end{definition}

\begin{definition}
Define the tangent space $\tan C$ of $C \in  DG\C_k$ to be $\ker( \Delta: C \to C\ten C)$.
\end{definition}

\begin{theorem}\label{model}
There is a cofibrantly generated model structure on $DG\C_k$, for which a morphism $f$ is 
\begin{enumerate}
\item a cofibration if it is injective;
\item a weak equivalence if either of the following equivalent conditions holds:
\begin{enumerate}
\item $\beta^* f$ is  a quasi-isomorphism.
\item $f$ can be expressed as a filtered colimit of quasi-isomorphisms $f_{\alpha}: C_{\alpha} \to D_{\alpha}$ between finite-dimensional objects of $DG\C_k$ (note that this is a stronger  than requiring that $f$ be a quasi-isomorphism);
\end{enumerate}
\item a fibration if $f$ is cofree as a morphism of ind-conilpotent coassociative graded coalgebras without counit.
\end{enumerate}
Moreover, for a fibrant object $C$  there is a natural isomorphism  $\H^n (\tan C) \cong \H^{n+1}(\beta^*C)$.

With respect to this model structure, $\beta^*$ is a left Quillen equivalence, when $DG\cA_k$ is given its standard model structure. 
\end{theorem} 
\begin{proof}
Existence of such a model structure is given in \cite{ddt1} Proposition \ref{ddt1-dgspmodel} for the analogous case of cocommutative coassociative coalgebras and Lie algebras, but the proof carries over to any Koszul-dual pair of quadratic operads, so it adapts to our context (coassociative coalgebras and associative algebras). The generating cofibrations are injective morphisms $f:C \to D$ between finite-dimensional objects, satisfying the additional property that the coproduct $\coker f \to D \ten \coker f$ is zero. Generating  trivial cofibrations have the additional property that $\H^*(\coker f )=0$.

Characterisation of the weak equivalences follows from \cite{ddt1} Proposition \ref{ddt1-tqiscor}. That $\beta^*$ is a Quillen equivalence follows from \cite{ddt1} Theorem \ref{ddt1-mcequiv}.
\end{proof}

\begin{remark}
Note that fibrant objects of $DG\C_k$ are those whose underlying  coalgebras are cofree. These are precisely  strong homotopy associative algebras (SHAAs), as in \cite{Kon}, and weak equivalences between these are tangent quasi-isomorphisms. A choice of cogenerators on an SHAA is precisely an $A_{\infty}$-algebra. This means that every $A_{\infty}$-algebra has a weakly equivalent DG associative algebra.
\end{remark}

\begin{corollary}\label{cfainfty}
There is a canonical equivalence class of $A_{\infty}$-algebras associated to any quasi-semigroup $V$ in the category of cochain complexes  over  $k$. 
\end{corollary}
\begin{proof}
The $A_{\infty}$-algebra is just a choice of cogenerators on a fibrant replacement for $\Tot C(V)$
\end{proof}

When $V$ is a homotopy semigroup, we wish to relate this $A_{\infty}$-algebra to $V$. 


\begin{proposition}\label{tancoho}
Assume that  $C \in  DG\C_k$ is equipped with an exhaustive  increasing filtration $0=F_0C \subset F_1C \subset \ldots$, comultiplicative in the sense that
$$
\Delta(F_n) \subset \sum_{i+j=n}F_i \ten F_j \subset C\ten C,
$$
and for which  the resulting maps $\Delta_n:\gr^F_nC \to (F_1C)^{\ten n}$ are quasi-isomorphisms. Then there are canonical isomorphisms
$$
\H^{n+1}(\beta^*C) \cong \H^n(F_1 C).
$$
\end{proposition}
\begin{proof}
The filtration $F_nC$ induces a filtration on $\beta^*C$ by
$$
F_p(C[-1]^{\ten n}):= \sum_{p_1+ \ldots +p_n=p} (F_{p_1}C[-1])\ten (F_{p_2}C[-1])\ten \ldots \ten (F_{p_n}C[-1]).
$$
Since this filtration is exhaustive and bounded below, the associated spectral sequence
$$
E^{pq}_0(F)=(\gr^F_{-p}\beta^*C)^{p+q} \abuts \H^{p+q}(\beta^*C)
$$
converges. 

Now, note that $\gr^F\beta^*C\cong \beta^*(\gr^FC)$, and that the quasi-isomorphisms
 $\Delta_n:\gr^F_nC \to (F_1C)^{\ten n}$ induce a graded quasi-isomorphism
$$
\delta: \gr^FC \to \bigoplus_{n>0} (F_1C)^{\ten n}
$$
in $DG\C_k$.

We now define a filtration $W$ on $\beta^*C$ by $W^m\beta^*C= \bigoplus_{n\ge m} C[-1]^{\ten n}$, noting that $F_p(\beta^*C) \cap W^{p+1}(\beta^*C)=0$, since $F_{n-1} (C^{\ten n})=0$. Thus $W$ induces a finite filtration on  $\gr^F_{-p}\beta^*C$, so the associated spectral sequence 
$$
E^{nq}_0(W) = (\gr_W^n\gr^F_{-p}\beta^*C)^{n+q} \abuts \H^{p+q}(\gr^F_{-p}\beta^*C)
$$
converges. Now, the left-hand side is just $\gr^F_{-p} (C[-1]^{\ten n})  $, so 
$
E^{nq}_1(W) \cong \H^{n+q}(\gr^F_{-p} (C[-1]^{\ten n}) ). 
$

Therefore $\delta$ induces an isomorphism of $E_1$ spectral sequences, making
$$
\gr^F_{-p}\beta^*(\delta):  \gr^F_{-p}\beta^*C \to\gr^F_{-p}\beta^*(\bigoplus_{n>0} (F_1C)^{\ten n} ) 
$$ 
 a quasi-isomorphism.
The right-hand side is just $\gr^F_{-p}\beta^*\beta (F_1C[-1])$, where $F_1C[-1]$ is regarded as an object of $DG\cA_k$ with zero multiplication. Therefore
$$
E_1^{pq}(F)= \H^{p+q}(\gr^F_{-p}\beta^*C) \cong \H^{p+q}(\gr^F_{-p}\beta^*\beta (F_1C[-1])).
$$
Now, since $\beta^* \dashv \beta$ are a pair of Quillen equivalences, the map $\beta^*\beta (F_1C[-1]) \to F_1C[-1]$ is a quasi-isomorphism, so
$$ 
E_1^{pq}(F) \cong \left\{ \begin{matrix} \H^{p+q}(F_1C[-1])  & p=-1, \\ 0 & \text{otherwise,} \end{matrix} \right. 
$$ 
and therefore the spectral sequence collapses at $E_1$, giving
$$
\H^{q-1}(\beta^*C) \cong \H^{q-1}(F_1C[-1])= \H^{q-2}(F_1C),
$$
as required.
\end{proof}

The following result confirms a conjecture from \cite{leinsteroperads} \S 3.5, although not with the proof envisaged there. 
\begin{proposition}
 If $V$ is a homotopy semigroup, then we may choose a representative  $A_{\infty}$-algebra with underlying cochain complex $V_1$. 
\end{proposition}
\begin{proof}
We adapt the proof of \cite{htpy} Lemma \ref{htpy-sminimal}.
Take a trivial cofibration $\Tot C(V) \to E$, with $E$ fibrant. The filtration $F_nC:= \bigoplus_{i \le n} V_n$ of $C(V)$ satisfies the conditions of Proposition \ref{tancoho}, since the maps $V_n \to (V_1)^{\ten n}$ are quasi-isomorphisms by hypothesis. Thus the map $V_1 \to \tan E$ is a quasi-isomorphism. Let the quotient be $Q$, and note that this is a contractible cochain complex. 

Let $B$ be the cofree coalgebra on generators $Q$; this is trivially fibrant. Since the morphism $\tan E \to E$ is a cofibration in $DG\C_k$, the composite map $\tan E \to Q \to B$ must extend to a morphism $f:E \to B$ in $DG\C_k$. Let $A$ be the coequaliser of $f$ and the zero map; this is again a cofree object in $DG\C_k$, hence fibrant, and $\tan A= \ker (\tan E \to Q) \cong V_1$.

Therefore $A$ is an SHAA with cogenerating space $V_1$, so defines an $A_{\infty}$-structure on $V_1$. 
\end{proof}

\subsection{DG co-operads}

We now show the relation  between $A_{\infty}$-algebras and homotopy semigroups of cochain complexes has an analogue for operads. Roughly speaking, we will show that a homotopy monad (in the sense of \S \ref{qmonadsn}) with suitable operadic structure is related to the homotopy operads of \cite{laan}.  

\begin{definition}
Given an additive  cocomplete monoidal category $\C$, we now define a full subcategory $\cE(\C)$ of the category $\End(\C)$ of endofunctors on $\C$.  Objects of $\cE$ correspond to  collections $\{P_n\}_{n \ge 0}$, with $P_n$ a $\C$-representation of the symmetric group $S_n$, with the associated endofunctor given by
$$
V \mapsto \bigoplus_n P_n\ten^{S_n}V^{\ten n}.
$$

 $\cE$ forms a monoidal category under composition of functors, and an additive category under $\oplus$.
\end{definition}

\begin{definition}
 An operad (resp. pseudo operad) on $\C$ is a monoid (resp. semigroup) in $\cE(\C)$, and a  co-operad (resp. pseudo co-operad) on $\C$ is a comonoid (resp. semicogroup) in $\cE(\C)$. 
\end{definition}

Note that since $\cE(\C)$ is an additive category, there is a natural retraction $(F\circ X) \oplus (F\circ Y) \to F\circ (X\oplus Y) \to (F\circ X) \oplus (F\circ Y)$, for any $F,X,Y \in \cE(\C)$.  Thus, augmented operads $\top$ on abelian categories $\C$ correspond to pseudo operads $S$ on $\C$, by setting $S:= \ker(\top \to 1)$ and $\top = S\oplus 1$. 

Let $dg\Vect$ be the category of chain complexes, and $g\Vect$ the category of graded vector spaces, both over a field of characteristic $0$.

\begin{lemma}
If $\C= dg\Vect$ or $g\Vect$, the forgetful functor from pseudo operads on $\C$ to $\cE(\C)$ has a left adjoint, denoted by $T$. Likewise, the forgetful functor from ind-conilpotent pseudo co-operads on $\C$ to $\cE(\C)$ has a right adjoint, denoted by $T'$. 
\end{lemma}
\begin{proof}
These are described in \cite{laan} \S 2. In both cases, the underlying functor is $F \mapsto \bigoplus_{n>0} \overbrace{F\circ F \circ \ldots \circ F}^n$.  
\end{proof}

\begin{definition}
Recall from \cite{laan} Definition 3.1 that an operad up to homotopy on  $dg\Vect$ is defined to be a collection $P \in \cE(g\Vect)$, together with a square-zero differential $\delta$ on the cofree pseudo co-operad
$$
T'(P[1]).
$$
\end{definition}

The notion of operad up to homotopy in \cite{laan} generalises pseudo operads (or, equivalently, augmented operads) rather than operads. This motivates the following comparison, noting that the quasi-monads of Appendix \ref{qmonadsn} are \emph{a fortiori} quasi-semigroups in $\End(\C)$, and that pseudo co-operads can be replaced by homotopy operads, similarly to Corollary \ref{cfainfty}. 
\begin{lemma}
Every quasi-semigroup $Q$ in $\cE(dg\Vect)$ (in the sense of Definition \ref{semigpdef}) naturally gives rise to a pseudo co-operad $\beta(Q)$ on $dg\Vect$.
\end{lemma}
\begin{proof}
Given a quasi-semigroup $Q= \{Q_n\}_{n >0}$ in $\cE(dg\Vect)$ , we may set $\beta(Q):= \bigoplus_{n>0} Q_n[n] $ in $\cE(g\Vect)$, with differential $\delta: d_Q \pm \sum_i (-1)^i\pd_i$. Here, $\pd_i:Q_n[n]_j \to Q_{n+1}[n+1]_{j-1}$  is the structural map $\pd_i: (Q_n)_{n+j} \to (Q_{n+1})_{n+j}$ of the quasi-semigroup. 

The functor $\beta(Q)$ has the natural structure of a pseudo co-operad, with coproduct
$
\beta(Q) \to \beta(Q) \circ \beta(Q)
$
given on $Q_n[n] \subset \beta(Q)$ by $\sum_{i+j=n} \xi_{ij}: Q_n[n] \to Q_i[i] \circ Q_j[j]$, making use of the natural retraction 
$(F\circ X) \oplus (F\circ Y) \to F\circ (X\oplus Y) \to (F\circ X) \oplus (F\circ Y)$.
\end{proof}

\bibliographystyle{alphanum}
\addcontentsline{toc}{section}{Bibliography}
\bibliography{references}
\end{document}